\documentclass{amsart}
\usepackage[utf8]{inputenc}
\usepackage{amssymb,latexsym} 
\usepackage{amsmath}
\usepackage{amsthm}
\usepackage{mathrsfs}
\usepackage{mathtools}
\usepackage{enumitem}
\usepackage{graphicx}
\usepackage{stmaryrd}
\usepackage{tikz}
\usepackage{tikz-cd}
\usetikzlibrary{decorations.markings, patterns, calc}
\usepackage[outline]{contour}\contourlength{1.5pt}
\usepackage{wasysym}
\usepackage{subcaption}
\usepackage{setspace}
\usepackage{hyperref}
\usepackage{doi}
\usepackage{bm}
\usepackage{cite}

\definecolor{bgreen}{rgb}{0,1,0.5}
\definecolor{amber}{rgb}{1.0, 0.75, 0.0}

\numberwithin{equation}{section}

\makeatletter
\def\paragraph{\@startsection{paragraph}{4}	\z@\z@{-\fontdimen2\font} {\normalfont\itshape\bfseries}}
\makeatother

\setlist[enumerate]{leftmargin=\parindent, itemsep=4pt, topsep=4pt, parsep=0pt, labelsep=0.2em}
\setlist[itemize]{leftmargin=\parindent, itemsep=4pt, topsep=4pt, parsep=0pt}

\makeatletter
\g@addto@macro \normalsize{%
	\setlength\abovedisplayskip{8pt plus 2pt minus 2pt}%
	\setlength\belowdisplayskip{8pt plus 2pt minus 2pt}}%
\makeatother

\def\mbf #1{\mathbf{#1}}

\def\Cal #1{\mathcal{#1}}
\def\mfr #1{\mathfrak{#1}}

\makeatletter
\newcommand*\bigcdot{\mathpalette\bigcdot@{.5}}
\newcommand*\bigcdot@[2]{\mathbin{\vcenter{\hbox{\scalebox{#2}{$\m@th#1\bullet$}}}}}
\makeatother

\newcommand{\mfrac}[2]{\tfrac{#1\rule[-0.3em]{0pt}{1em}}{#2\rule[-0.3em]{0pt}{1em}}}
\renewcommand{\i}{\mkern.5mu\mathrm{i}\mkern1mu}
\renewcommand{\d}{\mkern.5mu\mathrm{d}\mkern0.5mu}
\newcommand{\e}{\mkern1mu\ensuremath{\mathrm{e}}}

\def\sX{\mathsf{X}}
\def\sP{\mathsf{P}}
\def\sT{\mathsf{T}}
\def\sZ{\mathsf{Z}}
\def\sW{\mathsf{W}}
\def\sD{\mathsf{D}}
\def\sC{\mathsf{C}}

\def\sE{\mathsf{E}}

\newcommand{\C}{\mathbb C}

\newcommand{\Z}{\mathbb Z}
\newcommand{\Q}{\mathbb Q}
\newcommand{\R}{\mathbb R}
\newcommand{\K}{\mathbb K}
\renewcommand{\H}{\mathbb H}
\newcommand{\CP}{\mathbb{CP}}

\renewcommand{\sl}{\mathfrak{sl}}

\newcommand{\GL}{\mbf{GL}}
\newcommand{\SL}{\mbf{SL}}
\newcommand{\bT}{\mbf{T}}
\newcommand{\bB}{\mbf{B}}
\newcommand{\bN}{\mbf{N}}
\newcommand{\bU}{\mbf{U}}
\newcommand\rotT{\rotatebox[origin=c]{180}{$\mbf{T}$}}

\newcommand{\Nabla}{\nabla}

\newcommand{\tr}{\operatorname{tr}}
\newcommand{\id}{\operatorname{id}}
\newcommand{\res}{\operatorname{Res}}
\newcommand{\ord}{\operatorname{ord}}
\newcommand{\diag}{\operatorname{Diag}}

\newcommand{\Rep}{\operatorname{Rep}}
\newcommand{\proj}{\operatorname{proj}}
\newcommand{\Sing}{\mathsf{Sing}}
\newcommand{\Equilib}{\mathsf{Equilib}}
\newcommand{\Crit}{\mathsf{Crit}}
\newcommand{\Saddle}{\mathsf{Saddle}}
\newcommand{\hot}{\operatorname{h.o.t.}}
\newcommand{\cont}{\operatorname{cont}}
\newcommand{\const}{\operatorname{const}}
\newcommand{\Disc}{\operatorname{Disc}}
\newcommand{\Diff}{\operatorname{Diff}}

\newcommand{\Aut}{\operatorname{Aut}}
\newcommand{\holom}{\operatorname{hol}}
\newcommand{\Mon}{\mfr{M}}
\newcommand{\Char}{\mathcal{M}}
\newcommand{\jet}{\mathrm{j}}
\newcommand{\nf}{\mathrm{nf}}
\newcommand{\out}{\mathrm{out}}
\newcommand{\conf}{\mathrm{con}}
\newcommand{\cluster}{\mathrm{clu}}
\newcommand{\dd}[1]{{\mfrac{\partial}{\partial #1}}}
\newcommand{\tdd}[1]{{\tfrac{\d}{\d #1}}}
\newcommand{\Qnf}{Q_{\mathrm{nf}}}
\newcommand{\bt}{\mathbf{t}}

\newcommand{\filledstar}{\star}

\usepackage{cleveref}
\theoremstyle{plain}
\newtheorem{lemma}{Lemma}[section]
\newtheorem{proposition}[lemma]{Proposition}

\newtheorem{theorem}[lemma]{Theorem}
\newtheorem{corollary}[lemma]{Corollary}

\theoremstyle{definition}
\newtheorem{definition}[lemma]{Definition}
\newtheorem{example}[lemma]{Example}
\newtheorem{remark}[lemma]{Remark}

\theoremstyle{remark}
\newtheorem{notation}[lemma]{Notation}

\author{Martin Klime\v{s}}
\title[Deformations of singularities of $\sl_2(\C)$-connections]{Deformations of singularities of meromorphic $\sl_2(\C)$-connections and meromorphic quadratic differentials}

\begin{document}


\date{\today}

\begin{abstract} 
This paper contributes to the theory of singularities of meromorphic linear ODEs in traceless $2\times 2$ cases, focusing on their deformations and confluences. It is divided into two parts:
 
The first part addresses individual singularities without imposing restrictions on their type or degeneracy. The main result establishes a correspondence between local formal invariants and jets of meromorphic quadratic differentials. 
This result is then utilized to describe the parameter space of universal isomonodromic deformation of meromorphic $\sl_2(\C)$-connections over Riemann surfaces.
 
The second part examines the confluence of singularities in a fully general setting, accommodating all forms of degeneracies. 
It explores the relationship between the geometry of the unfolded Stokes phenomenon and the horizontal and vertical foliations of parametric families of quadratic differentials. The local moduli space is naturally identified with a specific space of local monodromy and Stokes data, presented as a space of representations of certain fundamental groupoids associated with the foliations. 
This is then used for studying degenerations of isomonodromic deformations in parametric families.

%
%
\end{abstract}

\maketitle

\section{Introduction}

Let $\Nabla(x)$ be a meromorphic connection on a complex vector bundle $E\to \sX$ over a Riemann surface $\sX$.
In local coordinates $(x,y)$ on $E$, at a singular point the connection takes the form
\begin{equation*}
\Nabla(x)=\d_x-A(x)\tfrac{\d x}{x^{k+1}},\qquad A(0)\neq 0.	
\end{equation*}
The local theory of singularities is now well established, drawing on works of
Birkhoff (1913), Hukuhara (1937), Turittin (1955), Levelt (1961), Sibuya (1967), Malgrange (1971), Balser, Jurkat, Lutz (1979), Ramis (1985), 
and many other authors (see e.g.\cite{Babbitt-Varadarajan, Balser, Balser-Jurkat-Lutz1, Balser-Jurkat-Lutz12, Ilashenko-Yakovenko, Loday-Richaud, Martinet-Ramis, Sibuya} for general presentation). 
Their analytic classification typically proceeds in two steps:
I - formal classification, II - description of a Stokes phenomenon at irregular singularities, which express the additional analytic obstructions to convergence of formal transformations. 
Even though a fully general theory exists, it is common to focus on non-resonant cases since it significantly simplifies the presentation in both steps.
However, in the $\sl_2(\C)$ situation, when $A(x)$ is a traceless $2\!\times\!2$ matrix, everything becomes significantly simpler. 
This allows for the treatment of degenerate cases in a unified manner alongside generic ones, without requiring non-resonance assumptions. Moreover, the analytic theory extends naturally to parametric deformations. 
This is the basic premise of the article.

\bigskip
In the first part, we review the formal theory of individual singularities of meromorphic $\sl_2(\C)$-connections.
These can be locally gauge-transformed into the form
\[\nabla(x)=\d_x -\begin{psmallmatrix}0&1\\[4pt]Q(x)&0\end{psmallmatrix}\tfrac{\d x}{x^{k+1}}\]
corresponding to a second-order liner ODE
\[\left(x^{k+1}\tdd{x}\right)^2y_1-Q(x)y_1=0.\]
An advantage of this form is that now the formal invariant can be identified  with a certain jet of the associated meromorphic quadratic differential (Theorems~\ref{theorem:formalgaugeequivalence} and~\ref{theorem:formalgaugecoordinateequivalence})
\[Q(x)\left(\tfrac{\d x}{x^{k+1}}\right)^2.\]

This is used in \S\,\ref{sec:isomonodromy} in description of the parameter space of universal isomonodromic deformation under a generic condition on the singular divisor.
An isomonodromic family  $\nabla_t(x)=\d_x -A(x,t)\frac{\d x}{x^{k+1}}$ is one that comes from a single flat (integrable) connection of the form
		\[\nabla(x,t)=\d_{x,t} -A(x,t)\tfrac{\d x}{x^{k+1}}-\sum_i B_i(x,t)\tfrac{\d t_i}{x^{k}}.\]
Based on the paper \cite{Heu} of V.~Heu, each such family is locally analytic gauge--coordinate equivalent to a trivial one of the form 
\[\tilde\nabla(x,t)=\d_{x,t} -\begin{psmallmatrix}0&1\\[4pt]Q(x)&0\end{psmallmatrix}\tfrac{\d x}{x^{k+1}}.\]
Consequently, the deformation parameter space is identified with a certain jet space of meromorphic quadratic differentials with fixed residue (Theorem~\ref{theorem:isomonodromic}).

The second part of the paper focuses on unfolding deformations of irregular singularities:
\[\nabla_\epsilon(x)=\d_x -A(x,\epsilon)\tfrac{\d x}{P(x,\epsilon)}, \qquad P(x,0)=x^{k+1}.\]
The content in this section is largely independent of the first part.

There are two particular phenomena that can be witnessed in such parametric families:
\begin{itemize}[leftmargin=2\parindent]
	\item[(i)] \emph{Confluence of singularities}: The limit divisor $\{x^{k+1}=0\}$ at $\epsilon=0$ splits to several branches $\{P(x,\epsilon)=0\}$ for $\epsilon\neq 0$.
	\item[(ii))] \emph{Confluence of eigenvalues}: While the divisor remains unchanged, $P(x,\epsilon)=x^{k+1}$, the vanishing order of $\det A(x,\epsilon)$ at the divisor jumps at the limit,
	i.e. the order up to which the two eigenvalues of $A(x,\epsilon)$ agree changes. In particular a singularity that is non-resonant  for $\epsilon\neq0$ can become resonant at $\epsilon=0$.
\end{itemize}	
Both phenomena can occur simultaneously.
We built upon the works of C.~Rousseau \cite{Hurtubise-Lambert-Rousseau, Hurtubise-Rousseau, LR1, Lambert-Rousseau, Rousseau, RT} on confluence, in particular on the paper \cite{Hurtubise-Lambert-Rousseau} with J.~Hurtubise and C.~Lambert,
as well as on our research \cite{Klimes-Rousseau, Klimes-Rousseau2, Klimes-Rousseau3} and \cite{Klimes1, Klimes3, Klimes6, KlimesWildMonodromy, Klimes-Stolovitch}.

The basic approach is to bring the family to the form 
\[\nabla_\epsilon(x)=\d_{x} -\begin{psmallmatrix}0&1\\[4pt]Q(x,\epsilon)&0\end{psmallmatrix}\tfrac{\d x}{P(x,\epsilon)},\]
and to connect its Stokes phenomen to the geometry of horizontal foliations in the rotating family of quadratic differentials
\[\e^{-2\i \vartheta}Q(x,\epsilon)\left(\tfrac{\d x}{P(x,\epsilon)}\right)^2,\qquad \vartheta\in \ ]\!-\tfrac{\pi}{2},\tfrac{\pi}{2}[\, .\]
In \S\,\ref{sec:foliation} we review basic properties of such foliations. These are used in \S\,\ref{sec:flagsplitting} in the exposition of the confluent Stokes phenomena of the family.

For individual singularities there are two basic types of approaches to the Stokes phenomenon, as noted by P.~Boalch \cite{Boalch21}:
\begin{itemize}[leftmargin=2\parindent]
	\item[(A)] construction of normalizing transformations and of canonical fundamental solutions on large sectors containing the Stokes rays (a.k.a. oscillation or separation directions),
	\item[(B)] construction of filtrations / flag structures on the solution space on small sectors around the anti-Stokes rays (a.k.a. singular or steepest descent directions).
\end{itemize}	
The relation between them is simple: a transverse pair of flags (B) determines a splitting of the solution space (A),
and vice-versa a splitting gives rise to a pair of filtrations on solutions by their growth rate.
These two approaches will be used also in the confluent setting but the form of the domains on which they live is considerably different.
We shall work with both approaches at the same time.

Following the works of J.-P.~Ramis \cite{Ramis, Martinet-Ramis, Deligne-Malgrange-Ramis}, in particular the recent papers with E.~Paul \cite{Paul-Ramis,Paul-Ramis1},
we interpret in \S\,\ref{sec:wildmonodromy} the Stokes phenomenon and the Stokes connection matrices as wild monodromy representations of certain fundamental groupoids. In our situation they are determined by the quadratic differential.
The usual wild (irregular) Riemann--Hilbert correspondence, which identifies the analytic moduli space of connections (with fixed formal data) with a moduli space of wild monodromy representations (with fixed formal data): the wild character variety \cite{Babbitt-Varadarajan, Boalch14, Put-Saito},
can be also adapted to the confluent setting.
We prove (Theorem~\ref{theorem:unfoldedclassification}) that the confluent wild monodromy representation determines the analytic invariants of the family, i.e. invariants with respect to transformation analytic in $(x,\epsilon)$.

In \S\,\ref{sec:cluster}, we show how a cluster of wild monodromies at individual singularities needs to be modified to a single confluent wild monodromy whose outer part converges to the wild monodromy of the limit singularity.  
This change of presentation leads to a different, but birationally equivalent, wild character variety.

Finally, in \S\,\ref{sec:confluentisomonodromy} we explore confluent families of isomonodromic deformations
		\[\nabla_{\epsilon,t}(x)=\d_{x} -A(x,t,\epsilon)\tfrac{\d x}{P(x,t,\epsilon)},\]
deriving from parametric families of flat connections
		\[\nabla_\epsilon(x,t)=\d_{x,t} -A(x,t,\epsilon)\tfrac{\d x}{P(x,t,\epsilon)}-\sum_iB_i(x,t,\epsilon)\tfrac{\d t_i}{P(x,t,\epsilon)}.\]
We show that, under generic conditions on the singular divisor, these families are locally conjugated to trivial ones (Theorem~\ref{theorem:confluentisomonodromy})
\[\tilde\nabla_\epsilon(x,t)=\d_{x,t} -\begin{psmallmatrix}0&1\\[4pt]Q(x,\epsilon)&0\end{psmallmatrix}\tfrac{\d x}{\tilde P(x,\epsilon)}.\]
This in particular implies that the confluent wild monodromy stays constant in the deformation parameter $t$, generalizing the well-known result about constancy of Stokes matrices along isomonodromy \cite{Jimbo-Miwa-Ueno, Malgrange-Iso}.  

The main motivation for this paper stems from Painlev\'e equations and their associated isomonodromy problems, which both degenerate following the diagram \cite{Ohyama-Okumura}:
\newlength{\arrowlength}\settowidth{\arrowlength}{$\searrow$}
\[\begin{matrix}
	& & & & P_{III}^{D_6} &\hskip-9pt\to\hskip-9pt & P_{III}^{D_7} & \hskip-9pt\to\hskip-9pt & P_{III}^{D_8}\\[-2pt]
	& & &\hskip-9pt\nearrow\hskip-9pt & &\hskip-9pt\nearrow\hskip-\arrowlength\searrow\hskip-9pt & &\hskip-9pt\searrow\hskip-9pt &   \\[-2pt]
	P_{VI}& \hskip-9pt\to\hskip-9pt & P_V & \hskip-6pt\to\hskip-6pt & P_V^{deg} & & P_{II}^{JM} & \hskip-6pt\to\hskip-6pt & \hskip-6pt P_I\\[-2pt]
	& & &\hskip-9pt\searrow\hskip-9pt & &\hskip-9pt\searrow\hskip-\arrowlength\nearrow\hskip-9pt & &\hskip-8pt\nearrow\hskip-9pt&\\[-2pt]
	& & & & P_{IV}& \hskip-9pt\to\hskip-9pt & P_{II}^{FN} & &
\end{matrix}
\]
This is schematically pictured in Figure~\ref{figure:Painleve}, due to L.~Chekhov, M.~Mazzocco and V.~Rubtsov \cite{Chekhov-Mazzocco-Rubtsov},
	\begin{figure}
		\centering
		\scalebox{1}[0.66]{\includegraphics[width=0.8\textwidth]{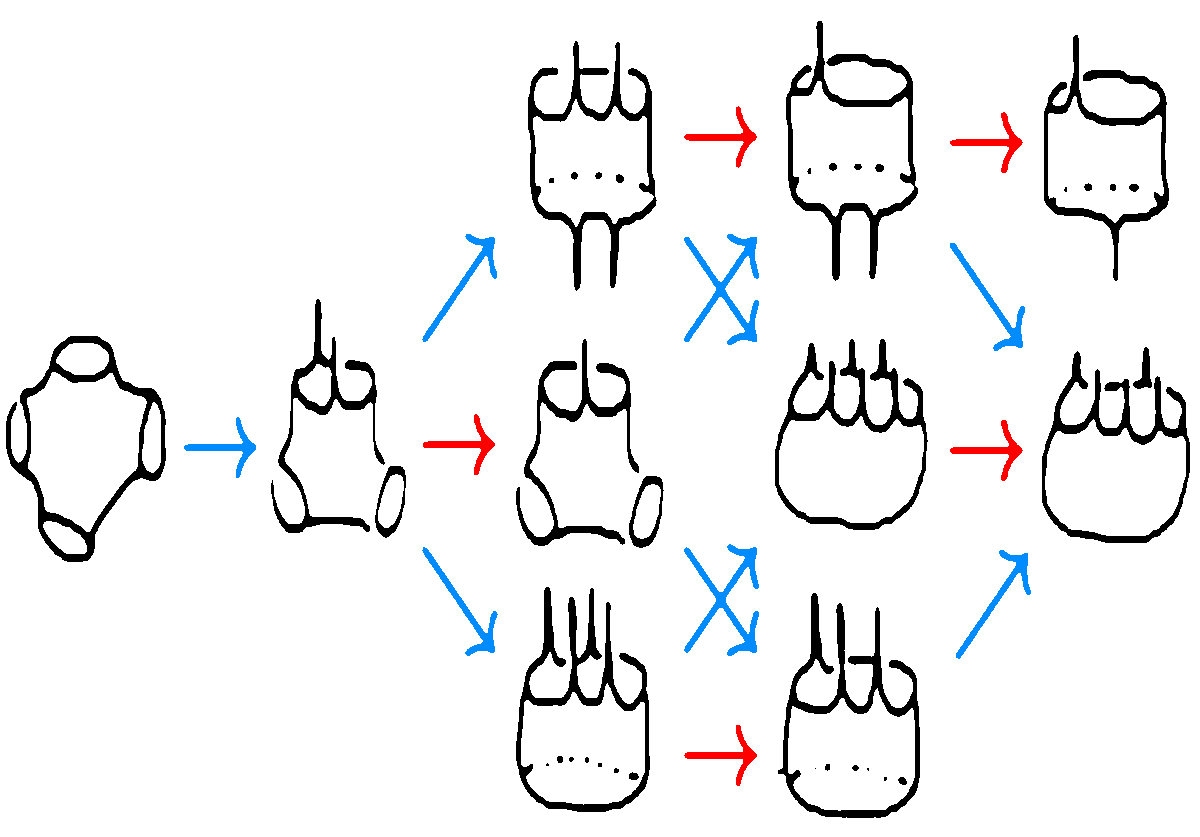}}	
		\caption{Degeneration of the isomonodromy problems corresponding to Painlev\'e equations according to \cite{Chekhov-Mazzocco-Rubtsov}.
		Blue arrows: confluence of singularities, red arrows: confluence of eigenvalues.}
		\label{figure:Painleve}
	\end{figure}	
where each surface corresponds to a class of isomonodromic deformations problems on $\CP^1$, each bordered hole in the surface represents a singularity, and each cusp on the border a Stokes direction.
Two kinds of confluent procedures (chewing gum operations) were described in \cite{Chekhov-Mazzocco, Chekhov-Mazzocco-Rubtsov}, see Figure~\ref{figure:chewinggum}:
\begin{itemize}[leftmargin=2\parindent]
	\item[(i)] Confluence of singularities (hole-hooking): As two holes come together, the strip between them becomes thinner and thinner until it snaps at the limit creating two new cusps.
	\item[(ii)] Confluence of eigenvalues (cusps removal): As two cusps come together, they merge into a single cusp with two prongs that disappear at the limit.
\end{itemize}
We describe a similar picture (see Figures~\protect\ref{figure:lagoons2} and \protect\ref{figure:lagoons}) in which the cusped holes represent domains (``lagoons'') over which the solution space acquires a natural flag structure, thus breaking the structure group $\SL_2(\C)$ of the vector bundle to its Borel subgroup of upper triangular matrices. 
	\begin{figure}
		\centering
		\includegraphics[width=0.7\textwidth]{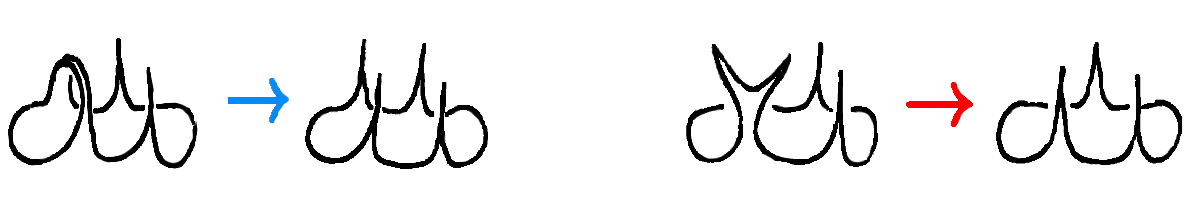}	
					\caption{Chewing gum operations.}
			\label{figure:chewinggum}
\end{figure}

\section{Individual singularities}\label{sec:1}

We consider germs of meromorphic connections having a singularity at $x=0$
\begin{equation}\label{eq:connection}
	\Nabla(x)=\d-A(x)\tfrac{\d x}{x^{k+1}},\qquad A(0)\neq 0,	
\end{equation}
with $A(x)\in\sl_2\big(\Cal O(\C,0)\big)$ a traceless $2\!\times\!2$ matrix with analytic coefficients.
They correspond to systems of linear ODEs
\begin{equation}\label{eq:system}
	x^{k+1}\!\tdd{x} y=A(x)y.
\end{equation}

There are two kinds of local changes of coordinates on the underlying vector bundle:
\begin{enumerate}[label=\arabic*.]	
	\item \emph{Analytic gauge transformations:} 
	\[(x,y)\mapsto \big(x,\ T(x)y\big),\quad 
	\text{with}\quad T(x)\in\GL(\Cal O(\C,0))\]
	a germ of analytic invertible matrix function, $\det T(0)\neq 0$.
	
	\item \emph{Analytic gauge--coordinate transformations:}
	\[(x,y)\mapsto \big(\phi(x),\ T(x)y\big),\quad 	\text{with}\quad \phi(x)\in\Diff(\C,0),\quad T(x)\in\GL(\Cal O(\C,0)),\]
	germs of an analytic diffeomorphism and of an analytic invertible matrix function, $\phi(0)=0$, $\tfrac{\d \phi}{\d x}(0)\neq 0$, $\det T(0)\neq 0$.
\end{enumerate}
The word \emph{analytic} is important: we don't want to modify the vector bundle.

One considers also their formal analogues $(x,y)\mapsto \big(\hat\phi(x),\ \hat T(x)y\big)$, with $\hat{\phi}(x),\ \hat T(x)$ formal power series.

	For a system \eqref{eq:system}, let us denote
	\[m=\ord_0 \det A(x)\]
	the vanishing order of $\det A(x)$ at the singularity.
	The number 
	\[\nu=\max\Big\{0,k-\tfrac{m}{2}\Big\} \] 
	is called the \emph{Katz rank} of the singularity, while $k\geq0$ is the \emph{Poincar\'e rank}.
	According to $k$ and $m$, the singularity is: 
	\begin{itemize}
		\item[] \emph{Irregular} if $0\leq m< 2k$, i.e. if $\nu>0$. 
		\begin{itemize}
			\item 	It is \emph{non-ramified} if $m$ is odd, and \emph{ramified} if $m$ is even. 
			\item 	It is \emph{non-resonant irregular} if $m=0$.  
		\end{itemize}
		
		\item[] \emph{Regular} if $m\geq 2k$, i.e. if $\nu=0$. 
		\begin{itemize}
			\item 	It is \emph{Fuchsian} if $k=0$. 
			\item 	It is \emph{resonant regular}\footnote{Our definition of resonance in the regular case is slightly more general than e.g. \cite{Ilashenko-Yakovenko} by including also the case $\sqrt{4\mu+k^2}=0$.} if $\sqrt{4\mu+k^2}\in\Z$ where $\mu=-\lim_{x\to0}x^{-2k}\det A(x)$.
		\end{itemize}
	\end{itemize}

\begin{notation}
	\begin{enumerate}[label=(\roman*)]
		\item For a germ of function $f(x)=\sum_{j=0}^{+\infty}f^{(j)}x^j$ we denote $\jet^nf(x)=\sum_{j=0}^{n}f^{(j)}x^j$ its $n$-jet.
		\item $\hot$ stands for \emph{higher order terms}. 
	\end{enumerate}
\end{notation}

\subsection{Point transformations of companion systems}

\begin{lemma}\label{lemma:systemQ}
	For any system \eqref{eq:system} there exists an analytic gauge transformation $y'=T(x)y$ which brings it to the form
\begin{equation}\label{eq:systemQ'}
x^{k+1}\!\tdd{x} y'=\begin{psmallmatrix}0 & 1\\[4pt] Q(x) & 0 \end{psmallmatrix}y',
\end{equation}	
where $Q(x)$ is given by \eqref{eq:Q}.
\end{lemma}

\begin{proof}
Let $A(x)=\left(a_{ij}(x)\right)$ be the matrix of the system. Since $0\neq A(0)\in\sl_2(\C)$, one can always arrange that $a_{12}(0)\neq 0$. Then the change of coordinate
\begin{equation}\label{eq:renormalization}
y'=\tfrac{1}{\sqrt{a_{12}(x)}}\begin{psmallmatrix} 1&0\\[4pt] a_{11}(x)-\frac12 x^{k+1}\!\frac{\d}{\d x}\log a_{12}(x) &\  a_{12}(x) \end{psmallmatrix}y
\end{equation}
is analytic and analytically invertible, and transforms \eqref{eq:system} to 	
\eqref{eq:systemQ'} 
with
\begin{equation}\label{eq:Q}
	\begin{split} Q(x)= &-\det A(x)+x^{k+1}\!\tdd{x} a_{11}(x)-a_{11}(x)x^{k+1}\!\tdd{x}\log a_{12}(x)\\ &-\tfrac{1}{2}\big(x^{k+1}\!\tdd{x}\big)^2\log a_{12}(x)+\tfrac{1}{4}\left(x^{k+1}\!\tdd{x}\log a_{12}(x)\right)^2.\end{split}
\end{equation}
\end{proof}

A system of the form
\begin{equation}\label{eq:systemQ}
x^{k+1}\!\tdd{x} y=\begin{psmallmatrix}0 & 1\\[4pt] Q(x) & 0 \end{psmallmatrix}y,
\end{equation}
with $y=\begin{psmallmatrix}y_1\\y_2\end{psmallmatrix}$ is a \emph{companion system} to a second order linear ODE 
\begin{equation}\label{eq:LDE}
\big(x^{k+1}\!\tdd{x}\big)^2y_1-Q(x)y_1=0.
\end{equation}

While coordinate transformations $x\mapsto \phi(x)$ alone do not preserve the class of companion systems \eqref{eq:systemQ},
their post-composition with the gauge transformation \eqref{eq:renormalization} do. 
These transformations have the from
	\begin{equation}\label{eq:coordinatechange}	
	(x,y)\mapsto (x',y')=\big(\phi(x),\ G_\phi(x)y\big), 
	\qquad G_\phi(x)=\left(\begin{smallmatrix} \psi(x)^{\frac12} & 0\\ \tfrac{x^{k+1}\!\tdd{x} \psi(x)}{2\psi(x)^{\frac32}} & \psi(x)^{-\frac12}\end{smallmatrix}\right),	
	\end{equation}
where 
\[ \psi=\frac{x^{k+1}\!\tdd{x}\phi}{\phi^{k+1}},\quad \text{is such that }\  x'^{k+1}\!\tdd{x'}=\psi(x)^{-1} x^{k+1}\!\tdd{x}.\]
They transform  companion system \eqref{eq:systemQ} with $Q(x)$ to one with $Q'(x')$ where
\begin{equation}\label{eq:coordinateQ}
	Q(x)=\psi^2\cdot Q'\circ\phi-\tfrac12\Big[x^{k+1}\!\tdd{x}\Big(\tfrac{x^{k+1}\!\tdd{x}\psi}{\psi}\Big)-\tfrac{1}{2}\Big(\tfrac{x^{k+1}\!\tdd{x}\psi}{\psi}\Big)^2\Big].
\end{equation}

\begin{definition}
	A \emph{formal, resp. analytic, point transformation} of a companion system \eqref{eq:systemQ} is a map \eqref{eq:coordinatechange} 
	where $\phi(x)$ is a formal, resp. analytic, germ of diffeomorphism of $(\C,0)$.	
\end{definition}

The point transformation \eqref{eq:coordinatechange} is the prolongation to $y=\left(\begin{smallmatrix} y_1\\ x^{k+1}\!\tdd{x} y_1\end{smallmatrix}\right)$ of the point transformation	
	\begin{equation}\label{eq:pointchange}	
	x\mapsto x'=\phi(x),\qquad y_1\mapsto y_1'=\psi(x)^{\frac12}y_1,
	\end{equation}
of the ODE \eqref{eq:LDE}.
As observed by Hawley \& Schiffer \cite{Hawley-Schiffer}, the solutions $y_1(x)$ of \eqref{eq:LDE} 
transform under \eqref{eq:pointchange} in the way of the $(-\tfrac12)$-differentials
\[y_1(x)\big(\tfrac{\d x}{x^{k+1}}\big)^{-\frac{1}{2}}=y_1'(x')\big(\tfrac{\d x'}{x'^{k+1}}\big)^{-\frac{1}{2}}.\]
Consequently, a composition of two point transformations \eqref{eq:coordinatechange} associated to diffeomorphisms $\phi_1$ and $\phi_2$ is the point transformation associated to $\phi_1\circ\phi_2$.

	If $Y(x)=\big(y_{ij}(x)\big)_{i,j}$ is a fundamental solution matrix for the system \eqref{eq:systemQ}, and $f(x)=\frac{y_{11}(x)}{y_{12}(x)}$, then $Q(x)=-\tfrac12\Cal S_{x^{k+1}\partial_{x}}(f)$, where 
	\begin{equation}\label{eq:Schwarzian}
		\Cal S_{x^{k+1}\partial_{x}}(f)=x^{k+1}\!\tdd{x}\bigg(\tfrac{\big(x^{k+1}\!\tdd{x}\big)^2f}{x^{k+1}\!\tdd{x} f}\bigg)-\tfrac{1}{2}\bigg(\tfrac{\big(x^{k+1}\!\tdd{x}\big)^2f}{x^{k+1}\!\tdd{x} f}\bigg)^2 
	\end{equation}
	is the \emph{``Schwarzian derivative''} associated to $x^{k+1}\!\tdd{x}$. The transformation rule \eqref{eq:coordinateQ} is the transformation rule for $\Cal S_{x^{k+1}\partial_{x}}$, see \cite{Klimes6}.
	
	The relation $Q(x)=\frac{\big(x^{k+1}\!\tdd{x}\big)^2y_{11}(x)}{y_{11}(x)}$ also implies that
	\begin{equation}\label{eq:Schwarziandifferential}
		Q(x)=-\tfrac12\Cal S_{x^{k+1}\partial_{x}}\Big(\int \tfrac{\d x}{y_{11}(x)^2x^{k+1}}\Big),
	\end{equation}
	cf. \eqref{eq:y1}.

\begin{lemma}\label{lemma:equivalenceofy11}
	Let $y_{11}(x)$, resp. $y_{11}'(x')$, be a particular non-trivial solution to \eqref{eq:LDE} with $Q(x)$, resp. $Q'(x')$.
	If $x'=\phi(x)$ is a coordinate transformation which pullbacks the differential form $\frac{\d x'}{y'_{11}(x')^2x'^{k+1}}$ to $\frac{\d x}{y_{11}(x)^2x^{k+1}}$, then
	the point transformation associated to $\phi$ transforms between the two companion systems \eqref{eq:systemQ}.
\end{lemma}

\begin{proof}
Follows from \eqref{eq:Schwarziandifferential}.
\end{proof}

\begin{remark}[$\SL_2(\C)$-opers]
Let $\sX$ be an open neighborhood of $0$ on which the system \eqref{eq:systemQ} is defined, $E=\sX\times\C^2$ a rank 2 vector bundle on $\sX$ and $\Cal E$ its sheaf of sections.
Let $\Sing=\{x^{k+1}=0\}$ be the polar divisor in $\sX$, and denote $\Omega^1_\sX(\Sing)=\Cal O_\sX\cdot \tfrac{\d x}{x^{k+1}}$ the sheaf of meromorphic 1-forms with poles bounded by $\Sing$ (no worse than order $k+1$ at $0$).

And let $\Nabla=\d-\begin{psmallmatrix}0 & 1\\ Q(x) & 0 \end{psmallmatrix}\tfrac{\d x}{x^{k+1}}$ the traceless meromorphic connection associated with the system \eqref{eq:systemQ} in the $(x,y)$-coordinate, $\Nabla:\Cal E\to \Cal E\otimes\Omega^1_\sX(\Sing)$.
Let $E_1\subset E$ be the sub-bundle whose sheaf of sections $\Cal E_1$ is defined by $y_2(x)=x^{k+1}\tdd{x}y_1(x)$.
Then $\Nabla\Cal E_1\cap \Cal E_1\otimes\Omega^1_\sX(\Sing)=\{0\}$, which means that $\Nabla$ induces an isomorphism 
\[\Cal E_1\xrightarrow{\simeq}(\Cal E/\Cal E_1)\otimes \Omega^1_\sX(\Sing).\]
Namely, if $\iota:\Cal O_\sX\xrightarrow{\simeq}\Cal E_1$ is the isomorphism $\iota\big(y_1(x)\big)=\begin{psmallmatrix} y_1(x)\\x^{k+1}\tdd{x}y_1(x)\end{psmallmatrix}$, and
$\pi_2:\Cal E\to\Cal O_\sX$ the projection $\pi_2\big(y(x)\big)=y_2(x)$, then
\[\pi_2\circ\Nabla\circ\iota:\Cal O_\sX\xrightarrow{\simeq}\Omega^1_\sX(\Sing),\qquad y_1(x)\mapsto \Big[\big(x^{k+1}\tdd{x}\big)^2y_1(x)-Q(x)y_1(x)\Big]\tfrac{\d x}{x^{k+1}}.\]
The structure $\big(E,E_1,\Nabla\big)$ is a meromorphic $\SL_2(\C)$-oper on $\sX$.
The analytic $\SL_2(\C)$-gauge--coordinate transformations that preserve the $\SL_2(\C)$-oper structure are compositions of point transformations with the flip $y\mapsto \pm y$, which corresponds to the liberty of choice of the root $\pm\psi^{\frac12}$ in \eqref{eq:pointchange}.
\end{remark}

\subsection{Formal theory}

There are three natural groups of local formal/analytic transformations acting on companion systems \eqref{eq:systemQ}:
\begin{itemize}[leftmargin=2\parindent]
	\item  gauge transformations, 
	\item gauge--coordinate transformations, and 
	\item point transformations,
\end{itemize}
and each of them contains a normal subgroup of tangent-to-identity transformations.
We shall investigate the formal invariants with respect to each of these groups.

In the \emph{non-resonant} regular and irregular cases the formal gauge invariants are simply polar parts of the eigenvalues of the matrix form 
$\begin{psmallmatrix}0 & 1\\ Q(x) & 0 \end{psmallmatrix}\tfrac{\d x}{x^{k+1}}$,
\[\pm\jet^k\sqrt{Q(x)}\tfrac{\d x}{x^{k+1}}.\]
As we shall show this has a natural generalization in taking a certain jet of the meromorphic quadratic differential 
\[Q(x)\left(\tfrac{\d x}{x^{k+1}}\right)^2.\]
Let
\[m=\ord_0 Q(x).\]

\begin{definition}
	Let $Q(x)\big(\tfrac{\d x}{x^{k+1}}\big)^2$ be a meromorphic quadratic differential.
	Its \emph{square residue} is defined as
\begin{equation}\label{eq:mu}
	\mu=\res_0^2 Q(x)\big(\tfrac{\d x}{x^{k+1}}\big)^2:=\left(\res_{x=0}\sqrt{Q(x)}\,\tfrac{\d x}{x^{k+1}}\right)^{\! 2}.
\end{equation}
	If $m$ is either odd or $m>2k$, then $\res_0^2 Q(x)\big(\tfrac{\d x}{x^{k+1}}\big)^2=0$. \\
	If $m\leq 2k$ is even, then $\res_0^2 Q(x)\big(\tfrac{\d x}{x^{k+1}}\big)^2$ is determined by the jet $\jet^{k+\frac{m}{2}}Q(x)$.
\end{definition}

Let us recall (see \cite[Theorems 6.1, 6.3, 6.4]{Strebel}) that the analytic class of a meromorphic quadratic differential $Q(x)\big(\tfrac{\d x}{x^{k+1}}\big)^2$ with respect to coordinate transformations $x\mapsto\phi(x)$ is completely determined by the number $2k-m$ and the quadratic residue $\mu$ \eqref{eq:mu}. 
Namely, it can be analytically transformed to one of the following normal forms:
\begin{itemize}[leftmargin=2\parindent]
	\item $m<2k$: \ $x^{m}\big(1+2\sqrt\mu\,  x^{k-\frac{m}{2}}\big)\big(\tfrac{\d x}{x^{k+1}}\big)^2$,\quad where $\mu=0$ for $m$ odd,
	\item $m=2k$:  \ $\mu\, \big(\tfrac{\d x}{x}\big)^2$,
	\item $m>2k$: \ $x^{m-2k}\big(\tfrac{\d x}{x}\big)^2$.
\end{itemize}

\begin{theorem}[Formal gauge classification]\label{theorem:formalgaugeequivalence}
Two such systems \eqref{eq:systemQ} of the same Poincar\'e rank $k\geq 0$ are formally gauge equivalent if and only if
\begin{equation}\label{eq:jetequality}
	\jet^{k+N}Q(x)=\jet^{k+N}Q'(x),\qquad 
N=\begin{cases} 0 & \text{if }\ m=0,\\ m-1 & \text{if }\ 0<m\leq 2k, \\  2k & \text{if }\ 2k<m, \end{cases}
\end{equation}
unless it is resonant 
\begin{equation}\label{eq:strongresonance}
m=2k,\quad \text{and}\quad  4\mu+k^2=l^2 \quad\text{for some }\ l\in\Z,\ l\geq \max\{1,k\}
\end{equation}
where $Q(x)=\mu x^{2k}+\hot$,
in which case additionally the following needs to be satisfied
\[\rho=0\ \Leftrightarrow\ \rho'=0,\] 
where $\rho,\rho'$ are defined by \eqref{eq:rho} below.

They are  formally tangent-to-identity gauge equivalent
if and only if 
\begin{equation}\label{eq:jetequality1}
	\jet^{k+\tilde N}Q(x)=\jet^{k+\tilde N}Q'(x),\qquad \tilde N=\min\{m,\, 2k\},
\end{equation}
except when \eqref{eq:strongresonance} in which case additionally the following needs to be satisfied
\[\rho=\rho'.\] 
In particular, each system \eqref{eq:systemQ} is formally tangent-to-identity gauge equivalent to a unique one in a formal normal form
\[\Qnf(x)=\begin{cases}
	\jet^{k+\tilde N} Q(x),&\\
	\jet^{k+\tilde N} Q(x)+a x^{l+2k}, \quad a\in\C,&\text{in the case \eqref{eq:strongresonance}}.
\end{cases}\]
\end{theorem}

\begin{theorem}[Formal gauge--coordinate classification]\label{theorem:formalgaugecoordinateequivalence}
Two systems \eqref{eq:systemQ} of the same Poincar\'e rank $k\geq 0$ are formally gauge--coordinate equivalent if and only if
\begin{itemize}[leftmargin=2\parindent]
	\item $0\leq m<2k$: \ $m=m'$ and \ $\mu=\mu',$
	\item $2k\leq m\leq 3k$: \ $\mu=\mu'$ and \ $\rho=0\Leftrightarrow \rho'=0$, where $\rho$ is \eqref{eq:rho} if $\sqrt{4\mu+k^2}\in\Z$ otherwise $\rho=0$.
	\item $m>3k$: \ always.
\end{itemize}	

In particular, each system \eqref{eq:systemQ} is formally gauge--coordinate equivalent to one with $Q(x)$ of the form
\[Q(x)=\begin{cases}
	x^m\big(1+2\sqrt\mu\, x^{k-\frac{m}{2}}\big),& \\[2pt]
	x^{2k}\mu,&\mu\neq 0,\\[2pt]
	x^{2k}\big(\mu + u(\rho) x^l\big),& 4\mu+k^2=l^2,\quad u(\rho)=\begin{cases}0&\text{if}\ \rho=0,\\ 1&\text{if}\ \rho\neq0.\end{cases} 
\end{cases} \]
\end{theorem}

\begin{theorem}[Formal/analytic point classification]\label{theorem:pointclassification}
	Two systems \eqref{eq:systemQ} are formally, resp. analytically, point equivalent if and only if they are formally, resp. analytically, gauge--coordinate equivalent. Likewise for tangent-to-identity equivalence.
\end{theorem}

The relative simplicity of Theorem~\ref{theorem:formalgaugeequivalence} stands in a contrast to the general theory of formal invariants by \cite{Balser-Jurkat-Lutz12, JLP12}, 
which is given in terms of a canonical form of a formal fundamental matrix solution.
In the $\sl_2(\C)$ case, it takes the form
\[\hat Y(x)=\begin{cases}
\hat T(x)\left(\begin{smallmatrix}1&p(x^{-1})\\[4pt]0&1\end{smallmatrix}\right)\!\!
\left(\begin{smallmatrix}x^{\alpha}&0\\[3pt]0&x^{-\alpha}\end{smallmatrix}\right)\!U
\e^{\int\jet^{k-1}\!\sqrt{Q(x)}\,\tfrac{\d x}{x^{k+1}}\left(\begin{smallmatrix} 1&0\\[3pt]0&-1 \end{smallmatrix}\right)},&m<2k,\\[12pt]
\hat T(x)\left(\begin{smallmatrix}1&p(x^{-1})\\[4pt]0&1\end{smallmatrix}\right)\!\!
\left(\begin{smallmatrix}x^{\alpha}&0\\[3pt]0&x^{-\alpha}\end{smallmatrix}\right)\!\!
	\left(\begin{smallmatrix}1&\delta_\rho\log x\\[3pt]0&1\end{smallmatrix}\right),&m\geq2k,\\
\end{cases}\]
where
\begin{itemize}[itemsep=1pt]
	\item[-] $\hat T(x)$ is a formal gauge transformation,
	\item[-] $p(x^{-1})$ is a polynomial of order $\leq\lfloor\frac{\tilde N}{2}\rfloor=\min\{\lfloor\frac{m}{2}\rfloor,k\}$ in $x^{-1}$ without a constant term,
	\item[-] $\alpha=\res_{x=0}\sqrt{Q(x)\left(\frac{\d x}{x^{k+1}}\right)^2+\left(\frac{\d Q(x)}{4Q(x)}\right)^2}-\frac{m}{4}=
	\begin{cases} \sqrt{\mu}-\frac{m}4, & m<2k,\\ 
		\sqrt{\mu+\left(\frac{m}{4}\right)^2}-\frac{m}4, & m=2k,\\
		0, & m>2k,\end{cases}$
	\item[-] $U=\left(\begin{smallmatrix}1&1\\[3pt]1&-1\end{smallmatrix}\right)$ if $m$ is odd, $U=I$ if $m$ is even,
	\item[-] $\delta_\rho=0$, unless $m\geq2k$ and $\alpha\in\frac12\Z$ and $\rho\neq 0$ \eqref{eq:rho} in which case it is $1$.
\end{itemize}
The quantities $\jet^{k-1}\!\sqrt{Q(x)}\,\tfrac{\d x}{x^{k+1}}$, $\alpha$, $\delta_\rho$ and $p(x^{-1})$ then form a complete set formal gauge invariants.
They are all determined by $\left(\jet^{k+N}Q(x)\right)\left(\frac{\d x}{x^{k+1}}\right)^2$, with the exception of $\delta_\rho$ for resonant regular singularities with $m=2k$.

\begin{lemma}\label{lemma:strongresonance}
	Assume the singularity is resonant regular: $m\geq 2k$ and $4\mu+k^2=l^2$ for some $l\in\Z_{\geq 0}$, where $\mu=\lim_{x\to0}\frac{Q(x)}{x^{2k}}$. 
	Then there exists a unique normalized ``subdominant'' solution $y_{11}(x)=x^{\frac{l-k}{2}}\big(1+o(x)\big)$ of \eqref{eq:LDE} with $x^{-\frac{l-k}{2}}y_{11}(x)$ analytic.
	
	The tangent-to-identity point transformations which preserve the system are precisely those associated to the diffeomorphisms
	\begin{equation}\label{eq:flow}
		\phi_s(x)=\exp\big(s\,y_{11}(x)^2x^{k+1}\!\tdd{x}\big),\quad s\in\C,
	\end{equation}
	preserving the differential $\frac{\d x}{y_{11}(x)^2x^{k+1}}$.
	Its residue
	\begin{equation}\label{eq:rho}
		\rho:=\res_0\frac{\d x}{y_{11}(x)^2x^{k+1}}. 
	\end{equation}	
is an invariant with respect to tangent-to-identity point transformations.
The monodromy is diagonalizable if and only if $\rho=0$.
\end{lemma}

\begin{proof}
	Write $Q(x)=x^{2k}\tilde Q(x)$ with $\tilde Q(0)=\mu=\frac{l-k}{2}\frac{l+k}{2}$.
	The equation \eqref{eq:LDE} is the same as
	\[\big(x\tdd{x}\big)^2y_1+kx\tdd{x}y_1-\tilde Q(x)y_1=0,\]
	whose companion matrix  $\left(\begin{smallmatrix} 0 & 1\\ \tilde Q(x) & -k\end{smallmatrix}\right)$ at $x=0$ has eigenvalues $\frac{l-k}{2}$ and $\frac{-l-k}{2}$.
	From the standard theory of linear differential equations it is known that if $l>0$ then there exists a unique 1-dimensional subspace of ``subdominant'' solutions
	of the form $y_1(x)=c_1x^{\frac{l-k}{2}}u(x)$, $c_1\in\C$, with $u(x)=1+\hot$ analytic, associated to the eigenvalue $\frac{l-k}{2}$ 
	This is also true if $l=0$, because in that case the leading matrix of the companion system is conjugated to $\begin{psmallmatrix}
		-\frac{k}2&1\\0&-\frac{k}2\end{psmallmatrix}$.
	Denoting $y_{11}(x)=x^{\frac{l-k}{2}}u(x)$ this normalized subdominant solution, the general 2-parameter solution can be expressed as
	\begin{equation}\label{eq:y1}
		y_1(x)=y_{11}(x)\Big(c_1+c_2\int\frac{\d x}{y_{11}(x)^2x^{k+1}}\Big),\qquad c_1,c_2\in\C.
	\end{equation}
	The monodromy is non-diagonalizable if and only if all non-subdominant solutions, i.e. those in \eqref{eq:y1} with $c_2\neq0$, contain a logarithmic term, i.e. if and only $\rho\neq0$ \eqref{eq:rho}.
	
	Point transformations that preserve the system must preserve also the subspace of subdominant solutions,
	which by \eqref{eq:pointchange} means that $x\mapsto\phi(x)$ must preserve the vector field $y_{11}(x)^2x^{k+1}\!\tdd{x}$ up to a multiplicative constant.
	In the case of transformation tangent to identity, the only diffeomorphisms that have this property are flow maps \eqref{eq:flow}.
\end{proof}

\begin{lemma}\label{lemma:sufficiency}
	Let two systems in a companion form \eqref{eq:systemQ} with $Q(x)$ and $Q'(x)$ of the same Poincar\'e rank $k$ be such  that
	\[Q'(x)=Q(x)+O(x^{\tilde N+l}),\quad\text{where}\quad \tilde N=\min\{m,2k\},\quad l>0,\]
	and moreover $\res_0\sqrt{Q'(x)}\frac{\d x}{x^{k+1}}=\res_0\sqrt{Q(x)}\frac{\d x}{x^{k+1}}$ if $m<2k$,
 or $\rho'=\rho$ \eqref{eq:rho} if $m\geq 2k$ and $\sqrt{4\mu+k^2}\in\Z$, $\mu=\res_0^2Q(x)\left(\frac{\d x}{x^{k+1}}\right)^2$.
 Then there exists a formal diffeomorphism 
	\[\phi(x)=x+O(x^{l+1}),\]
	whose associated point transformation $(x,y)\mapsto\big(\phi(x),\ G_\phi(x)y\big)$ \eqref{eq:coordinatechange} is a formal tangent-to-identity point equivalence between the systems.
	Such $\phi$ is unique except for when $m$ is even and $l\leq k-\frac{m}{2}$.
\end{lemma}

\begin{proof}
	It is enough to show that there exists $a_l\in\C$ such that the point transformation associated to the diffeomorphism $\phi_l(x)=x+a_lx^{l+1}$ 
	transforms the system with $Q'(x)$ to one with $\tilde Q(x)$ 
		\begin{equation}\label{eq:Qj}
		\tilde Q(x):= 
		\psi_l^2\cdot Q'\circ\phi_l-\tfrac12\Big[x^{k+1}\!\tdd{x}\Big(\tfrac{x^{k+1}\!\tdd{x}\psi_l}{\psi_l}\Big)-\tfrac{1}{2}\Big(\tfrac{x^{k+1}\!\tdd{x}\psi_l}{\psi_l}\Big)\!{\vphantom{\Big|}}^2\,\Big],	
	\end{equation}
	where $\psi_l=\tfrac{x^{k+1}\!\tdd{x}\phi_l}{\phi_l^{k+1}}$,
	satisfying 
	$\tilde Q(x)=Q(x)+O(x^{\tilde N+l+1})$. 
	The sought formal diffeomorphism $\phi(x)$ satisfying \eqref{eq:coordinateQ} would be constructed by induction as a formal infinite composition of such germs $\phi_j(x)=x+a_jx^{j+1}$, $j\geq l$. 
	
Let $Q'(x)-Q(x)=b_lx^{\tilde N+l}+O(x^{\tilde N+l+1})$.	
Write
\[Q'(x)=c'x^{m'}+\hot,\quad c'\neq 0,\quad\text{where }\  
\begin{cases}m'=m, & \text{if } m\leq 2k,\\ m'>2k, &\text{if } m>2k,\end{cases}\]
and  $\psi_l(x)=1+(l-k)a_lx^{l}+\hot$.
By \eqref{eq:Qj}
	$\tilde Q(x)-Q'(x)$ is the sum of
\[		\psi_l^2\cdot Q'\circ\phi_l-Q'=(2l-2k+m')c' a_lx^{m'+l}+\hot\]
and 
\[		-\tfrac12\big[x^{k+1}\!\tdd{x}\big(\tfrac{x^{k+1}\!\tdd{x}\psi_l}{\psi_l}\big)-\tfrac{1}{2}\big(\tfrac{x^{k+1}\!\tdd{x}\psi_l}{\psi_l}\big)^2\big]=-\tfrac12 (l-k)l(l+k)a_lx^{2k+l}+\hot.\]
	If $m'<2k$ then the first summand is of lower order, so if $l\neq k-\frac{m'}{2}$ then $a_l=-\frac{b_k}{(2l-2k+m')c'}$.
	If $l= k-\frac{m'}{2}$, then the assumption  $\res_0\sqrt{Q'(x)}\frac{\d x}{x^{k+1}}=\res_0\sqrt{Q(x)}\frac{\d x}{x^{k+1}}$ means that
	in fact $Q'=Q(x)+O(x^{\tilde N+l+1})$ so one can take $a_l=0$.
	If $m'>2k$ then the second summand is of lower order, so  $a_l=\frac{2b_k}{(l-k)l(l+k)}$ unless $l=k$,
and if $m'=2k$, then  $a_l=\frac{2b_k}{l\big((l-k)(l+k)-4c'\big)}$ unless $(l-k)(l+k)=4c'$.
Both of the exceptional cases correspond to a resonance condition $4\mu+k^2=l^2$ with $\mu=c'$ if $m'=2k$ and $\mu=0$ if $m'>2k$.
In this case, let $y_{11}(x)=x^{\frac{l-k}{2}}u(x)$ and $y_{11}'(x)=x^{\frac{l-k}{2}}u'(x)$ be the normalized subdominant solutions of Lemma~\ref{lemma:strongresonance}.
	If $\rho=\rho'$ \eqref{eq:rho}, then there exists an analytic diffeomorphism $\phi(x)$ which pullbacks the meromorphic differential $\frac{\d x}{y'_{11}(x)^2x^{k+1}}=\frac{\d x}{u'(x)^2x^{l+1}}$ to $\frac{\d x}{y_{11}(x)^2x^{k+1}}=\frac{\d x}{u(x)^2x^{l+1}}$, which by Lemma~\ref{lemma:equivalenceofy11} is what is needed. Let us show that moreover one can have $\phi(x)=x+O(x^{l+1})$.
	This holds if and only if $\jet^{l-1} u(x)=\jet^{l-1}u'(x)$, which is true thanks to the assumption that  $\jet^{2k+l-1} Q(x)=\jet^{2k+l-1}Q'(x)$.
	Indeed the equation \ref{eq:LDE} satisfied by $y_{11}(x)$ is equivalent to
	\[\big(x\tdd{x}\big)^2u+l\,x\tdd{x}\!u+\big(\mu-x^{-2k}Q(x)\big)u=0,\qquad u(0)=1,\]
	and it is easy to see that $\jet^{l-1} u(x)$ is uniquely determined by $\jet^{l+2k-1}\big(Q(x)\big)$.
	Therefore the assumption $\jet^{2k+l-1} Q(x)=\jet^{2k+l-1}Q'(x)$ implies $\jet^{l-1} u(x)=\jet^{l-1}u'(x)$. 
\end{proof}

\begin{lemma}[{\cite[Lemma 3.6]{Heu}}]\label{lemma:Heu}
	Given a system \eqref{eq:system} and a formal/analytic germ of diffeomorphism
	\[\phi(x)=x+O(x^{k+j+1}),\quad\text{for some }\ j\in\Z_{\geq0},\]
	there exists a formal/analytic germ of linear transformation $T(x)=I+O(x^j)$ such that the gauge--coordinate transformation
	\[x\mapsto\phi(x),\quad y\mapsto T(x)y\]
	preserves the system.
\end{lemma}

\begin{proof}
	Let $\hat v(x)x^{k+1}\!\tdd{x}$, $\hat v(x)=O(x^{j})$, be the formal infinitesimal generator for $\phi(x)$, i.e. the formal vector field whose formal time-1-flow is $\phi(x)$ (see e.g. \cite[chapter 3]{Ilashenko-Yakovenko}). 
	If $k=j=0$ then such infinitesimal generator is not unique, but that is fine.
	And	let $(x,y)\mapsto\big(\phi(x),\ \hat T(x)y\big)$ be the formal time-1-flow of the vector field 
	$\hat{\xi}(x,y)=\hat v(x)\!\left(x^{k+1}\!\tdd{x}+\big(\tfrac{\partial}{\partial y_1},\tfrac{\partial}{\partial y_2}\big)\cdot A(x)y\right)$. 
	Clearly it preserves the vector field $\hat{\xi}$, and therefore also the associated system \eqref{eq:system}. 
	The analyticity is by Lemma~\ref{lemma:formalanalytic}.
\end{proof}

\begin{lemma}\label{lemma:formalanalytic}
	Let  $(x,y)\mapsto\big(\hat\phi(x),\ G_{\hat\phi}(x)y\big)$ and $(x,y)\mapsto\big(x,\ \hat T(x)y\big)$ be a formal point transformation and a formal gauge transformation,
	such that their composition $(x,y)\mapsto\big(\hat\phi(x),\ \hat T(x)G_{\hat\phi}(x)y\big)$ preserves a meromorphic systems \eqref{eq:systemQ} and $\hat T(0)G_{\hat\phi}(0)=I$.
	Then $\hat\phi$ is convergent if and only if $\hat T(x)$ is.
\end{lemma}

\begin{proof}
	Assume first that $\hat\phi(x)$ is analytic. Then the system transformed by the point transformation  $(x,y)\mapsto\big(\hat\phi(x),\ G_{\hat\phi}(x)y\big)$ has the same Stokes matrices, and at the same formal invariants, since time it is also formally gauge equivalent to it by $y\mapsto\hat T(x)y$. 
	This means that if $\hat Y(x)$ is a formal fundamental solution matrix for the original system, then $G_{\hat\phi}(x)\hat Y(\hat\phi(x))=H(x)\hat Y(x)$ for an analytic gauge transformation $H(x)$. Since $\hat T(0)=H(0)$, it follows that $\hat T(x)=H(x)$ is analytic (the group of formal tangent-to-identity gauge transformations preserving a system is trivial).
		
	Conversely, assume that $\hat T(x)$ is analytic. Let $c=\tdd{x}\hat\phi(0)$, $\tilde\phi(x)=\frac1c\hat\phi(x)$. 
	Then the system transformed by $(x,y)\mapsto\big(cx,\,\hat T(x)G_cy\big)$ is in a companion form \eqref{eq:systemQ}, and $\tilde\phi(x)$ is the unique formal coordinate change of Lemma~\ref{lemma:sufficiency}. 
	The analyticity of this unique $\tilde\phi(x)$ has been proven for $m=0,1$ in \cite[Theorem 1.1]{Klimes6}. In case of traceless systems the proof works for general $m$ as well (see also the proof of Theorem~\ref{prop:unfoldedpointequivalence} in this paper which follows the same lines).
\end{proof}

\begin{corollary}
	Formal gauge--coordinate transformations between regular singularities ($m\geq 2k$) are convergent.
\end{corollary}

\begin{proof}[\bfseries Proof of Theorem~\ref{theorem:formalgaugeequivalence}]
Let us first show the necessity of the conditions.	
If $y'=T(x)y$ is a formal gauge transformation, then 
\begin{equation*}\label{eq:gaugeA}
	A(x)=T(x)^{-1}\Big( A'(x)T(x)-x^{k+1}\!\tdd{x} T(x)\Big).
\end{equation*}
implies that $\det A(x)=\det A'(x)+O(x^{k+1})$, 
and hence $\ord_0(Q'-Q)\geq k+1$.
For $m>0$ we need to be a bit more precise.
Writing $T(x)$ as
\begin{equation}\label{eq:T}
T(x)=a(x)I+b(x)\left(\begin{smallmatrix}0&1\\[4pt]Q(x)&0\end{smallmatrix}\right)+c(x)\left(\begin{smallmatrix}0&0\\[4pt]1&0\end{smallmatrix}\right)+
d(x)\left(\begin{smallmatrix}1&0\\[4pt]0&-1\end{smallmatrix}\right),
\end{equation}
the transformation equation
\begin{equation}\label{eq:conjugationT}
x^{k+1}\!\tdd{x} T=\big[\left(\begin{smallmatrix}0&1\\[4pt]Q&0\end{smallmatrix}\right)\!,\,T\big]+
(Q'-Q)\left(\begin{smallmatrix}0&0\\[4pt]1&0\end{smallmatrix}\right)T
\end{equation}
becomes
\begin{equation}\label{eq:abcd}
\begin{aligned}
x^{k+1}\!\tdd{x} b&=-2d\\
x^{k+1}\!\tdd{x}(a+d)&=c\\
x^{k+1}\!\tdd{x}(a-d)&=-c+(Q'-Q)b\\
 x^{k+1}\!\tdd{x}(Qb+c)&=(Q'+Q)d+(Q'-Q)a,
\end{aligned}	
\end{equation}
from which
\[ d=-\tfrac12x^{k+1}\!\tdd{x} b,\qquad c=\tfrac12(Q'-Q)b-\tfrac12\big(x^{k+1}\!\tdd{x}\big)^2b,\]
and
\begin{equation}\label{eq:ab}
\begin{aligned}
x^{k+1}\!\tdd{x} a&=\tfrac12 (Q'-Q)b,\\
(Q'+Q)x^{k+1}\!\tdd{x} b+\tfrac12bx^{k+1}\!\tdd{x}(Q'+Q)-\tfrac12\big(x^{k+1}\!\tdd{x}\big)^3b&=(Q'-Q)a.
\end{aligned}
\end{equation}
If $m>0$, then $T$ is invertible (resp. tangent to identity) if and only if $a(0)\neq 0$ (resp. $a(0)=1$, $b(0)=0$), and it follows from \eqref{eq:ab} that 
\[\ord_0(Q'-Q)\geq\min\{k+m,\ 3k+1\},\quad\text{resp. }\ \ord_0(Q'-Q)\geq\min\{k+m+1,\ 3k+1\}\]
which gives \eqref{eq:jetequality}.
In the case \eqref{eq:strongresonance} Lemma~\ref{lemma:strongresonance} applies.

The sufficiency of the conditions follows from the combination of Lemma~\ref{lemma:sufficiency} 
with
$l=\begin{cases} k,& \text{if}\ 0<m\leq 2k,\\ k+1,&\text{otherwise},\end{cases}$ if \eqref{eq:jetequality} or
$l=k+1$ if \eqref{eq:jetequality1},
and of Lemma~\ref{lemma:Heu}.
Note that the residue is determined at the jet $\jet^{k+\frac{\tilde N}{2}}Q(x)$, $k+\lceil\frac{\tilde N}{2}\rceil\leq \tilde N+l-1$.
\end{proof}

\begin{proof}[\bfseries Proof of Theorem~\ref{theorem:formalgaugecoordinateequivalence}]
	By Lemma~\ref{lemma:sufficiency}, one can use the point transformation associated with $\phi_l(x)=x+a_lx^{l+1}$ to get rid of the term $x^{\tilde N+l}$, $\tilde N=\min\{m,2k\}$, in $Q(x)$ for any $l\geq 1$,
	except if $l=k-\frac{m}{2}$ and $\tilde N=m\leq2k$ (i.e. of the term carrying the quadratic residue $\mu$ \eqref{eq:mu}), or if $l^2=4\mu+k^2$ and $\tilde N=2k\geq m$ (i.e. of the term carrying the residue $\rho$ \eqref{eq:rho} in the resonant regular case).
	Note that if $\sqrt{\mu'}=-\sqrt{\mu}$, $m<2k$, one may have to use the transformation $x\mapsto \e^{\frac{2\pi\i}{2k-m}}x$ to invert the sign. 
\end{proof}

\begin{proof}[\bfseries Proof of Theorem~\ref{theorem:pointclassification}]
	We need to show that gauge--coordinate equivalence implies point equivalence. Since point transformations form a group, it is enough to show that
	gauge equivalence implies point equivalence (resp. tangent-to-identity gauge equivalence implies tangent-to-identity point equivalence). 
	In the formal setting this follows from Lemma~\ref{lemma:sufficiency} and Theorem~\ref{theorem:formalgaugeequivalence}.
	The analytic version follows additionally from Lemma~\ref{lemma:formalanalytic}.
\end{proof}

\begin{proposition}\label{prop:flow}
Given a companion system \eqref{eq:systemQ} with even $m<2k$, 
there exists a formal vector field $\hat Y_0(x)=x^{k-\frac{m}{2}+1}(1+\hot(x))\tdd{x}$, such that the formal point transformations \eqref{eq:coordinatechange} associated
to its formal flow maps $\hat\phi_s(x)=\exp(s\hat Y_0)(x)\in\widehat{\Diff}(\C,0)$, $s\in\C$, preserve the system.
The vector field $\hat Y_0(x)$ is analytic if and only if the system is meromorphically gauge equivalent to a diagonal one.
\end{proposition}

\begin{proof}
Write $\hat Y_0(x)=\hat h(x)\,x^{k+1}\!\tdd{x}$. The action of the flow maps $\hat\phi_s$ give rise to a 1-parameter group of formal point transformations. Their formal infinitesimal generator is the vector field in $(x,y)$-coordinates
\[\hat Y_2(x,y)= \hat h(x)x^{k+1}\!\tdd{x} + \tfrac12\big(x^{k+1}\!\tdd{x}\hat h(x)\big)\big(y_1\tdd{y_1}-y_2\tdd{y_2}\big)\, + 
\tfrac12\big((x^{k+1}\!\tdd{x})^2\hat h(x)\big)y_1\tdd{y_2},
\]	
which is the jet prolongation of $\hat Y_1(x,y_1)= \hat h(x)x^{k+1}\!\tdd{x} + \tfrac12\big(x^{k+1}\!\tdd{x}\hat h(x)\big)y_1\tdd{y_1}$.
The point transformations, given by the flow of $\hat Y_2$, preserve the system \eqref{eq:systemQ} if and only if
\[\big[X,\hat Y_2\big]=\hat{\alpha}(x)X,\qquad\text{where}\quad X(x,y)=x^{k+1}\!\tdd{x} +y_2\tdd{y_1}+Q(x)y_1\tdd{y_2},\]
for some $\hat{\alpha}(x)$.
This is equivalent to $\hat\alpha(x)=x^{k+1}\!\tdd{x}\hat h(x)$ and $\hat h(x)$ being a solution to
\begin{equation}\label{eq:thirdorderODE}
	\big(x^{k+1}\!\tdd{x}\big)^3\hat h(x)-4Q(x)\big(x^{k+1}\!\tdd{x}\hat h(x)\big)-2\hat h(x)\big(x^{k+1}\!\tdd{x}Q(x)\big),
\end{equation}
see \cite[\S1.4]{Klimes6}. This 3rd order linear ODE has a unique formal power series solution of the form $\hat h(x)=x^{-\frac{m}{2}}(1+\hot(x))$, which in the irregular case $m<2k$ is analytic if and only if the system is meromorphically diagonalizable \cite[Theorem 1.13 and Proposition 1.7]{Klimes6}.
%
\end{proof}

\begin{corollary}
For any system \eqref{eq:system} with even $m<2k$, and any $s\in\C$, there exists an analytic gauge--coordinate transformation $\big(\phi(x),\ T(x)y\big)$, with $\phi(x)=x+sx^{k-\frac{m}{2}+1}+\hot(x)$, which preserves the system.	
\end{corollary}
\begin{proof}
We can assume that the system is of the companion form \eqref{eq:systemQ}. Let $\hat{\phi}_s(x)$ be the formal diffeomorphism of Proposition~\ref{prop:flow} whose associated point transformation $(\hat\phi_s,\ G_{\hat\phi_s})$ preserve the system.
Let $\phi=\jet^{k+N+1}\hat\phi_s$ where $N$ is as in Theorem~\ref{theorem:formalgaugeequivalence}.
Then the system transformed by the point change associated to $\phi$ has the same $\jet^{k+N}Q$ and therefore is in the same formal gauge equivalence class. 
The convergence of the formal gauge equivalence $T(x)$ is by Lemma~\ref{lemma:formalanalytic}.
\end{proof}

\section{Isomonodromic deformations of meromorphic connections}\label{sec:isomonodromy}

Universal isomonodromic deformations of meromorphic connections on Riemann surfaces have been constructed by Jimbo, Miwa \& Ueno \cite{Jimbo-Miwa-Ueno},
Malgrange \cite{Malgrange-Iso} and others, under a generic assumption on non-resonance of their singularities.
In the case of rank 2 vector bundles, a construction of a universal isomonodromic deformation  has been achieved by Heu \cite{Heu} for connections with
any types of singularities. 
We will show that the natural space for the universal deformation parameter is given by a product 
of the configuration space of singularities with a tuple of jet spaces of quadratic differentials, one for each singularity. 

\medskip
A family of meromorphic connections $\nabla_t(x)$ on a family of analytic vector bundles $E_t\to \sX$ over a Riemann surface $\sX$, with poles
at a divisor $\Sigma_t\subset \sX$, all depending analytically on some parameter $t\in \sT$, is called \emph{isomonodromic}, if there exists a meromorphic  connection $\nabla(x,t)=\d-\Omega(x,t)$ on the total bundle $\Cal E=\coprod_{t}E_t\to \sX\times \sT$ with poles over $\Sigma=\coprod_{t}\Sigma_t$ 
which is \emph{flat} 
\[\d\Omega-\Omega\wedge\Omega=0,\]
and of which $\nabla_t(x)$ are restrictions for all $t\in \sT$.

We assume the polar divisor $\Sigma$ is smooth at $t=0$, i.e. has no crossing. 
In a local coordinate $(x,t,y)$ in which the divisor is $\{x=0\}$, the family becomes
\begin{equation*}
\nabla_t(x)=\d-A(x,t)\tfrac{\d x}{x^{k+1}},\qquad t\in(\C^l,0).
\end{equation*}
The isomonodromicity asks for existence of meromorphic matricial forms $B_i(x,t)\tfrac{\d t_i}{x^k}$ that make the connection
\begin{equation}\label{eq:connection-t} 
	\nabla=\d-A(x,t)\tfrac{\d x}{x^{k+1}}-\sum_{i=1}^l B_i(x,t)\tfrac{\d t_i}{x^k}
\end{equation}
flat.

\begin{definition}
	Let us suppose that $A(0,0)\neq0$, i.e. that the order of the pole of $A(x,t)\tfrac{\d x}{x^{k+1}}$ at $x=0$ is locally constant. 
	Flat connection $\nabla(x,t)$ is said to satisfy a \emph{transversality condition} \cite{Heu} at $\{x=0\}$, if it takes the form \eqref{eq:connection-t}
	with analytic $B_i(x,t)$, $i=1,\ldots,l$.
\end{definition}

\begin{proposition}[Heu {\cite[Proposition 2.5]{Heu}}]\label{proposition:Heu}
A flat traceless meromorphic connection on rank 2 vector bundle \eqref{eq:connection-t} with locally constant Poincar\'e rank $k$ and Katz rank $\nu=\max\{0,k-\frac{m}{2}\}$ along the polar divisor satisfies the transversality condition if and only if there exists an analytic gauge-coordinate transformation
$(x,t,y)\mapsto\big(\phi(x,t),\,t,\,T(x,t)y\big)$ 
transforming the isomonodromic deformation to a constant one of the form
\[\nabla=\d-\begin{psmallmatrix}0&1\\[3pt]Q(x)&0\end{psmallmatrix}\frac{\d x}{x^{k+1}}.\]
\end{proposition}

\begin{corollary}\label{cor:Heu}
	Isomonodromic deformations satisfying transversality condition are iso-Stokes.
\end{corollary}

\begin{remark}
 In some cases there exist isomonodromic deformations that don't satisfy the transversality condition. In the case with simple poles (Fuchsian singularities) these are the non-Schlesinger deformations \cite{Bolibrukh-isom}.
\end{remark}

\begin{theorem}\label{theorem:isomonodromic}
The space of parameters of a universal isomonodromic deformation of a connection $\nabla_{0}$ with $p$ singularities 
$\Sigma_0=\{a_{1}^*,\ldots,a_{p}^*\}$ of Poincar\'e ranks $k_i$ and Katz ranks $\nu_i$
is the universal cover $\tilde \sT$ of
\[\sT=\Big(\Sigma\times\prod_{i=1}^{n}\sP_i\Big) \big/\Aut(\sX,\Sigma),\]
where  $\Sigma=\{(a_1,\ldots,a_n)\in \sX^n\mid a_i\neq a_j\}$ is the space of configurations of $p$ distinct points in $\sX$,
and $\sP_i$ are the spaces of local deformation parameters at the singularity $a_i$,
identified with jet spaces of meromorphic quadratic differentials:
\begin{equation}\label{eq:universalparameter}
	\sP_i= \Big\{\Delta_i(x,\bm q_i)=\big(q_{i,n_i}x^{n_i}+\ldots+q_{i,k_i+N_i}x^{k_i+N_i}\big)\big(\tfrac{\d x}{x^{k_i+1}}\big)^2 \ \mid\ \res_0^2 \Delta_i(x,\bm q_i)=\mu_i\Big\},
\end{equation}
where $n_i=\min\{m_i,2k_i+1\}$ and $N_i$ is as in \eqref{eq:jetequality}, namely
\[(n_i,\,k_i+N_i)=\begin{cases} (0,\,k_i) & \text{if }\ m_i=0,\\ (m_i,\,k_i+m_i-1) & \text{if }\ 0<m_i\leq 2k_i, \\  (2k_i+1,\,3k_i) & \text{if }\ 2k_i<m_i, \end{cases}\] 
and $\bm q_i=(q_{i,n_i},\ldots,q_{i, k_i+N_i})\in\begin{cases}
\C^*\times\C^{k_i+N_i-n_i}&\text{if}\ m_i\leq 2k_i\\	
\C^{k_i+N_i-n_i+1} &\text{if}\ 2k_1<m_i
\end{cases}$.\\
In particular, 
\[\dim \sP_i=\begin{cases} k_i-1,& \text{if}\ 0<m_i\leq 2k_i\ \text{is even},\\ k_i,&\text{otherwise}.\end{cases}\]
\end{theorem}

\begin{proof}
Following the construction of universal isomonodromic deformations of meromorphic $\sl_2(\C)$-connections by V.~Heu \cite{Heu}, the local parameter spaces correspond to
the quotient
\begin{equation*}
\sP_i=\left\{\begin{array}{c}\text{analytic gauge-coordinate}\\ \text{equivalence class} \\ \text{of $\Nabla_0$ at $a_i^*$}\end{array}\right\}\Big/
\left[\begin{array}{c}\text{analytic gauge}\\ \text{equivalence}\end{array}\right].
\end{equation*}
One may here replace ``analytic'' by ``formal'' since by Corollary~\ref{cor:Heu} the Stokes modulus of the singularity stays constant along the isomonodromic deformation: 
\begin{equation*}
\sP_i\simeq\left\{\begin{array}{c}\text{formal gauge-coordinate}\\ \text{equivalence class} \\ \text{of $\Nabla_0$ at $a_i^*$}\end{array}\right\}\Big/
\left[\begin{array}{c}\text{formal gauge}\\ \text{equivalence}\end{array}\right],
\end{equation*}
which by Theorems~\ref{theorem:formalgaugecoordinateequivalence} and~\ref{theorem:formalgaugeequivalence}  can be identified with \eqref{eq:universalparameter}.
\end{proof}

\begin{remark}
In \cite{Heu} the local parameter space $\sP_i$ at the singularity $x_i$ of Poincar\'e rank $k_i>0$ is identified with $\jet^{k_i}$-jets of germs diffeomorphisms.
These diffeomorphisms act on the quadratic differential $\big(\jet^{k_i+N_i}Q_i^*(x)\big)\big(\tfrac{\d x}{x^{k+1}}\big)^2$ that is the formal invariant of $\Nabla_0$ at the singularity $a_i^*$ by point transformations \eqref{eq:coordinateQ} restricted to the $\jet^{k_i+N_i}$ jet.
By Proposition~\ref{prop:flow}, in the case when $0<m_i\leq 2k_i$ is even
there exists a one-parameter subgroup of the group of formal diffeomorphisms with non-trivial $\jet^{k_i}$-jet, whose point transformation action \eqref{eq:coordinateQ} preserve the differential
$\jet^{k_i+N_i}Q_i^*(x)\big(\tfrac{\d x}{x^{k+1}}\big)^2$. 
Therefore in this case the number of local parameters is only $k_i-1$, correcting the statement of \cite[\S 3.3.2]{Heu}. 
\end{remark}

\section{Confluences and degenerations of singularities}\label{sec:2}

An unfolding of a meromorphic connection \eqref{eq:connection} at the singularity at $x=0$ is a parametric family of connections
\begin{equation}\label{eq:unfoldedconnection}
\nabla_\epsilon(x)=\d-A(x,\epsilon)\tfrac{\d x}{P(x,\epsilon)},\qquad (x,\epsilon)\in(\C\times\C^l,0),
\end{equation}
with $P(x,0)=x^{k+1}$, $A(x,\epsilon)\in\sl_2(\C)$, depending analytically on a parameter $\epsilon\in(\C^l,0)$.
It corresponds to a family of systems
\begin{equation}\label{eq:unfoldedsystem}
	P(x,\epsilon)\tdd{x} y=A(x,\epsilon)y.
\end{equation}
With the use of Weierstrass preparation theorem, one may assume that
\begin{equation}\label{eq:P}
	P(x,\epsilon)=x^{k+1}+P_k(\epsilon)x^k+\ldots+P_0(\epsilon).
\end{equation}
We denote
\[m=\ord_0 \det A(x,0).\]
We will only study unfoldings of irregular singularities, $m<2k$.


\begin{lemma}\label{lemma:systemQepsilon}
	There exists an analytic gauge transformation $y\mapsto T(x,\epsilon)y$ which brings the family \eqref{eq:unfoldedsystem} to the form
	\begin{equation}\label{eq:unfoldedQ}
		P(x,\epsilon)\tdd{x} y=\begin{psmallmatrix}0&1\\[3pt]Q(x,\epsilon)& 0\end{psmallmatrix}y.
	\end{equation}
\end{lemma}
\begin{proof}
Same as Lemma~\ref{lemma:systemQ} where one replaces the derivative $x^{k+1}\tdd{x}$ by $P(x,\epsilon)\tdd{x}$.
\end{proof}

The system \eqref{eq:unfoldedQ} is a companion system associated to the second order linear ODE
\begin{equation}\label{eq:unfoldedODE}
	\left(P(x,\epsilon)\tdd{x}\right)^2y_1-Q(x,\epsilon)\,y_1=0. 
\end{equation}
If $y_{11}$, $y_{12}$ are two linearly independent solutions, then $Q(x,\epsilon)=-\tfrac12\Cal S_{P(x,\epsilon)\partial_{x}}\Big(\frac{y_{11}(x,\epsilon)}{y_{12}(x,\epsilon)}\Big)$ as in \eqref{eq:Schwarzian}.

\medskip
Again there are three kinds of \emph{parametric transformations}:
\begin{enumerate}[label=\arabic*.]	
	\item \emph{Parametric gauge transformations}:
	\[y'=T(x,\epsilon)y, \qquad \det T(0,0)\neq 0.\]
	\item \emph{Parametric gauge--coordinate transformations}:
	\[		x'=\phi(x,\epsilon),\quad  y'=T(x,\epsilon)y, 		\qquad \tdd{x}\phi(0,0)\neq 0, \quad\det T(0,0)\neq 0,	\]
	\item  \emph{Parametric point transformations}:
	\begin{equation}\label{eq:parametricpoint}
	x'=\phi(x,\epsilon),\qquad y'=\left(\begin{smallmatrix} \psi(x,\epsilon)^{\frac12} & 0\\ \!\!\mfrac{P(x,\epsilon)\tdd{x}\psi(x,\epsilon)}{2\psi(x,\epsilon)^{\frac32}} & \ \ \psi(x,\epsilon)^{-\frac12}\!\!\end{smallmatrix}\right)y, \qquad \tdd{x}\phi(0,0)\neq 0,
	\end{equation}
	with \ $\psi(x,\epsilon)=\mfrac{P(x,\epsilon)\tdd{x}\phi}{P'\circ\phi(x,\epsilon)}$, \	
where  $P(x,\epsilon)$, $P'(x',\epsilon)$ are monic Weierstrass polynomials of the form \eqref{eq:P} and $P$ divides $P'\circ\phi$.	
\end{enumerate}
They are \emph{analytic} if $\phi(x,\epsilon)$, $T(x,\epsilon)$ are analytic germs, and \emph{formal} if
\[ \phi(x,\epsilon)=\sum_{i,\bm j\geq 0}\phi_{i,\bm j} x^i\epsilon^{\bm j}, \qquad  T(x,\epsilon)=\sum_{i,\bm j\geq 0}T_{i,\bm j} x^i\epsilon^{\bm j}.\]
It is easy to verify that both in the formal and the analytic setting each of the three kinds of parametric transformations form a group.

\subsection{Formal theory}

In the cases $k>0,\ m=0$ (non-resonant irregular) and $k>0,\ m=1$ (generic resonant irregular) the formal gauge invariants of the system at $\epsilon=0$ are given by $\jet^kQ(x,0)=Q(x,0)\mod x^{k+1}$. 
This easily generalizes in the parametric setting as the Weierstrass division remainder $Q(x,\epsilon)\mod P(x,\epsilon)$ (Theorem~\ref{thm:unfoldedformal}).
On the other hand, for general $m$ things get a bit more tricky and we shall instead use a weaker notion of formal equivalence. 

\begin{definition}[Weak formal gauge equivalence]
	Two analytic germs of parametric systems \eqref{eq:unfoldedsystem} with the same singular divisor $P(x,\epsilon)$ are \emph{weakly formally gauge equivalent}
	if for each $\epsilon$ small enough the two systems restricted to $\epsilon$ have the same formal invariants
	at each zero of $P(x,\epsilon)$ which is not a resonant regular singularity.
\end{definition}

The values of $\epsilon$ for which some singularity is resonant regular form at most a countable union of divisors.

For $m=0,1$, this weak formal gauge equivalence agrees with the formal gauge equivalence of parametric families (Theorem~\ref{cor:weakformal}).

\begin{proposition}\label{prop:weakformal}
Consider two systems \eqref{eq:unfoldedQ} with germs $P(x,\epsilon)$, $Q(x,\epsilon)$, and $P'(x,\epsilon)$, $Q'(x,\epsilon)$.
If $P'(x,\epsilon)= P(x,\epsilon)$ and
\begin{equation}\label{eq:reducedDelta}
	\tfrac{Q(x,\epsilon)}{Q'(x,\epsilon)}=1\mod P(x,\epsilon),
\end{equation}	 
then the two systems are weakly formally equivalent.
\end{proposition}

\begin{proof}
For any zero $a_i(\epsilon)$ of $P(x,\epsilon)$ of multiplicity $k_i+1>0$, and with $Q(x,\epsilon)$ vanishing at $a_i$ up to an order $m_i\geq 0$,
the $k_i+m_i$-the jets with respect to $x$ at $a_i$ of $Q$ and $\tilde Q$ agree 
\[\jet_{a_i}^{k_i+m_i} Q'=\jet_{a_i}^{k_i+m_i}Q.\]
By Theorem~\ref{theorem:formalgaugeequivalence} this means that the two systems have the same formal tangent-to-identity invariants at all their singularities.	
\end{proof}

\subsection{Subdominant solutions and mixed solution bases}\label{sec:subdominant}

Let \eqref{eq:unfoldedQ} be a parametric system in the companion form  and $Q(x,\epsilon)\left(\tfrac{\d x}{P(x,\epsilon)}\right)^2$  its associated
meromorphic quadratic differential.
Following Proposition~\ref{prop:weakformal} we shall take its Weierstrass division by a multiplicative factor $\big(1+P(x,\epsilon)\C\{x,\epsilon\}\big)$.
Namely, decomposing by the Weierstrass preparation and division theorems
\[Q(x,\epsilon)=q(x,\epsilon)\big(U(x,\epsilon)+P(x,\epsilon)f(x,\epsilon)\big),\]
with $q(x,\epsilon)$, resp. $U(x,\epsilon)$, polynomials in $x$ of orders $m$, resp. $k$, with $u(0,0)\neq 0$ and $f(x,\epsilon)\in\C\{x,\epsilon\}$, 
let 
\[\tilde Q(x,\epsilon)=q(x,\epsilon)U(x,\epsilon)\]
be the unique polynomial in $x$ of order $k+m$ satisfying $\frac{Q'(x,\epsilon)}{\tilde Q(x,\epsilon)}=1\mod P(x,\epsilon)\C\{x,\epsilon\}$.
The quadratic differential 
\begin{equation}\label{eq:quaddiffomega}
	\Delta(x,\epsilon)=\tilde Q(x,\epsilon)\left(\tfrac{\d x}{P(x,\epsilon)}\right)^2
\end{equation}
carries information about the formal invariants but not about the analytic class.
We denote
\[\pm\Delta^{\frac12}=\pm\tfrac{\sqrt{\tilde Q}}{P(x,\epsilon)}\d x,\qquad \pm\Delta^{-\frac12}=\pm\tfrac{P(x,\epsilon)}{\sqrt{\tilde Q}}\d x,\]
the associated differential form and vector field. 



\smallskip

Solutions of the ODE \eqref{eq:unfoldedODE} are expected to exhibit asymptotic behavior near the singular points, approximately proportional to either
$y_1^+\simeq Q(x,\epsilon)^{-\frac14}\e^{\pm\int \Delta^{\frac12}}$ or $y_1^-\simeq Q(x,\epsilon)^{-\frac14}\e^{\mp\int \Delta^{\frac12}}$
on appropriately chosen domains, as detailed in Theorem~\ref{prop:subdominant}.
Depending which of the two exponentials dominates, that is, whether
\begin{equation}\label{eq:expasymptotic}
	\left|\tfrac{y_1^+}{y_1^-}\right|\simeq\left|\e^{\pm2\int \Delta^{\frac12}}\right| 
\end{equation}
tends to infinity or zero, the solutions are classified as \emph{dominant} or \emph{sub-dominant}, respectively (terminology from Sibuya \cite{Sibuya0, Sibuya2}).
This classification introduces the concept of a \emph{growth rate filtration}, leading to a \emph{flag structure} on the solution space.
The domains where such filtrations exist are intrinsically tied to two distinct types of foliations:

\begin{definition}[Horizontal foliation]\label{def:horizontalfoliation}~
\begin{enumerate}
\item	
The \emph{horizontal foliation} of the quadratic differential $\Delta(x,\epsilon)$ is formed by the
gradient curves  of \eqref{eq:expasymptotic}:  those along which $\Im\left(\pm\int\Delta^{\frac12}\right)$ stays constant.
They are the real time trajectories of the vector field $\pm\Delta^{-\frac12}$,
well defined up to orientation.	
They are also called the \emph{steepest descent curves}, or \emph{anti-Stokes curves}.

\item The \emph{vertical foliation} of the quadratic differential $\Delta(x,\epsilon)$ is formed by the
level curves of \eqref{eq:expasymptotic}:  those along which $\Re\left(\pm\int\Delta^{\frac12}\right)$ stays constant.
They are the real time trajectories of the vector field $\pm\i\Delta^{-\frac12}$,
well defined up to orientation.	
They are also called the \emph{oscillation curves}, or \emph{Stokes curves}.
\end{enumerate}
\end{definition}

More generally, we will consider also the rotations
\begin{equation}\label{eq:steep} 
	\e^{-2\i\vartheta}\Delta\qquad \text{for an angle}\quad \vartheta\in \ ]\!-\tfrac{\pi}{2},\tfrac{\pi}{2}[,
\end{equation}
and their associated horizontal foliations, which correspond to the real-time trajectories of  $\pm\e^{\i\vartheta}\Delta^{-\frac12}$.
Along these curves there is a dominance relation between the exponentials \eqref{eq:expasymptotic},
allowing for a filtration on the solution space.
The existence of such filtrations is well established in the non-parametric case, and has been proven for unfoldings of non-resonant irregular singularities by Hurtubise, Lambert \& Rousseau \cite{Hurtubise-Lambert-Rousseau}, using a theorem of Levinson.

\begin{theorem}[Subdominant solutions]\label{prop:subdominant}
Let $\Delta$ be \eqref{eq:quaddiffomega}, $m<2k$.
	For any fixed small parameter $\epsilon$ and an angle $\vartheta\in \ ]\!-\frac{\pi}{2},\frac{\pi}{2}[$, and a horizontal trajectory of $\e^{-2\i\vartheta}\Delta$ tending to an equilibrium point $x=a_i(\epsilon)$ of $\Delta$ (pole of order $\geq 2$), there is a unique 1-dimensional subspace of the solution space of the parametric system \eqref{eq:unfoldedsystem}, called the space of \emph{subdominant solutions}, consisting of those solutions that have vanishing limit at the singularity. 	
	
	More precisely, let $\bt(x,\epsilon)$ be a branch of a determination of $\pm\int \Delta^{-\frac12}$, such that $\e^{-\bt(x,\epsilon)}\to 0$ when $x\to a_i(\epsilon)$ along the trajectory:
	\begin{itemize}
		\item If $a_i(\epsilon)$ is either Fuchsian or irregular singularity, then there is a unique normalized subdominant solution 
		\[y(x,\epsilon)\sim \e^{-\bt(x,\epsilon)}\left(\begin{smallmatrix} \tilde Q(x,\epsilon)^{-\frac14}\big(1+o(1)\big) \\[3pt] \tilde Q(x,\epsilon)^{\frac14}\big(1+o(1)\big) \end{smallmatrix}\right).\]
		\item If $a_i(\epsilon)$ is non-Fuchsia regular singularity (i.e. double pole of $\Delta$ which is a zero of multiplicity $m_i>0$ of $\tilde Q$), then there is a unique normalized subdominant solution 
		\[y(x,\epsilon)\sim \e^{\mp\int\sqrt{\Delta(x,\epsilon)+\left(\frac{m_i\d x}{4(x-a_i(\epsilon))}\right)^2}}\left(\begin{smallmatrix} \tilde Q(x,\epsilon)^{-\frac14}\big(1+o(1)\big) \\[3pt] \tilde Q(x,\epsilon)^{\frac14}\big(1+o(1)\big) \end{smallmatrix}\right).\]
	\end{itemize}
	This solution does not depend on the trajectory or the rotation angle $\vartheta$ as long as they are varied continuously. 
	Moreover the solution $y(x,\epsilon)$ depends analytically on $\epsilon$ as long as the trajectory varies continuously with the parameter $\epsilon$ .
\end{theorem}

\begin{proof}
The existence and uniqueness of such normalized subdominant solutions follows for each fixed $\epsilon$ from the usual theorems on existence of local normalizing transformations at the point $x=a_i$, such as the Borel (multi)-summability of the Hukuhara--Levelt--Turrittin formal fundamental solutions \cite[Theorem 20.13]{Ilashenko-Yakovenko}.
The local analytic dependence on $\epsilon$ is quite clear as long as the type of the singularity does not change, however it is less clear that it is also analytic at the values for which $a_i$ is a resonant Fuchsian singularity, and that it passes well to the limit when $a_i$ becomes irregular.
We will follow the method of proof of \cite{Hurtubise-Lambert-Rousseau}, using 
a parametric version of the Levinson's theorem \cite[Theorem~5.3]{Hurtubise-Lambert-Rousseau} (see \cite[\S3, Theorem 8.1]{Coddington-Levinson} for the original  non-parametric version):
	
\begin{itemize}[leftmargin=\parindent, itemindent=0em]
\item[] \textit{\textbf{Levinson's theorem.}
	Consider a system of linear differential equations on the real line of the form
	\[\frac{\d \tilde y}{\d\bt}=\big[\Lambda(\bt,\epsilon)+B(\bt,\epsilon)\big]\tilde y,\]
	where $\Lambda(\bt,\epsilon)$ is diagonal, its limit $\Lambda(+\infty,\epsilon)=\lim_{\bt\to+\infty}\Lambda(\bt,\epsilon)$ has eigenvalues with distinct real parts, and assume that
	\begin{equation}\label{eq:RC-Levinson1}
	\int_{0}^{+\infty}\Big|\frac{\d }{\d\bt}\Lambda(\bt,\epsilon)\Big|\d\bt<\infty,\qquad \int_{0}^{+\infty}\big|B(\bt,\epsilon)\big|\d\bt<\infty.
	\end{equation}
	Then for each eigenvalue $\lambda(\bt,\epsilon)$ of $\Lambda(\bt,\epsilon)$ and  an eigenvector $v_\lambda(+\infty,\epsilon)$ of $\Lambda(+\infty,\epsilon)$ associated to  $\lambda(+\infty,\epsilon)$
	there exists $t_0>0$ and a solution $\tilde y=\phi_\lambda(\bt,\epsilon)$ on $]t_0,+\infty[$ such that
	\begin{equation}\label{eq:RC-Levinson2}
	\lim_{\bt\to+\infty} \phi_\lambda(\bt,\epsilon)\cdot \exp\Big(-\int_{t_0}^\bt \lambda(\bt,\epsilon)\d\bt\Big)=v_\lambda(+\infty,\epsilon),
	\end{equation}
	If the system and the eigenvector $v_\lambda(+\infty,\epsilon)$ depend continuously (resp. analytically) on a parameter $\epsilon$ over compact sets in the $\bt$-space, with the integrals in \eqref{eq:RC-Levinson1}
	uniformly bounded, then the solution can be chosen depending continuously (resp. analytically) on $\epsilon$.
	}
\end{itemize}

Applying the gauge transformation $y=\left(\begin{smallmatrix}Q(x,\epsilon)^{-\frac{1}{4}}\!\!\!&0\\0&\!\!\!Q(x,\epsilon)^{\frac{1}{4}}\end{smallmatrix}\right)
\left(\begin{smallmatrix}1&1\\[5pt]1&-1\end{smallmatrix}\right)\tilde y$ to \eqref{eq:unfoldedQ}, then $\tilde y$ satisfies
\[\tfrac{P}{\sqrt Q}\tdd{x}\tilde y=\left[\begin{pmatrix}1&0\\0&-1\end{pmatrix}+\begin{pmatrix}0&b(x,\epsilon)\\b(x,\epsilon)&0\end{pmatrix}\right]\tilde y,\qquad b(x,\epsilon)=\tfrac{P}{4\sqrt{Q}}\tdd{x}\log Q.\]
An additional gauge transformation  
$\tilde y=\tfrac12\begin{pmatrix}1\!+\!\sqrt{1\!+\!b^2}\hskip-6pt& -b\\b& \hskip-6pt 1\!+\!\sqrt{1\!+\!b^2}\end{pmatrix}\tilde{\tilde y}$ 
leads to a system  
\begin{equation}\label{eq:tty}
	\tfrac{P}{\sqrt Q}\tdd{x}\tilde{\tilde y}=\left[\begin{pmatrix}\sqrt{1+b^2}&0\\0&-\sqrt{1+b^2}\end{pmatrix}-\tfrac{P}{\sqrt Q}\tfrac{\d b}{\d x}\begin{pmatrix}\frac{b}{1+b^2}&-\frac{1}{2\sqrt{1+b^2}^3\vphantom{\big|}}\\ \tfrac{1}{2\sqrt{1+b^2}^3\vphantom{\big|}}&\frac{b}{1+b^2}\end{pmatrix}\right]\tilde{\tilde y}.
\end{equation}
Here we take the branch of $\sqrt{1+b^2}$ with value $1$ for $b=0$ on
$b^2\in\C\smallsetminus]-\infty,-1]$.
We can now apply the Levinson's theorem with $\tilde\Lambda=\sqrt{1+b^2}\begin{psmallmatrix}1&0\\[3pt]0&-1\end{psmallmatrix}$.

The limit point $a_i(\epsilon)$ of the trajectory is a pole of $\Delta$ of order $2k_i-m_i+2\geq 2$, where $k_i+1$ and $m_i$ are the multiplicities of $a_i$ as a zero of $P$ and $Q$.
If it is an irregular singularity, $2k_i-m_i+2>0$, then $\int\frac{\sqrt{Q+Qb^2}-\sqrt{\tilde Q}}{P}\d x$ is bounded near $x=a_i$, so one 
replace $\e^{\pm\int \frac{\sqrt{Q+Qb^2}}{P}\d x}$ by $\e^{pm\int \frac{\sqrt{\tilde Q+Qb^2}}{P}\d x}=\e^{\bt}$.
Moreover one can deform the trajectory of $\e^{-2\i\vartheta}Q(x,\epsilon)\left(\tfrac{\d x}{P(x,\epsilon)}\right)$ to a trajectory of $\e^{-2\i\vartheta}\Delta$ with the same asymptotic tangent.
In fact the two differentials are locally tangent-to-identity conformally equivalent near $a_i(\epsilon)$.

If $a_i(\epsilon)$ is a regular singularity, then $\frac{\sqrt{Q+Qb^2}}{P}\sim\sqrt{\mu_i(\epsilon)+\left(\frac{m_i}{4}\right)^2}$, where $\mu_i(\epsilon)=\res_{x=a_i}^2\Delta(x,\epsilon)$. The singularity must be of focus type (see \S\,\ref{sec:parametricfoliation}) 
meaning that $\mu_i(\epsilon)\notin \e^{2\i\vartheta}\R_{<0}$.
Since $|\mu_i(\epsilon)|\to \infty$ as $\epsilon\to 0$, one can assume that $|\epsilon|$ is small enough that the closed segment
$[\mu_i(\epsilon),\  \mu_i(\epsilon)+\left(\frac{m_i}{4}\right)^2]$ doesn't cross the ray $\e^{2\i\vartheta}\R_{<0}$,
which means that $\e^{\pm\int\sqrt{Q+\alpha Qb^2}\frac{\d x}{P}}\to 0$ has the same limit $0$ or $\infty$ along the trajectory as $\e^{-\i \vartheta}\bt\to+\infty$ for all $\alpha\in[0,1]$.
\end{proof}

\begin{remark}
One can treat both cases in Theorem~\ref{prop:subdominant} in a unified way by replacing the quadratic differential	$\Delta$ by
\[\Delta(x,\epsilon)+\left(\tfrac14 \d\log R(x,\epsilon)\right)^2,\]
where $R(x,\epsilon)$ is the Weierstrass polynomial that is the greatest common divisor of $P(x,\epsilon)^2$ and $\tilde Q(x,\epsilon)$. 
\end{remark}

Taking a \emph{complete horizontal trajectory} of $\e^{-2\i\vartheta}\Delta$ for some $|\vartheta|<\frac\pi2$, one obtains a pair of subdominant solutions:
one associated with the positive end of the trajectory ($\Re\bt\to+\infty$) and one with the negative end ($\Re\bt\to-\infty$).
If they are linearly independent, they form a \emph{mixed solution basis}.
Mixed solution bases originated in the study of hypergeometric equations and have been investigated in the context of confluence by various authors, including Sch\"afke \cite{Schaefke}, Ramis \cite{Ramis}, Zhang \cite{Zhang}, Duval \cite{Duval}, Glutsyuk \cite{Gl1, Gl2}, Parise~\cite{Parise}
and Lambert \& Rouseau \cite{LR1, Lambert-Rousseau, Hurtubise-Lambert-Rousseau}.

In the case $m=0$ the linear independence can be established by a deformation argument as in \cite{Hurtubise-Lambert-Rousseau}: their limit when $\epsilon\to 0$  is one of the standard sectorial solution bases at the irregular singularity.
Consequently, the linear independence persists for $\epsilon$within a suitably defined ``sectorial'' domain of sufficiently small radius.
However, when $m>0$, some of these mixed solution pairs can fail to have a well-defined limit as
$\epsilon\to 0$. This occurs because the natural domain of their definition shrinks and disappears.
As a result, the deformation argument cannot be applied in general.

\begin{theorem}[Mixed solution bases]\label{proposition:mixedbasis}
	Assume $m<2k$. Let $\sigma$ be complete horizontal trajectory of $\e^{-2\i\vartheta}\Delta$ for some $|\vartheta|<\frac\pi2$ whose closure lies inside the domain
	\begin{equation}\label{eq:mixeddomain}
		\left|\tfrac{P(x,\epsilon)}{4\sqrt{\tilde Q(x,\epsilon)}}\tdd{x}\log \tilde Q(x,\epsilon)\right|<\tfrac12,\qquad \left|\tfrac{Q}{\tilde Q}-1\right|<\tfrac14.
	\end{equation}
	Then the associated pair of sub-dominant solutions is linearly independent, hence  forms a \emph{mixed solution basis}.
\end{theorem}

\begin{proof}
First let us remark that if $a_i$ is a zero of $P$ of order $k_i+1$ which is also a zero of $Q$ order $m_i$, then if the assumption of Theorem~\ref{prop:subdominant} is violated: $2k_i+2-m_i=2$ and $m_i>0$, then $b(a_i,\epsilon)=\frac{m_i}{4}\geq 2$ so in this case $a_i$ is not in the domain considered.  

Choose a determination of $\bt(x,\epsilon)=\pm\int\Delta^{\frac12}$, and let $y(x,\epsilon)=\begin{psmallmatrix}y_1(x,\epsilon)\\y_2(x,\epsilon) \end{psmallmatrix}$ be the subdominant solution along the negative end of the trajectory, $\e^{-\i \vartheta}\bt\to-\infty$.
We will show that $|y_1(x,\epsilon)|\to\infty$ the other end of the trajectory $\e^{-\i \vartheta}\bt\to+\infty$, meaning that is dominant there, and therefore it is not subdominant.
To simplify, let us assume that the angle $\vartheta=0$, i.e. that the trajectory is horizontal for $\Delta$.
As $y_1$ satisfies $\big(P\tdd{x}\big)^2y_1=Qy_1$, we have
\[\left(\tdd{\bt}\right)^2y_1=\tfrac{1}{\tilde Q}\big(P\tdd{x}\big)^2y_1-\tfrac{P^2}{2\tilde Q}\tdd{x}\log\tilde Q\cdot\tdd{x}y_1
=\tfrac{Q}{\tilde Q}y_1-2b\tdd{\bt}y_1,\]
where $\tilde b(x,\epsilon)=\frac14\tdd{\bt}\log\tilde Q(x,\epsilon)$.
Hence
\begin{align*}
	\left(\tdd{\Re \bt}\right)^2|y_1|^2&=2\Re\left(\frac{\tdd{\bt}^2y_1}{y_1}\right)|y_1|^2+2\left|\tdd{\bt}y_1\right|^2\\
	&=2\Re\left(\tfrac{Q}{\tilde Q}\right)|y_1|^2-4\Re\left(\tilde b\frac{\tdd{\bt}y_1}{y_1}\right)|y_1|^2+2\left|\tdd{\bt}y_1\right|^2.
\end{align*}
By assumption $|\tilde b|<\frac12$, hence $4\Re\left(\tilde b\frac{\tdd{\bt}y_1}{y_1}\right)|y_1|^2<2\left|\tdd{\bt}y_1\right|\cdot |y_1|$, and
$\Re\left(\tfrac{Q}{\tilde Q}\right)>\tfrac34$, so
\[\left(\tdd{\Re \bt}\right)^2|y_1|^2> 
\tfrac12|y_1|^2+\left|\tdd{\bt}y_1\right|^2+\left(|y_1|-\left|\tdd{\bt}y_1\right|\right)^2>  \tfrac12|y_1|^2.\]

Since near $\bt=-\infty$, $y_1\sim Q^{\frac14}\e^\bt$, then $\tdd{\Re \bt}|y_1|^2\sim 2|y_1|^2(1+\Re b)>0$,
which means that $|y_1|^2>C\e^{\frac12\Re \bt}$ for some $C>0$, hence it is exponentially big when $\Re \bt\to+\infty$. 
\end{proof}


\subsection{Geometry of horizontal foliations}\label{sec:foliation}

To understand the natural domain on which mixed solution bases exist, we need to understand the geometry of horizontal foliations in a rotating family of quadratic differentials $\e^{-2\i\vartheta}\Delta$ \eqref{eq:quaddiffomega}.
We begin by studying meromorphic quadratic differentials defined globally on $\CP^1$, or more generally on a compact Riemann surface.
Afterwards, we relativize the concepts to more general semi-local domains. 

\bigskip

Let $\Delta$ be a meromorphic quadratic differential on  a compact Riemann surface $\sX$, we denote $\Delta^{-\frac12}$ the dual vector field defined on a two-sheeted ramified covering of $\sX$.
Under certain generic assumptions, much of the theory of  real dynamics of rational vector fields on
 $\CP^1$ (see, for example, \cite{Benziger,Brickman-Thomas,Douady-Estrada-Sentenac,Hajek,Jenkins-Spencer,Klimes-Rousseau2, Tomasini})
generalizes directly to quadratic differentials on compact Riemann surfaces. 
Recent years have seen a growing interest in the study of meromorphic quadratic differentials and their moduli spaces in works such as
\cite{Bridgeland-Smith,Haiden-Katzarkov-Kontsevich,Tahar1}.

Given a quadratic differential $\Delta$, then the various determinations of
\begin{equation}\label{eq:t}
	\bt=\pm\int\sqrt{\Delta}
\end{equation}
form a system of flat local coordinates on $\sX\smallsetminus\Crit(\Delta)$, giving it a structure of a \emph{half-translation surface}.
The horizontal foliation of $\e^{-2\i \vartheta}\Delta=(\e^{-\i \vartheta}\d\bt)^2$ corresponds to the foliation by lines parallel to $\e^{\i \vartheta}\R$ in the coordinate $\bt$. It's leaves are the real-time trajectories of  $\e^{\i \vartheta}\Delta^{-\frac12}$.

\begin{figure}[t]
	\centering
	\begin{subfigure}[t]{0.24\textwidth}
		\includegraphics [width=\textwidth]{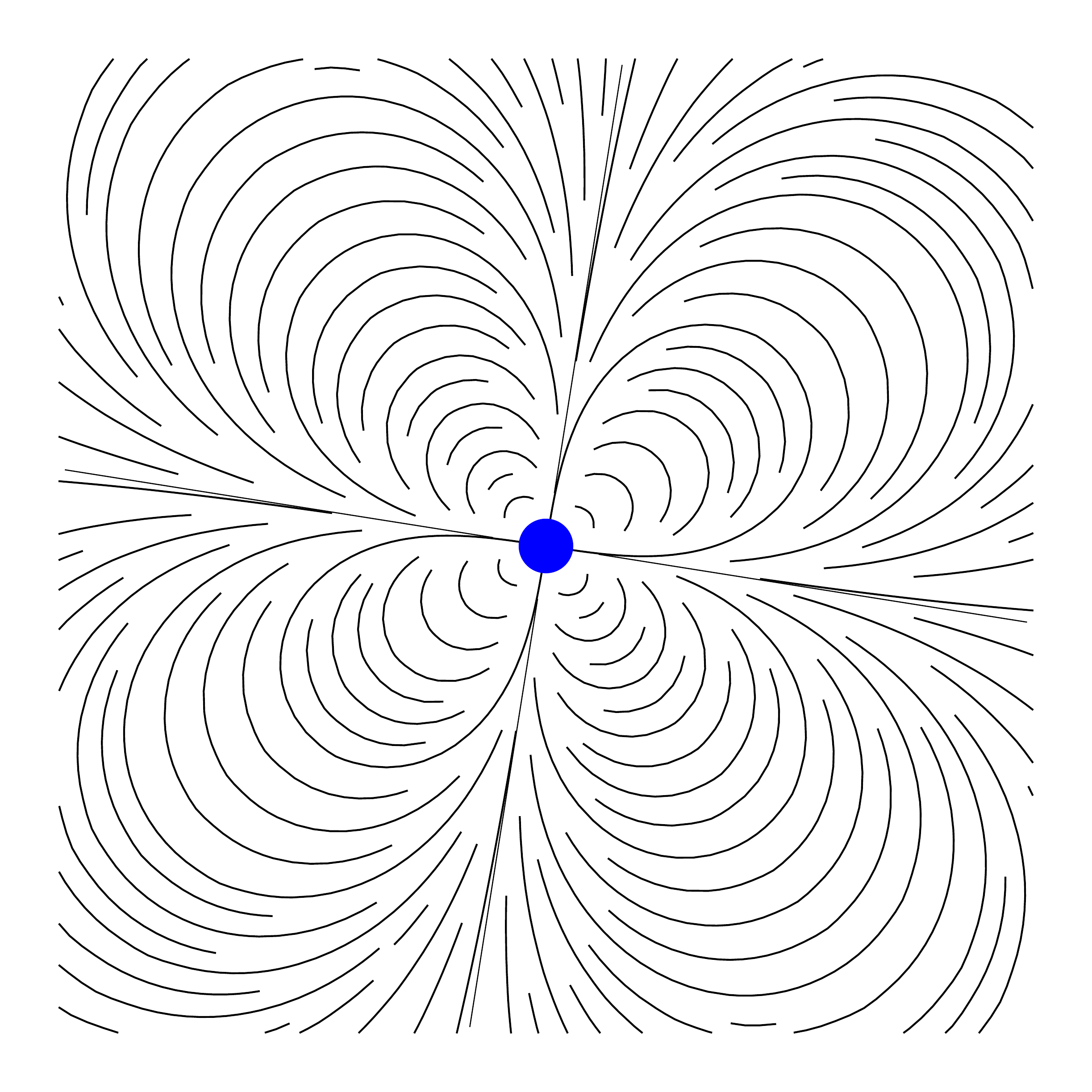}
		\caption{Parabolic}	
	\end{subfigure}
	\begin{subfigure}[t]{0.24\textwidth}
		\includegraphics [width=\textwidth]{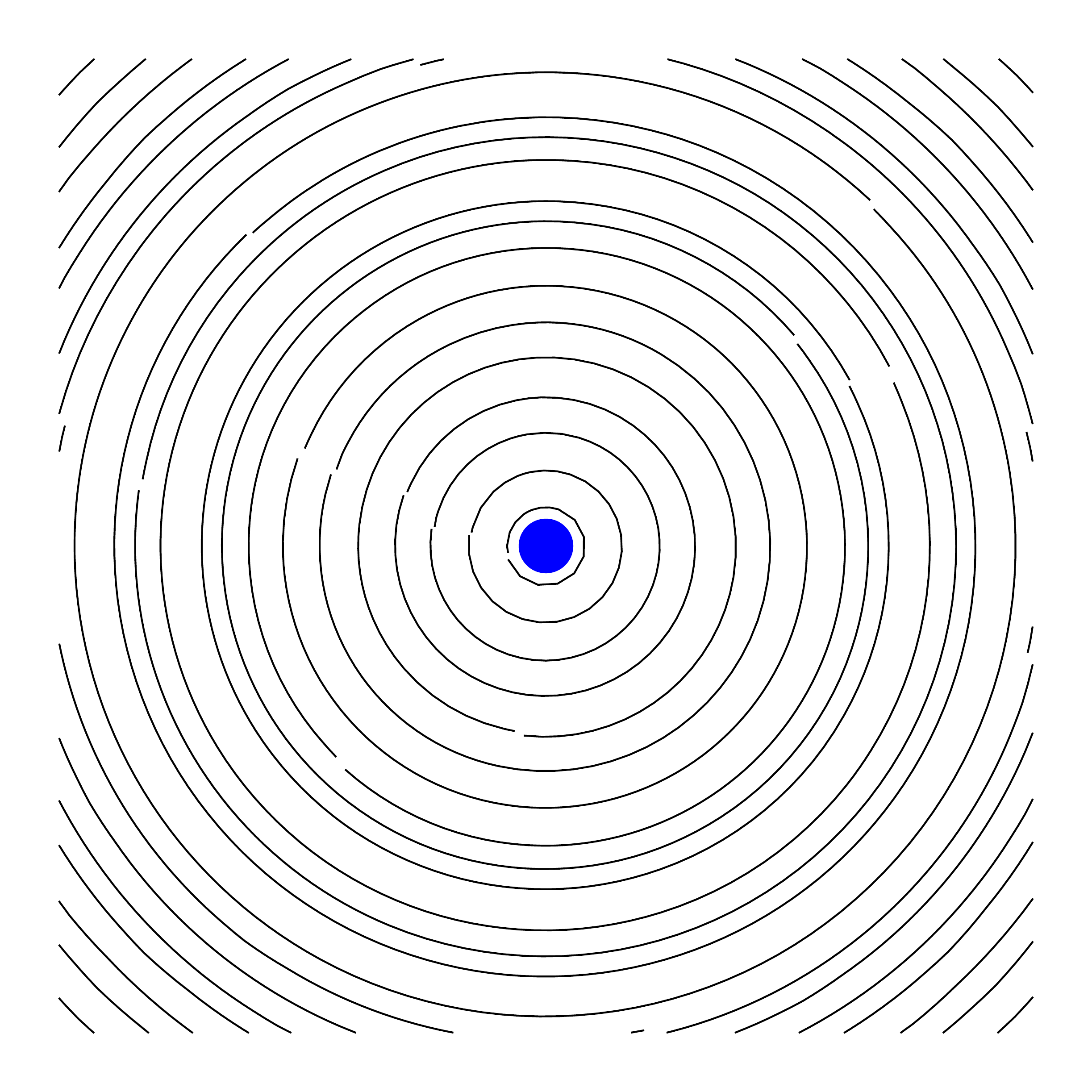}
		\caption{Center}	
	\end{subfigure}
	\begin{subfigure}[t]{0.24\textwidth}
		\includegraphics [width=\textwidth]{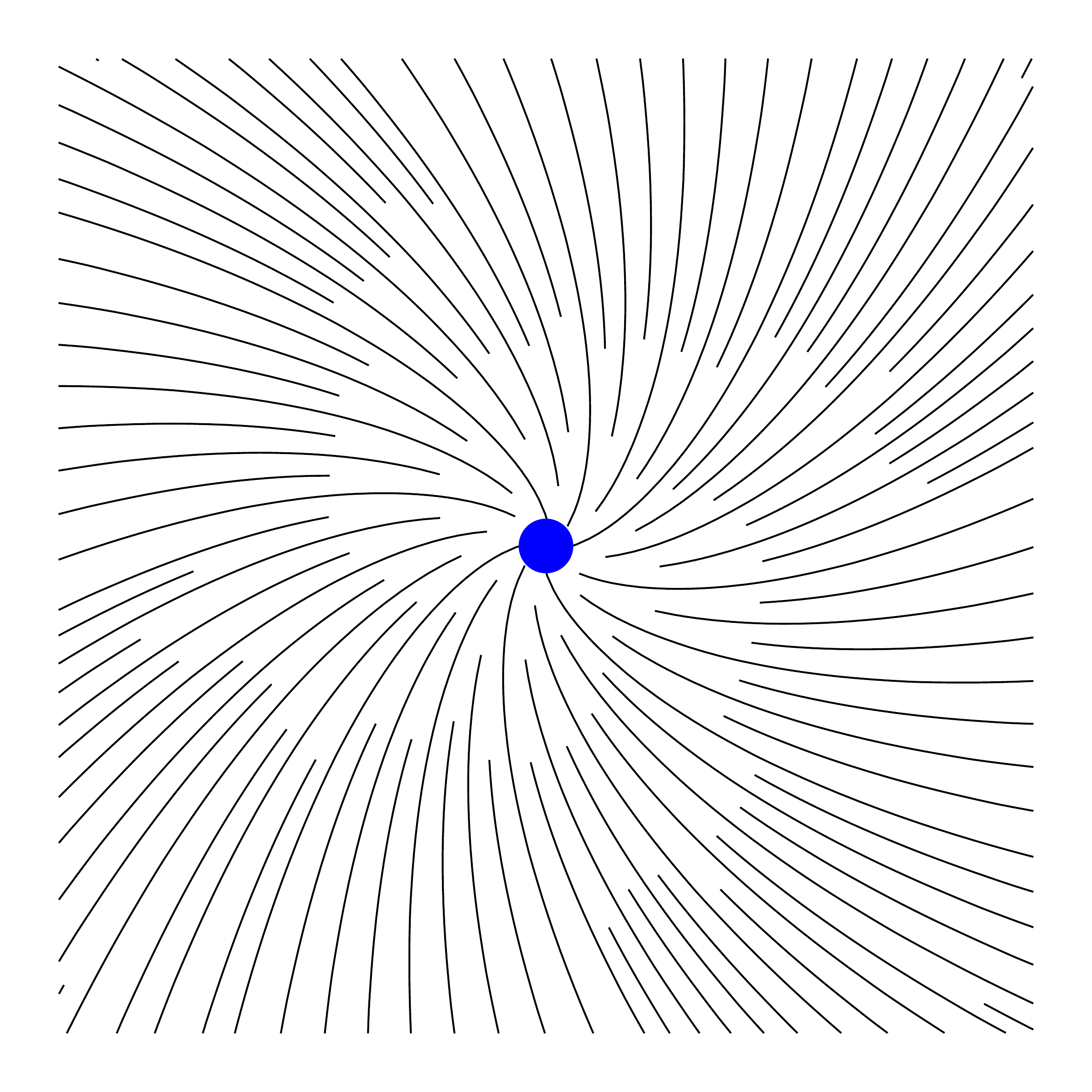}
		\caption{Focus}	
	\end{subfigure}
	\begin{subfigure}[t]{0.24 \textwidth}
		\includegraphics [width=\textwidth]{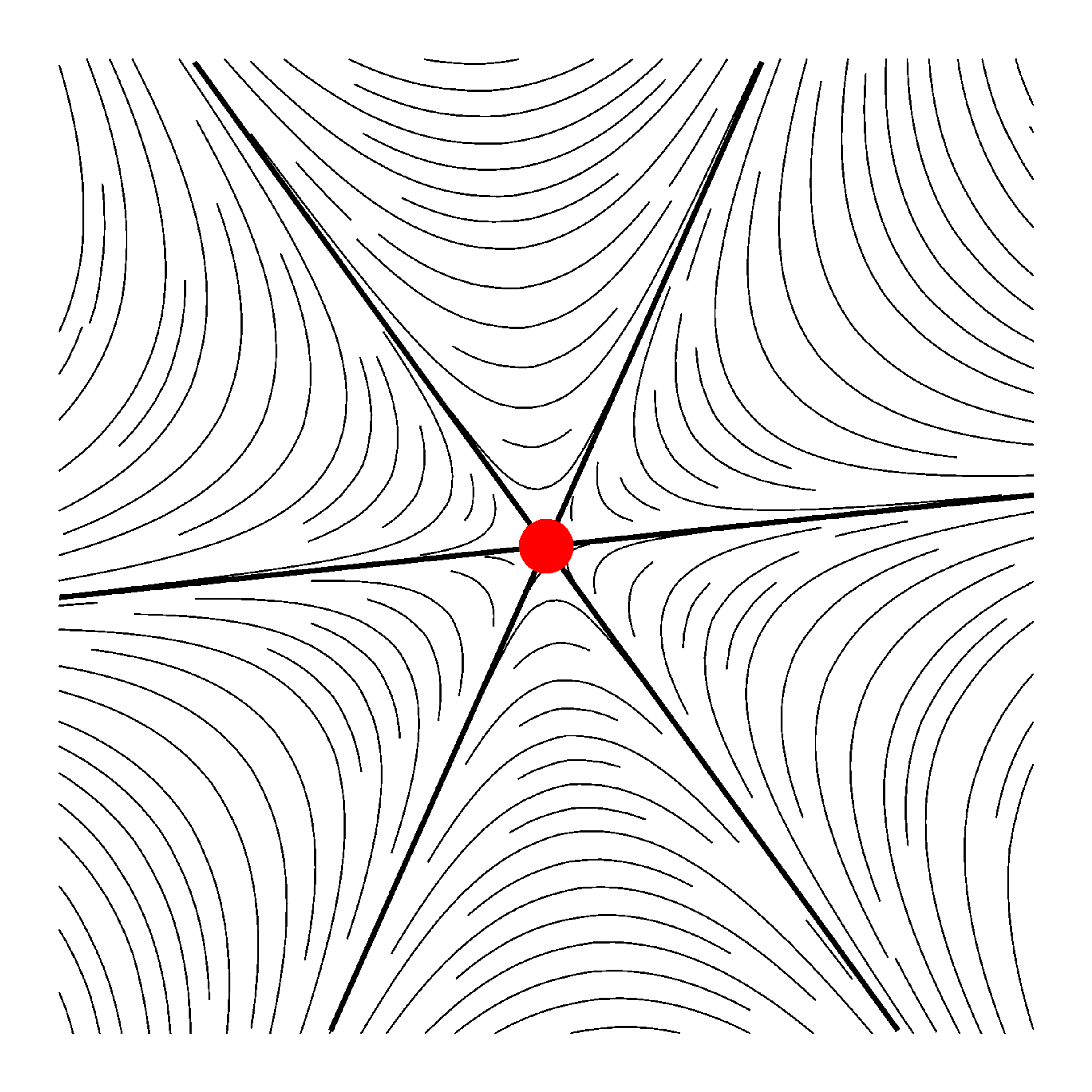}
		\caption{Saddle}	
	\end{subfigure}
	\caption{Critical points of a horizontal foliation.}
	\label{figure:critical}
\end{figure}

\subsubsection{Critical points}\label{sec:criticalpoints}
For a zero or a pole of a quadratic differential $\Delta$ of polar order $2\nu+2$ we call $\nu$ the \emph{rank} of the point (i.e. $\nu<-1$ for a zero of $\Delta$).
It corresponds to the Katz rank of the associated differential system.
The \emph{critical points} are of the following kinds \cite[\S7]{Strebel}:
\begin{enumerate}[leftmargin=2\parindent]
\item \emph{parabolic equilibrium} if rank $\nu>0$ (i.e. pole of order $>2$): the horizontal foliation has $2\nu$ sepal zones, separated by $2\nu$ \emph{separating directions} (asymptotic directions).
\item \emph{simple equilibrium} if rank $\nu=0$  (i.e. pole of order $2$):
		\begin{itemize}
		\item \emph{center} if $\mu=\res^2\Delta<0$: trajectories are periodic of the same period $\pm 2\pi\i \sqrt{\mu}$. 
		\item \emph{focus} if $\mu=\res^2\Delta\notin\R_{<0}$: trajectories are conformally equivalent to either straight rays or logarithmic spirals.
		\end{itemize}
\item \emph{saddle point} if rank $\nu<0$, this includes the zeros and the simple poles of $\Delta$ and possibly a finite number of marked regular points.
The local dynamics is of hyperbolic type with $2|\nu|$ \emph{hyperbolic sectors} separated by the same number of \emph{separatrices}. 
\end{enumerate}
The time $\bt$ \eqref{eq:t} it takes to arrive at an equilibrium from any regular point is infinite, while the time it takes to arrive at a saddle point is finite. 
We denote
\[\Equilib(\Delta)=\{\text{equilibriums of } \Delta\},\qquad \Saddle(\Delta)=\{\text{saddles of } \Delta\},\]
and 
\[\Crit(\Delta)=\Equilib(\Delta)\cup\Saddle(\Delta).\]
\emph{\textbf{We shall assume that }}
\begin{equation}\label{eq:assumption}
	\Equilib(\Delta)\neq\emptyset,\qquad\Saddle(\Delta)\neq\emptyset.
\end{equation}
The second condition can always be achieved by adding a marked regular point to the set of saddles (if $\sX=\CP^1$ then the only quadratic differentials without saddles are those with either two  equilibria of rank $0$ or with a single equilibrium of rank $1$).

\begin{figure}[t]
	\centering
	\begin{subfigure}[t]{0.49\textwidth}
		\includegraphics [width=\textwidth]{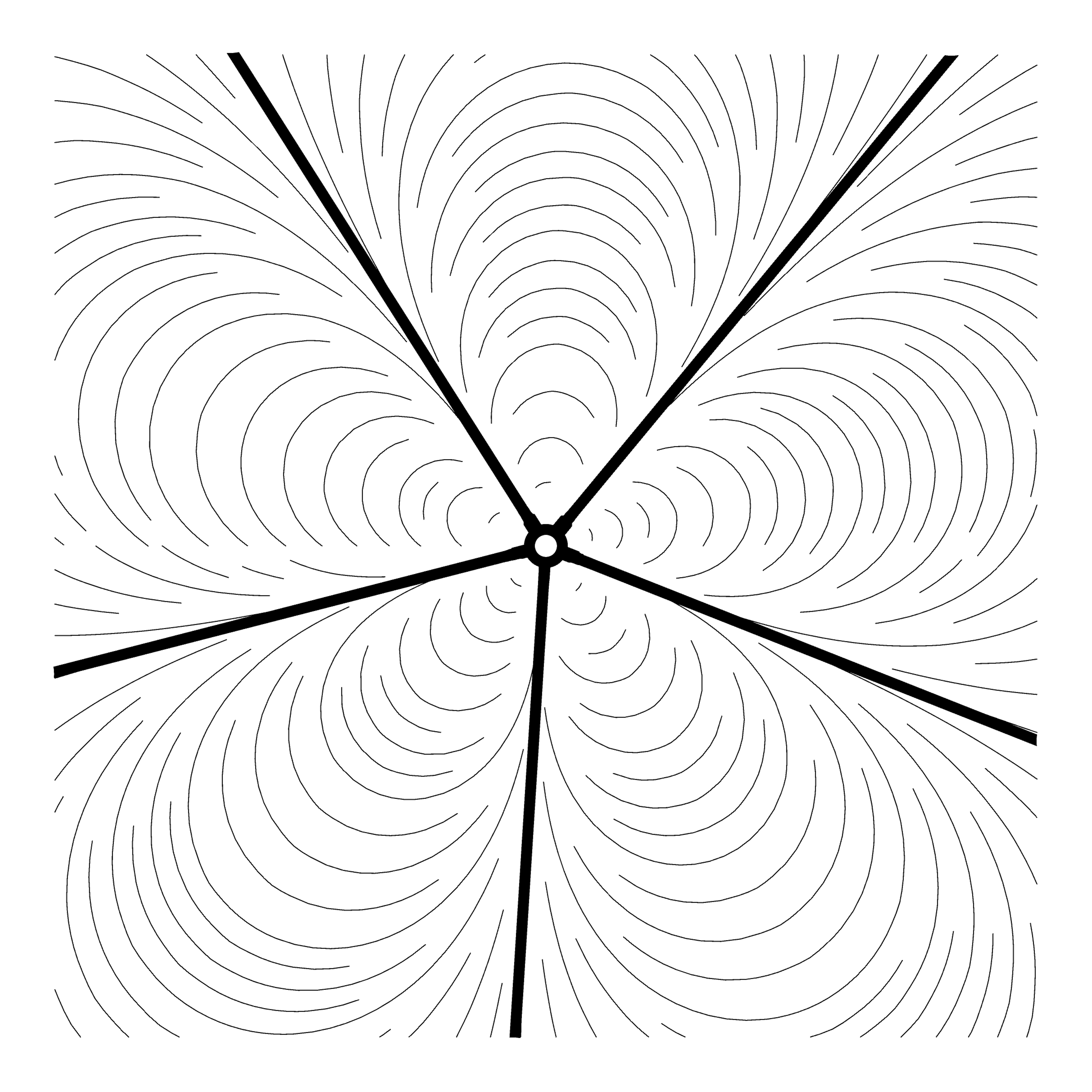}
		\caption{$\Delta=\frac{c}{x^7}(\d x)^2$}	
	\end{subfigure}
	\begin{subfigure}[t]{0.49\textwidth}
		\includegraphics [width=\textwidth]{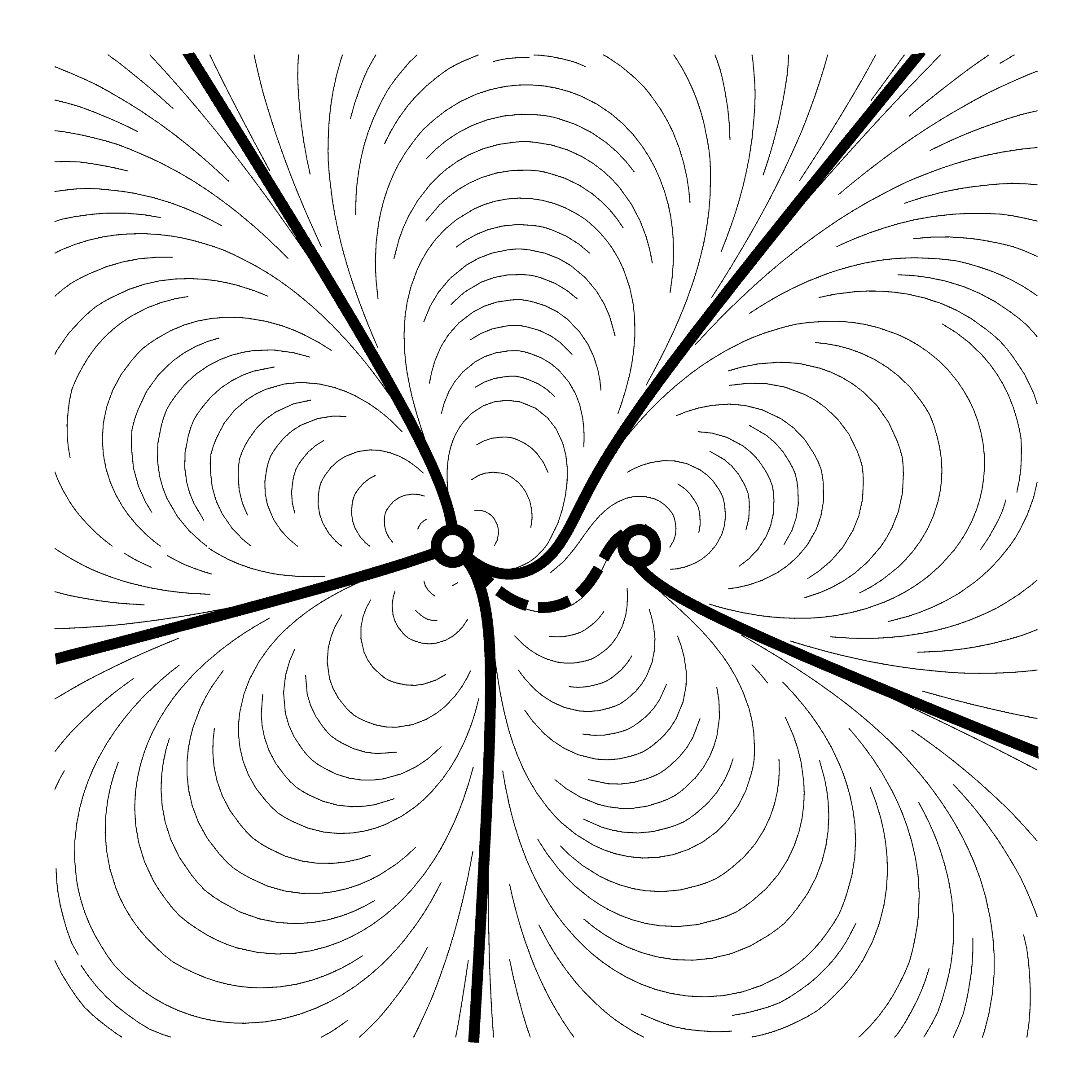}
		\caption{$\Delta=\frac{c}{(x-\epsilon)^2(x+\epsilon)^5}(\d x)^2$}	
	\end{subfigure}
	
	\begin{subfigure}[t]{0.49\textwidth}
		\includegraphics [width=\textwidth]{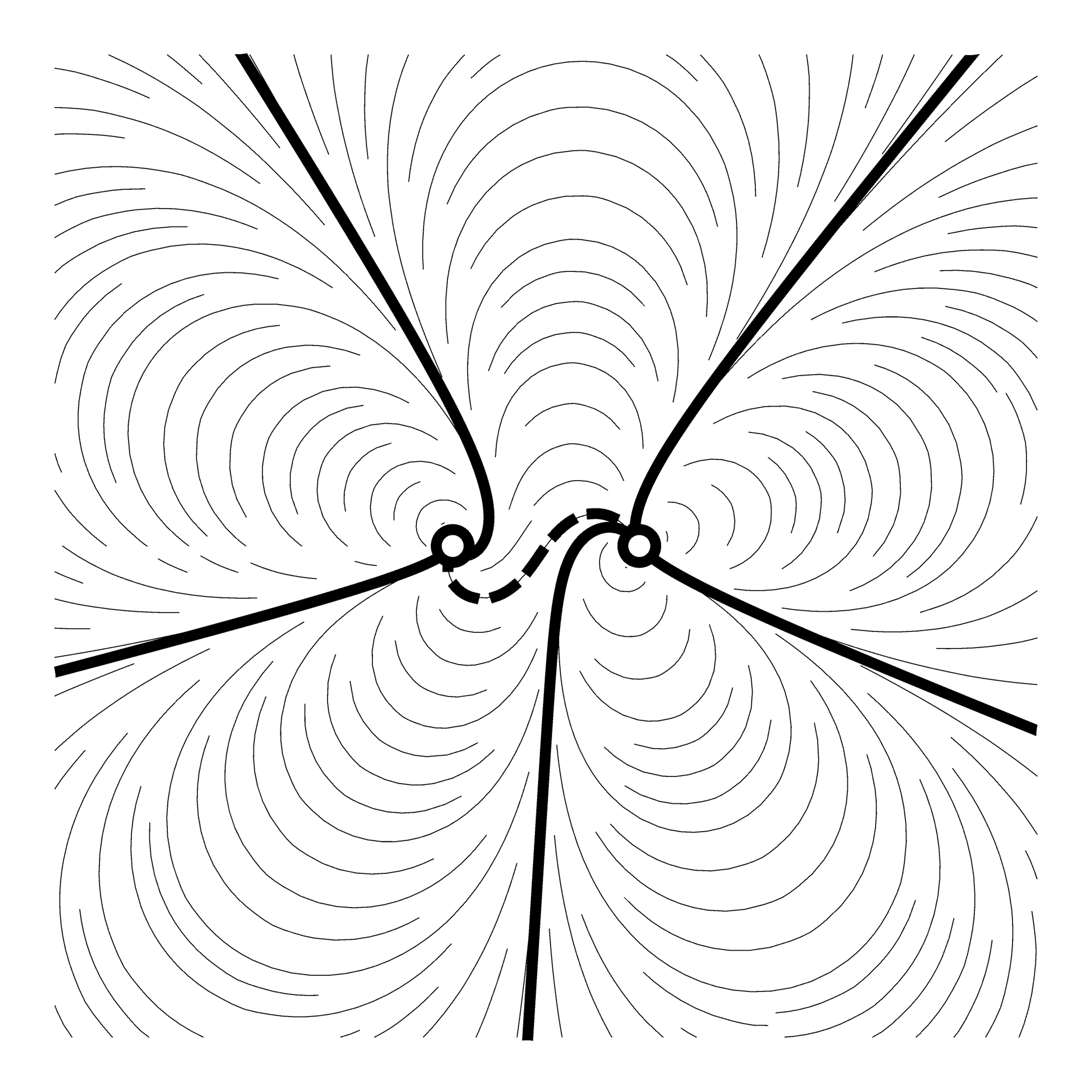}
		\caption{$\Delta=\frac{c}{(x-\epsilon)^3(x+\epsilon)^4}(\d x)^2$}	
	\end{subfigure}
	\begin{subfigure}[t]{0.49\textwidth}
		\includegraphics [width=\textwidth]{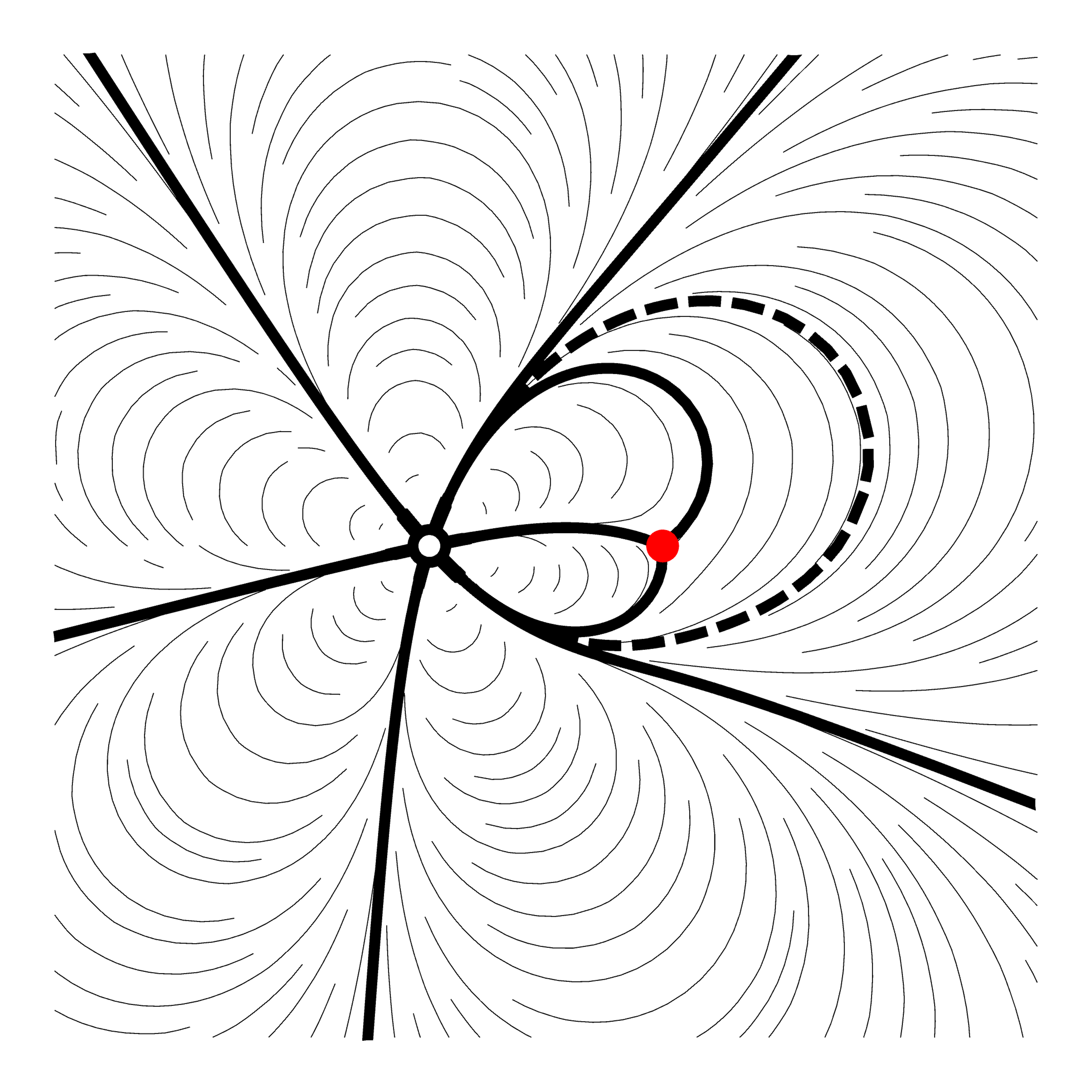}
		\caption{$\Delta=\frac{c(x-\epsilon)}{(x+\epsilon)^8}(\d x)^2$}	
	\end{subfigure}
	\caption{Examples of rotationally stable horizontal foliations: separatrices in thick, gates dashed. In all 4 examples the infinity is a saddle point of rank $-\frac52$.}
	\label{figure:a}
\end{figure}

\subsubsection{Zone decomposition}
The (maximal) trajectories of a quadratic differential $\e^{-2\i\vartheta}\Delta$ are of the following kinds \cite[\S11]{Strebel}:
\begin{enumerate}[leftmargin=2\parindent]
	\item closed: \emph{periodic trajectories},
	\item finite: \emph{saddle trajectories}, i.e. homoclinic and heteroclinic separatrices,
	\item semi-infinite: 
	\begin{itemize}
		\item \emph{converging separatrices}: on one side have a finite limit at a saddle point and on the other side an infinite limit at an equilibrium,
		\item \emph{recurrent separatrices}:  on one side have a finite limit at a saddle point and on the other side their limit set equals to their closure and has a non-empty interior,
	\end{itemize}
	\item infinite: 
	\begin{itemize}
		\item \emph{converging trajectories}: on both sides have infinite limits at equilibrium points (either 2 different, or the same parabolic one),
		\item \emph{recurrent trajectories}: on both sides their limit sets are the same, equal to their closure, and have a non-empty interior.  
	\end{itemize}
\end{enumerate}

Each focus equilibrium is the limit point of at least one convergent separatrix, and each separating direction of a parabolic equilibrium is tangent to at least one convergent separatrix.

\begin{definition}[Zones]
	The \emph{separating graph } is the closure of union of all the separatrices of saddle points $\Saddle(\Delta)$.
	The connected components in $\sX$ of the complement of the separating graph  and of the set of equilibria $\Equilib(\Delta)$ are called \emph{zones of $\e^{-2\i\vartheta}\Delta$ in $\sX$}. 	
They can be of the following kinds according to their form in the flat coordinate $\bt$ \cite[\S11]{Strebel}:
\begin{enumerate}[leftmargin=2\parindent]
	\item \emph{sepal zone}: all trajectories have their $\alpha$-limit and $\omega$-limit   at the same multiple singular point (parabolic equilibrium).
	On the translation surface it takes the form of a half-plane $\e^{\i \vartheta}\H^\pm$.  
	\item \emph{$\alpha\omega$-zone}: all trajectories have the same $\alpha$-limit and the same $\omega$-limit  which are two different equilibria (simple or multiple). On the translation surface it takes the form of an open strip parallel to $\e^{\i \vartheta}\R$.
	\item \emph{center zone}: all trajectories are periodic filling a neighborhood of a center equilibrium. 
	On the translation surface it takes the form of a semi-infinite cylinder.
	\item \emph{annular zone}: all trajectories are periodic. 
	On the translation surface it takes the form of a  finite cylinder.
\end{enumerate}
\end{definition}

\subsubsection{Saddle connections, periodic/parabolic domains of equilibria, and core}

Instead of looking at the horizontal foliation of $\e^{-2\i\vartheta}\Delta$ for a fixed rotation angle $\vartheta\in\R$ it is also useful to  consider all such rotations at once.

\begin{definition}[Saddle connections and period map]
	A \emph{saddle connection} of $\Delta$ is a saddle trajectory $\gamma$ of $\e^{-2\i\vartheta}\Delta$ for any rotation $\vartheta\in\R$.
	Its \emph{square period} is the square of its $\bt$-length:
	\begin{equation}\label{eq:period}
		\gamma\mapsto\left(\int_{\gamma}\Delta^{\frac12}\right)^2.
	\end{equation}
\end{definition}

\begin{definition}[Periodic/parabolic domains of equilibria and the core]\label{def:periodicparabolic}
	A \emph{periodic domain} of a simple equilibrium, resp. \emph{parabolic domain} of a multiple equilibrium,
	is the connected component of the equilibrium in the complement of the union of all saddle connections (over all rotations $\vartheta$).	
	Namely,
	\begin{itemize}[leftmargin=\parindent]
		\item for a simple equilibrium $a$ with $\mu=\res_{x=a}^2\Delta$, the periodic domain is the center zone of $a$ of $-\frac{|\mu|}{\mu}\Delta$,
		\item for a multiple equilibrium $a$, the parabolic domain is the union of all sepal zones of $a$ over all rotations $\e^{-2\i\vartheta}\Delta$, $\vartheta\in\R$. 	
	\end{itemize}	
	By the construction, it is the maximal domain such that for any rotation	$\vartheta$ any trajectory of $\e^{-2\i\vartheta}\Delta$ that enters the domain will not be able to leave and will asymptotically tend to the equilibrium. See \cite{Klimes-Rousseau2}.
	
	The periodic and parabolic zones of different equilibria are disjoint. Their complement is called the \emph{core} \cite{Haiden-Katzarkov-Kontsevich, Tahar1, Tahar2}
	(its image in the $\bt$-coordinate is also called the \emph{periodgon} \cite{Klimes-Rousseau2}). 
\end{definition}

\subsubsection{Rotational stability}

\begin{definition}[Rotational stability]
Assume \eqref{eq:assumption}. The horizontal foliation of $\e^{-2\i\vartheta}\Delta$ is called \emph{rotationally stable} if it contains no saddle trajectories.
\end{definition}

\begin{proposition}\label{prop:rotstability}
Assume \eqref{eq:assumption}.
\begin{enumerate}[leftmargin=2\parindent]
	\item The set of directions $\vartheta\in \R/\pi\Z$ for which $\e^{-2\i\vartheta}\Delta$ is not rotationally stable is at most countable closed set with at most finitely many accumulation points.	
	\item When  $\e^{-2\i\vartheta}\Delta$ is rotationally stable there are no recurrent trajectories. 
\end{enumerate}
\end{proposition}

The first statement is proven by Tahar \cite{Tahar2} in much greater generality. In the case of rational vector fields it is due to Muciño-Raymundo and Valero \cite{Mucino-Valero}.  
The second statement is essentially due to Strebel: the boundary of the closure of a recurrent trajectory always consists of a union of saddle points and saddle trajectories \cite[\S11]{Strebel}, but by the assumption there are none. 
At the same time the boundary could not be empty since by the assumption \eqref{eq:assumption} there is at least one  equilibrium and it cannot be in the limit set of any recurrent trajectory.

\begin{figure}[t]
	\centering
	\begin{subfigure}[t]{0.49\textwidth}
		\includegraphics [width=\textwidth]{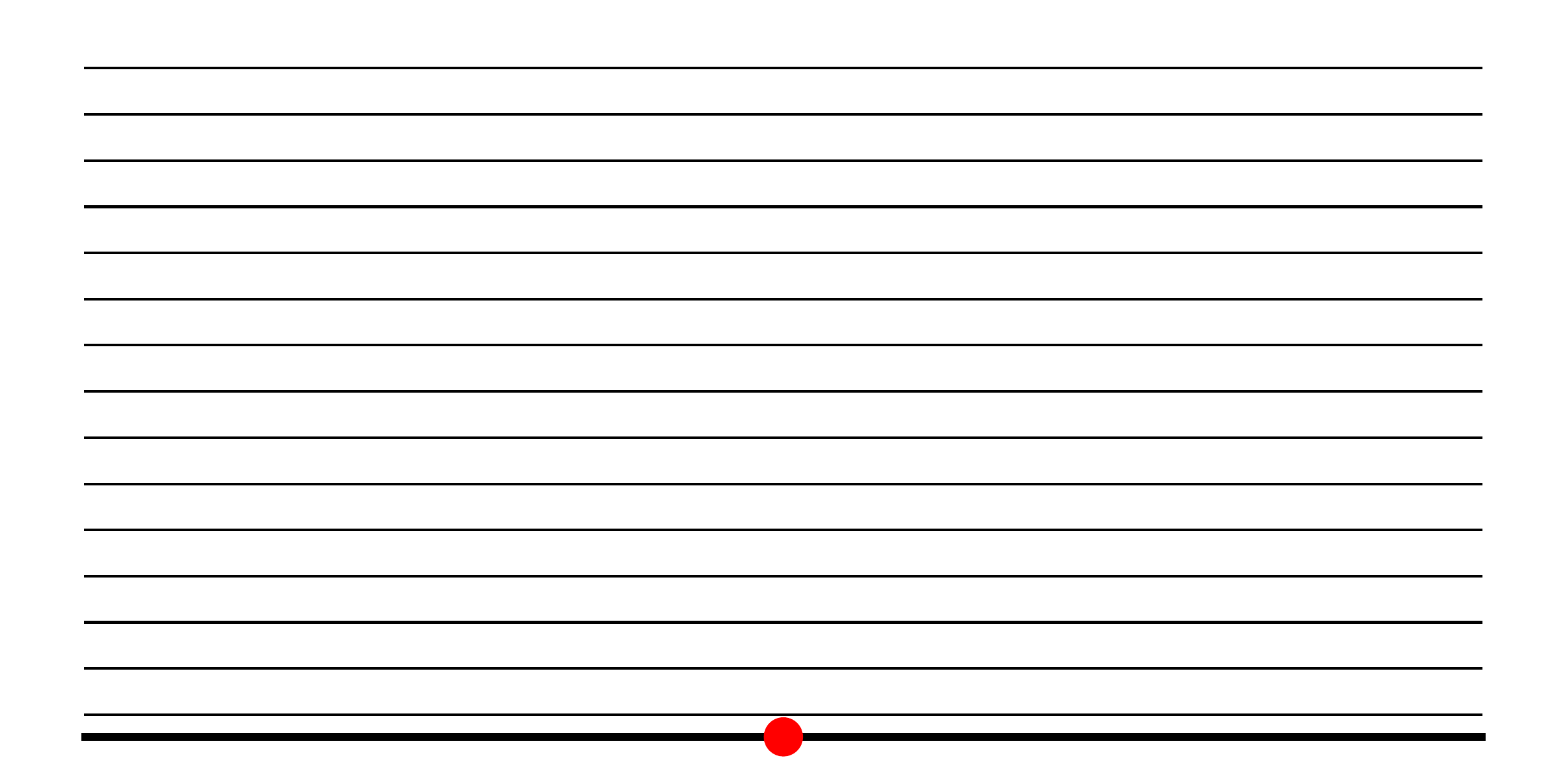}
		\caption{Sepal zone}	
	\end{subfigure}
	\begin{subfigure}[t]{0.49\textwidth}
		\includegraphics [width=\textwidth]{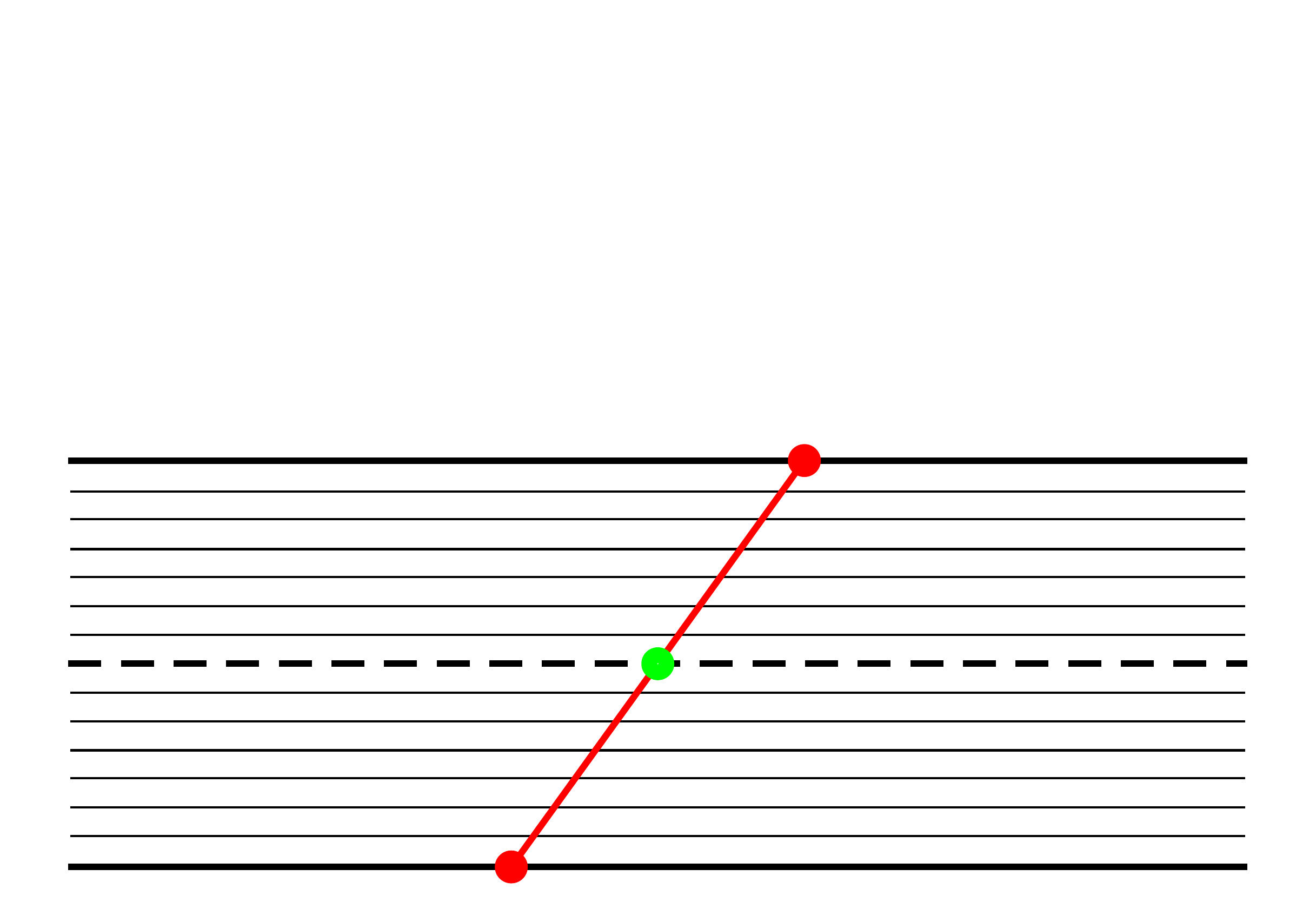}
		\caption{$\alpha\omega$-zone}	
	\end{subfigure}
	\caption{Images of zones of rotationally stable horizontal foliations in coordinate $\pm\e^{-\i\vartheta}\bt$: separatrices and saddles on the boundary,
		the equilibrium(s) correspond to the point(s) at infinity. An $\alpha\omega$-zone with its transversal (red), its midpoint and the gate trajectory (dashed). }
	\label{figure:zones}
\end{figure}

\bigskip

Assume \eqref{eq:assumption} and let $\e^{-2\i\vartheta}\Delta$ be rotationally stable.
Then there are only two kinds of zones (see Figure~\ref{figure:zones}): 
\begin{enumerate}[leftmargin=2\parindent]
	\item \emph{rotationally stable sepal zone} with one saddle point at its boundary line,
	\item \emph{rotationally stable $\alpha\omega$-zone} with one saddle point at each of its two boundary lines.
\end{enumerate}

\begin{figure}[t]
	\centering
		\includegraphics [width=0.55\textwidth]{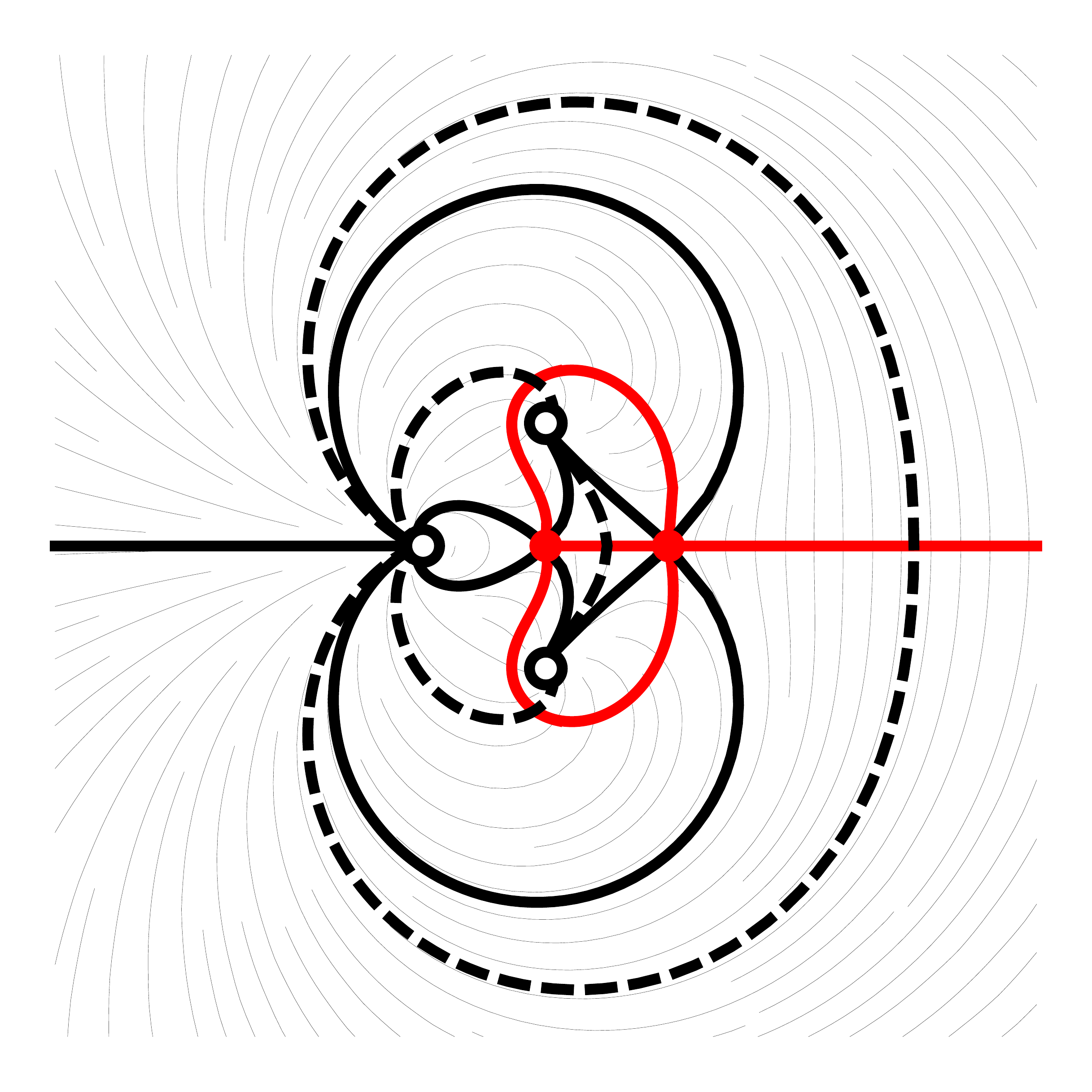}
	\caption{The separating graph  (solid black), the gate graph (dashed black) and the transverse graph (red) for $\Delta=-\frac{x^2(x-1)^2}{(x^2+1)^2(x+1)^3}(\d x)^2$. The horizontal foliation has 1 sepal zone and 4 $\alpha\omega$-zones.}
	\label{figure:example}
\end{figure}

\subsubsection{Gates, transversals and generalized petals}

\begin{definition}
\begin{enumerate}
\item 	Rotationally stable $\alpha\omega$-zone contains exactly one saddle connection, which joins its two saddles. It is called the \emph{transversal} of the zone \cite{Douady-Estrada-Sentenac}.
	Since  the transversal has a finite period \eqref{eq:period}, it has a well defined \emph{midpoint}; see Figure~\ref{figure:zones}. Note that the transversal and the midpoint are independent  of the angle $\vartheta$.  
	Following \cite{Hurtubise-Lambert-Rousseau}, the trajectory through the midpoint is called \emph{gate}. See also \cite{Douady-Estrada-Sentenac, Pilgrim, Oudekerk} for the notion of gates and transversals.
	
\item	The \emph{gate graph} has equilibria as vertices and the gates of $\alpha\omega$-zones as edges (see Figure~\ref{figure:zones}).

\item	The \emph{extended separating graph} is the union of the separating graph and the gate graph.
	
\item	The	\emph{(generalized) petals} of a rotationally stable horizontal foliation are the connected components of the complement of the extended separating graph. They are either 
\begin{itemize}[leftmargin=\parindent]
	\item  halves of rotationally stable  $\alpha\omega$-zones split lengthwise in two by their gate trajectories, or
	\item  rotationally stable sepal zones.
\end{itemize}
Thus each petal contains exactly one hyperbolic sector of some saddle point.   
	\end{enumerate}
	
\item	The \emph{transverse graph} has saddles as vertices and the transversals of $\alpha\omega$-zones as edges see Figure~\ref{figure:zones}.

\item	The \emph{extended transverse graph} (Reeb graph) has saddles as vertices and 
\begin{itemize}[leftmargin=\parindent]
	\item for each $\alpha\omega$-zones the transversal as an edge, divided by its the midpoint into a pair of half-edges (of finite length),
	\item for each sepal zones the unique separatrix of the vertical foliation of $\Delta$ inside the zone as an open half-edge (of infinite length).
\end{itemize}
Therefore each petal contains exactly one half-edge, whose points are in one-to-one correspondence with the trajectories inside the petal.
In particular this means that each vertex saddle point of rank $\nu<0$ is attached to $2|\nu|$ half-edges which alternate with its $2|\nu|$ separatrices. 

\end{definition}

\begin{lemma}\label{lemma:transversalgategraph}
Assume \eqref{eq:assumption}. The transverse graph  of a rotationally stable horizontal foliation is a homotopy retract of $\sX\smallsetminus\Equilib(\Delta)$: it divides $\sX$ into cells, one for each equilibrium. 
The gate graph is the dual graph: it connects all equilibria and is a homotopy retract of $\sX\smallsetminus\Saddle(\Delta)$.
\end{lemma}
\begin{proof}
Cut each $\alpha\omega$-zone across along its transversal, than each half is attached to just one equilibrium.
Now given an equilibrium point take the union of all the halves of $\alpha\omega$-zones, all the sepal zones and all the separatrices that are attached to it.
Clearly this set covers a full neighborhood of the equilibrium. Since the foliation is rotationally stable there are no saddle trajectories, therefore the boundary of this set can only consist of saddles and transversals of the $\alpha\omega$-zones attached to the equilibrium. 
 \end{proof}

\subsubsection{Flat coordinate system}
Let $\tilde\sX\to\sX$ be the 2-sheeted cover on which $\Delta^{\frac12}$ is defined, equipped with the lift of horizontal foliation of $\e^{-2\i\vartheta}\Delta$ and its petal decomposition.  
For each saddle point $b_i$ and each petal on $\tilde\sX$ attached to it there is a unique flat coordinate $\bt(x)=\int_{b_i}^x\Delta^{\frac12}$ \eqref{eq:t}
on the petal which vanishes at the saddle.
This flat coordinate system on $\tilde\sX$, determined uniquely up to a sign, is such that it
\begin{itemize}
	\item agrees over separatrices,
	\item differs by period of transversals over gates.
\end{itemize}

\begin{figure}[t]
	\centering
	\begin{subfigure}[t]{0.4\textwidth}
		\includegraphics [width=\textwidth]{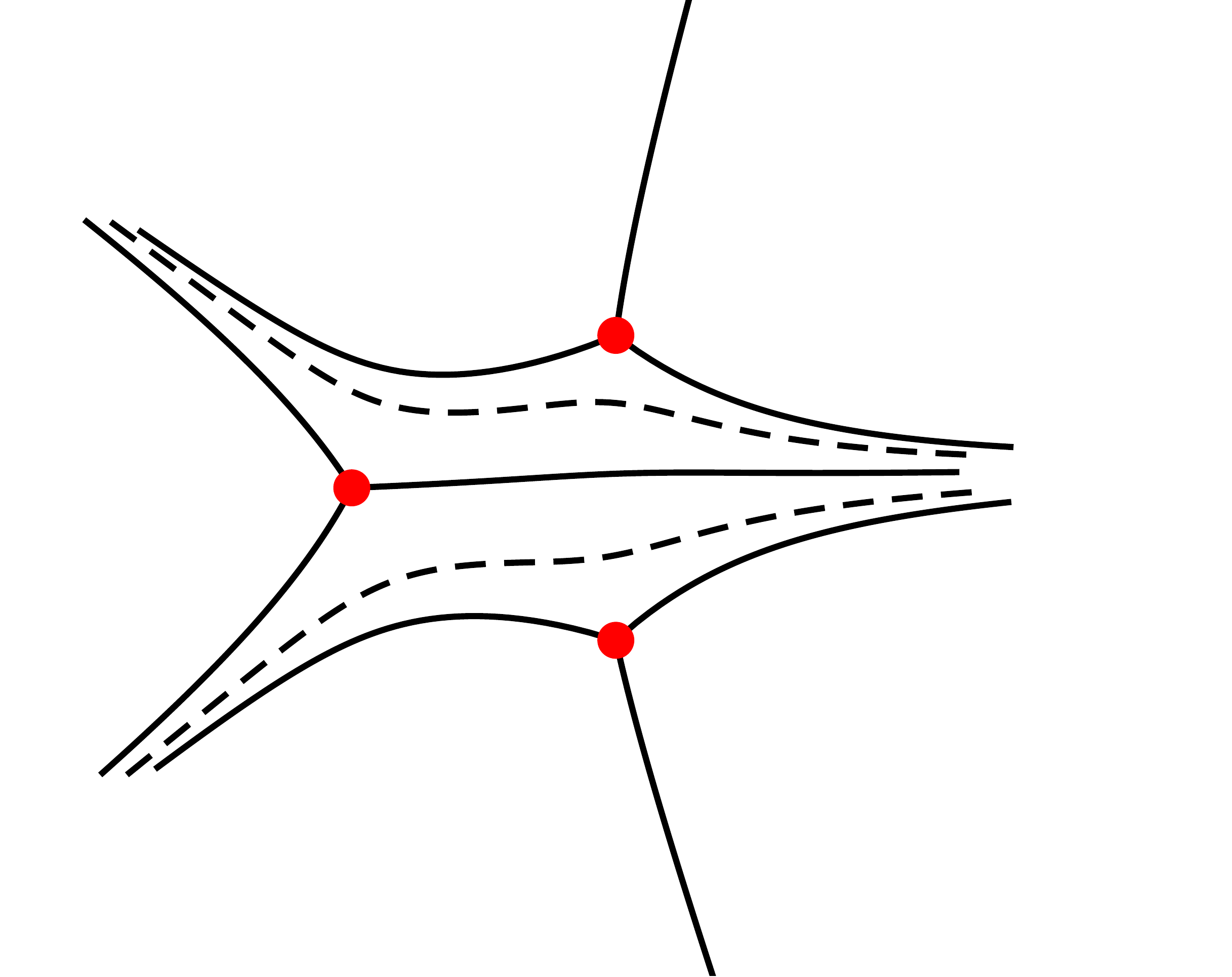}
		\caption{$\Delta_\epsilon=c(x^3+\epsilon)(\d x)^2$}	
	\end{subfigure}
	\hskip0.05\textwidth
	\begin{subfigure}[t]{0.4\textwidth}
			\includegraphics [width=\textwidth]{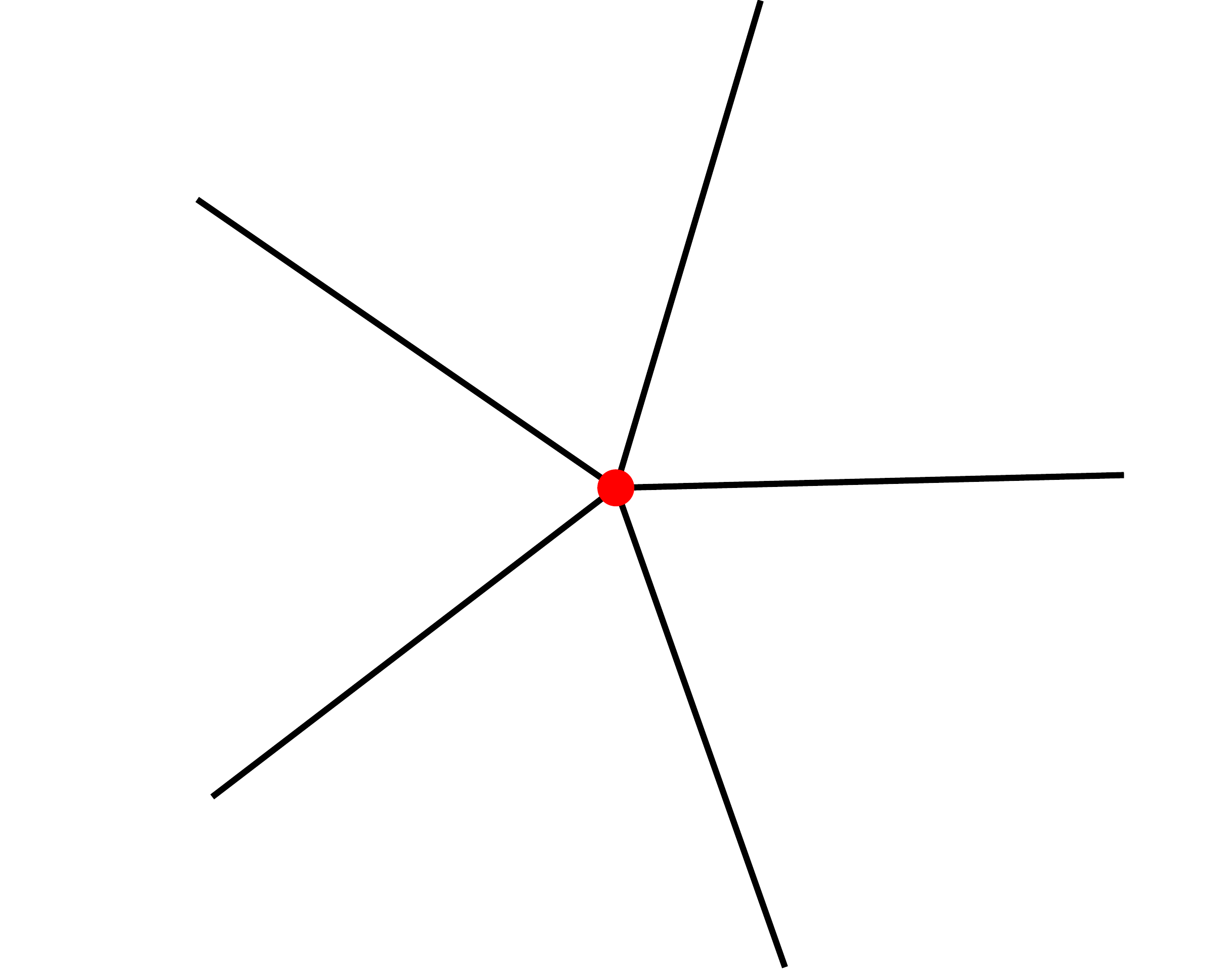}
		\caption{$\Delta_0=cx^3(\d x)^2$}
	\end{subfigure}\\
		\centering
	\begin{subfigure}[t]{0.4\textwidth}
		\includegraphics [width=\textwidth]{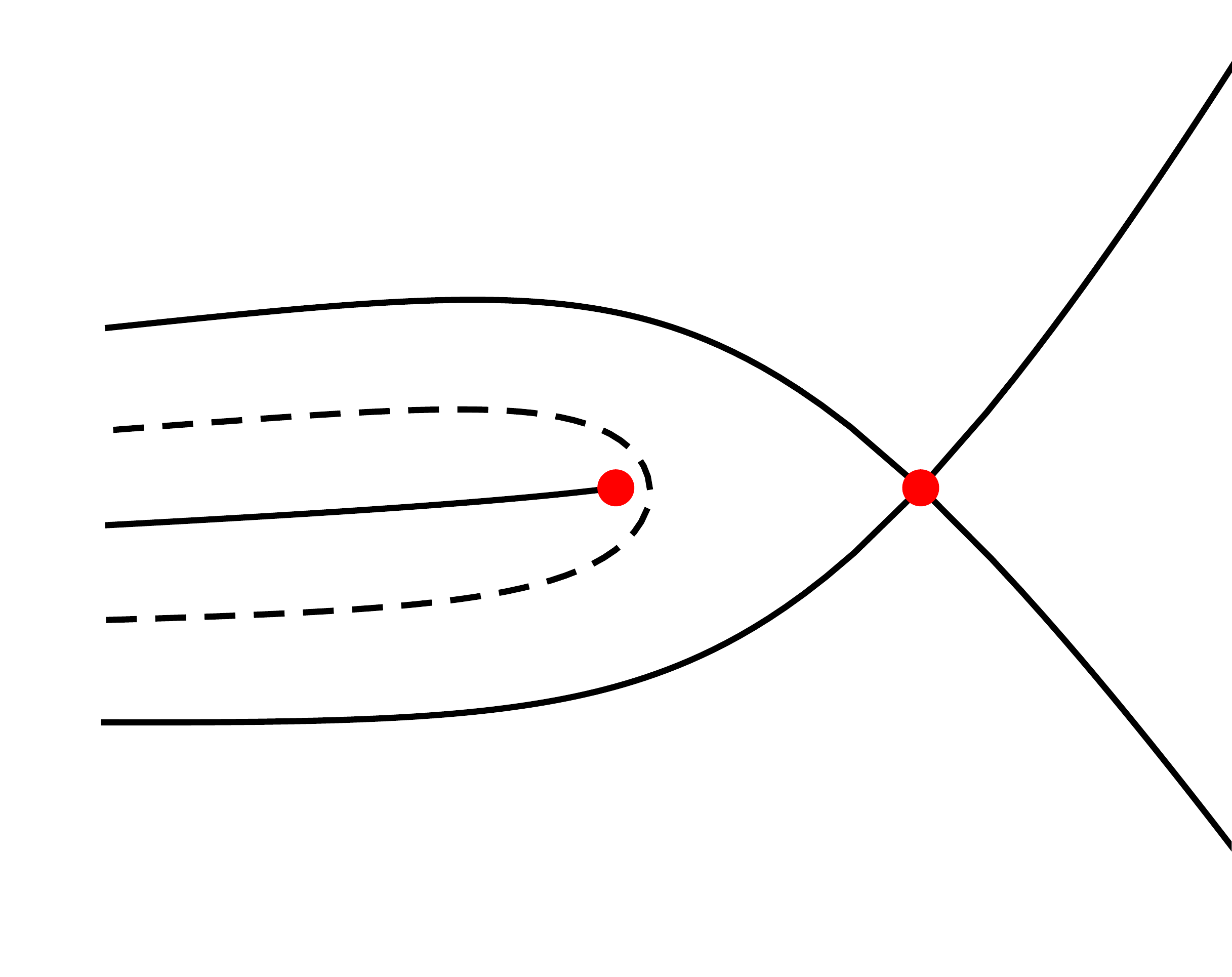}
		\caption{$\Delta_\epsilon=c\frac{x^2}{x+\epsilon}(\d x)^2$}	
	\end{subfigure}
	\hskip0.05\textwidth
	\begin{subfigure}[t]{0.4\textwidth}
		\includegraphics [width=\textwidth]{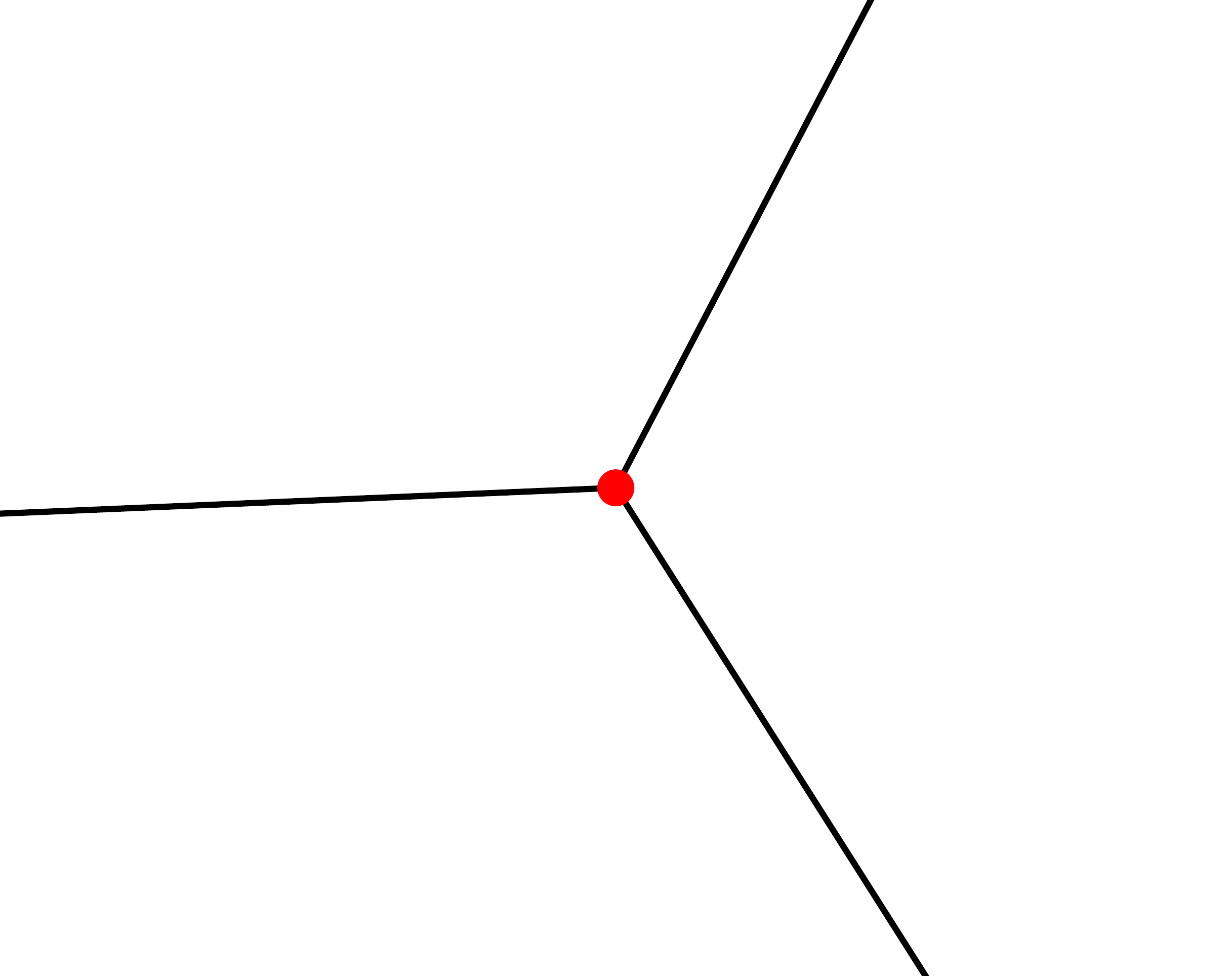}
		\caption{$\Delta_0=cx(\d x)^2$}
	\end{subfigure}
	\caption{Confluence of saddle points (separatrices in solid lines, gates dashed).}
	\label{figure:confluentsaddles}
\end{figure}

\subsubsection{Parametric theory}\label{sec:parametricfoliation}

Let $\Delta(x,\epsilon)$ be a parametric family of quadratic differentials over some open set of parameters $\sE\subseteq \C^l$.
For each $\epsilon\in\sE$ let $\Delta_\epsilon(x)=\Delta(x,\epsilon)$ and denote
\[\Crit(\Delta)=\coprod_{\epsilon\in\sE}\Crit(\Delta_\epsilon)\]
the \emph{critical divisor} in $\sX\times\sE$.
Let $\sC\subset\sE$ be the \emph{confluence divisor} where critical points merge and their ranks change.
Assume that 
\begin{itemize}
	\item $\Crit(\Delta)$ is closed in $\sX\times\sE$: if a regular point is created by a merging of critical points then it needs to be considered as a saddle of rank $-1$,
	\item the restriction of $\Crit(\Delta)$ to $\sX\times (\sE\smallsetminus\sC)$ is a smooth codimension 1 variety transverse to the fibers of $(x,\epsilon)\mapsto\epsilon$.
\end{itemize}
Also denote
\[\Cal{RS}=\{(\vartheta,\epsilon)\in\big(\R/\pi\Z\big)\times \sE\mid \text{$\e^{-2\i\vartheta}\Delta(x,\epsilon)$ is rotationally stable}\}=\coprod_{\epsilon\in \sE}\Cal{RS}_\epsilon.\]

\begin{proposition}
	Assume \eqref{eq:assumption}.
	\begin{enumerate}[leftmargin=\parindent]
		\item The interior $\mathring{\Cal{RS}}$ of the set $\Cal{RS}$ is dense in $\big(\R/\pi\Z\big)\times\sE$.
		\item $\Cal{RS}\smallsetminus\mathring{\Cal{RS}}=\coprod_{\epsilon\in\sC}\Cal{RS}_\epsilon$
	\end{enumerate}
\end{proposition}

\begin{proof}
	(1) This follows from Proposition~\ref{prop:rotstability}: the unstable directions $\R/\pi\Z\smallsetminus\Cal{RS}_\epsilon$
	are the arguments of periods of saddle connections which depend continuously on $\epsilon\in\sE\smallsetminus\sC$.
	Note that  a saddle connection can cease to exist through acquiring an extra saddle, i.e. when a triangle in the translation surface flattens to a segment.
	It happens over a real codimension 1 set in the parameter space $\sE$.
	
	(2) By the previous considerations $\coprod_{\epsilon\in\sE\smallsetminus\sC}\Cal{RS}_\epsilon\subseteq \mathring{\Cal{RS}}$.
Conversely, when $\epsilon\to \epsilon_0\in\sC$, as some critical points merge there is always an attached $\alpha\omega$-zone whose transversal has period which tends to either $0$ or $\infty$ with argument rotating in dependence on $\epsilon-\epsilon_0$. This means that the $\epsilon_0$-slice of $\big(\R/\pi\Z\big)\times \sE\smallsetminus\mathring{\Cal{RS}}$ is the whole direction circle $\R/\pi\Z$.
\end{proof}

Let us look in more detail at the various confluences of critical points that can arise:

\begin{itemize}
	\item \emph{Confluence of saddles:}
	If a saddle point of rank $\nu<0$, which has $2|\nu|$ separatrices, is a limit of merging of $j$ saddle points of ranks $\nu_1, \ldots, \nu_j<0$, \ $\nu_1+\ldots+\nu_j+j=\nu+1$, this means that before the limit for $(\vartheta,\epsilon)\in\mathring{\Cal{RS}}$ there are extra $2j-2$ separatrices which will merge together with other separatrices.
	The extra $2j-2$ parabolic sectors belong to $j-1$ ``thin'' $\alpha\omega$-zones that pairwise separate the $j$ saddles (``thin'' meaning that their $\bt$-width $\to 0$). 
	In fact in this situation a local version of Lemma~\ref{lemma:transversalgategraph} holds: the transversals of these $j-1$ $\alpha\omega$-zones form a tree which connects the $j$ saddle points, and which shrinks to a point at the limit. See Figure~\ref{figure:confluentsaddles}.

	\item \emph{Confluence of equilibria:}
	If an equilibrium of rank $\nu>0$, which has $2\nu$ sepals, is a limit of merging of $j$ equilibria  of ranks $\nu_1, \ldots, \nu_j\geq 0$, \ $\nu_1+\ldots+\nu_j+j=\nu+1$.
	Then for $(\vartheta,\epsilon)\in\mathring{\Cal{RS}}$ there are $j-1$ ``thick'' $\alpha\omega$-zones that connect the equilibria (``thick'' meaning that their $\bt$-width $\to \infty$).
	Again a local version of Lemma~\ref{lemma:transversalgategraph} holds:  the gates of these $j-1$ $\alpha\omega$-zones form a tree which connects the $j$ equilibria and splits the $j-1$ $\alpha\omega$-zones into $2j-2$ petals.
	At the limit, the gates shrink to a point but the petals persists. See Figure~\ref{figure:b}.

	\item \emph{Confluence of equilibria and saddles:}
	Assume an equilibrium of rank $\nu>0$ is a limit of merging of $j$ equilibria  of ranks $\nu_1, \ldots, \nu_j\geq 0$ and of $l$ saddles of ranks
	$\nu_{j+1}, \ldots, \nu_{j+l}<0$, \ $\nu_1+\ldots+\nu_{j+l}+j+l=\nu+1$.	
	Then there are $2\nu$ \emph{persisting petals} which have a non-empty limit when $(\vartheta, \epsilon)\to (\vartheta_0,0)$ in $\Cal{RS}$,
and $2|\nu_{j+1}+\ldots+\nu_{j+l}|$ \emph{vanishing  petals} which shrink and disappear at the limit.
The vanishing petals are those attached to the saddles; for each saddle the closure of their union is a simply connected neighborhood bounded by a loop of vanishing gates.
The homotopy type of the vanishing gate graph is that of a bouquet of $l$ loops, the Euler characteristic is $l-1$. It has $j$ vertices at the equilibria and therefore $j+l-1$ edges.
Therefore one has again a local version of Lemma~\ref{lemma:transversalgategraph}.
	
\end{itemize}

\begin{proposition}
	Assume \eqref{eq:assumption}.
The petals of  $\e^{-2\i\vartheta}\Delta(x,\epsilon)$ depend continuously under a Hausdorff metric on their compact closures in $\sX$ on $(\vartheta,\epsilon)\in \Cal{RS}$, even though some can disappear at points of $\Cal{RS}\smallsetminus\mathring{\Cal{RS}}$.
\end{proposition}

\begin{proof}
The petals are in one-to-one correspondence with the hyperbolic sectors of saddle points. 
The boundary of a petal consists of the pair of separatrices delimiting the hyperbolic sector, plus potentially by a gate trajectory.
Near a saddle point the separatrices depend continuously on $(\vartheta,\epsilon)$ as long as the rank of the saddle point doesn't change. 
And it takes them a finite $\bt$-time to enter the periodic/parabolic domain of an equilibrium (Definition~\ref{def:periodicparabolic}), once there, it will inevitably tend to the equilibrium.
The equilibrium point to which it tends therefore cannot change with $(\vartheta,\epsilon)$ unless passing through a rotationally unstable situation
where the separatrix becomes a boundary component of the periodic/parabolic domain.
Also the periodic and parabolic domains depend continuously on $\epsilon$ (they are independent of $\vartheta$) as long as the rank of the equilibrium doesn't change. 
In fact, the boundary of each domain consists of a finite number of saddle connections which depend continuously on $\epsilon\in\sE\smallsetminus\sC$.
Although saddle connections can cease to exist through flattening of a triangle in the translation surface, this is still a continuous process.

For $\epsilon\to \epsilon_0 \in\sC$ some saddle connections can disappear at the limit: either they shrink and vanish or they get pinched by merging critical points.
This can only increase the size of the periodic/parabolic domain of the limit equilibria.
However, it can also happen that an equilibrium point disappears through a confluence with saddles when the limit point is a saddle (including that of rank $-1$). 
But in this case the limit of a separatrix landing at such equilibrium would be a saddle trajectory, i.e. $(\vartheta, \epsilon)\to(\vartheta_0,\epsilon_0)\notin \Cal{RS}$.
\end{proof}

\begin{figure}[t]
	\centering
	\begin{subfigure}[t]{0.49\textwidth}
		\includegraphics [width=\textwidth]{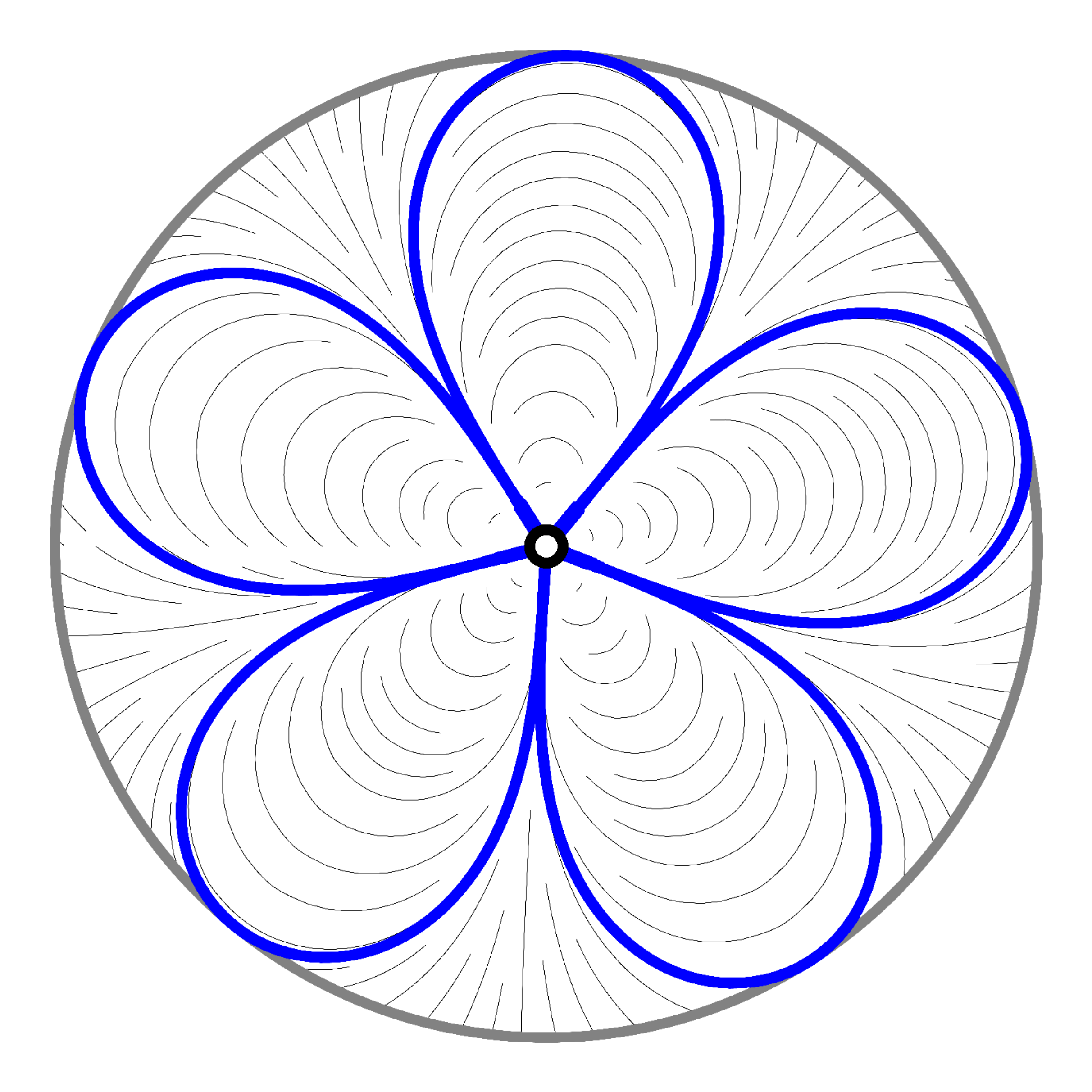}
		\caption{$\Delta_0=\frac{c}{x^7}(\d x)^2$}	
	\end{subfigure}
	\begin{subfigure}[t]{0.49\textwidth}
		\includegraphics [width=\textwidth]{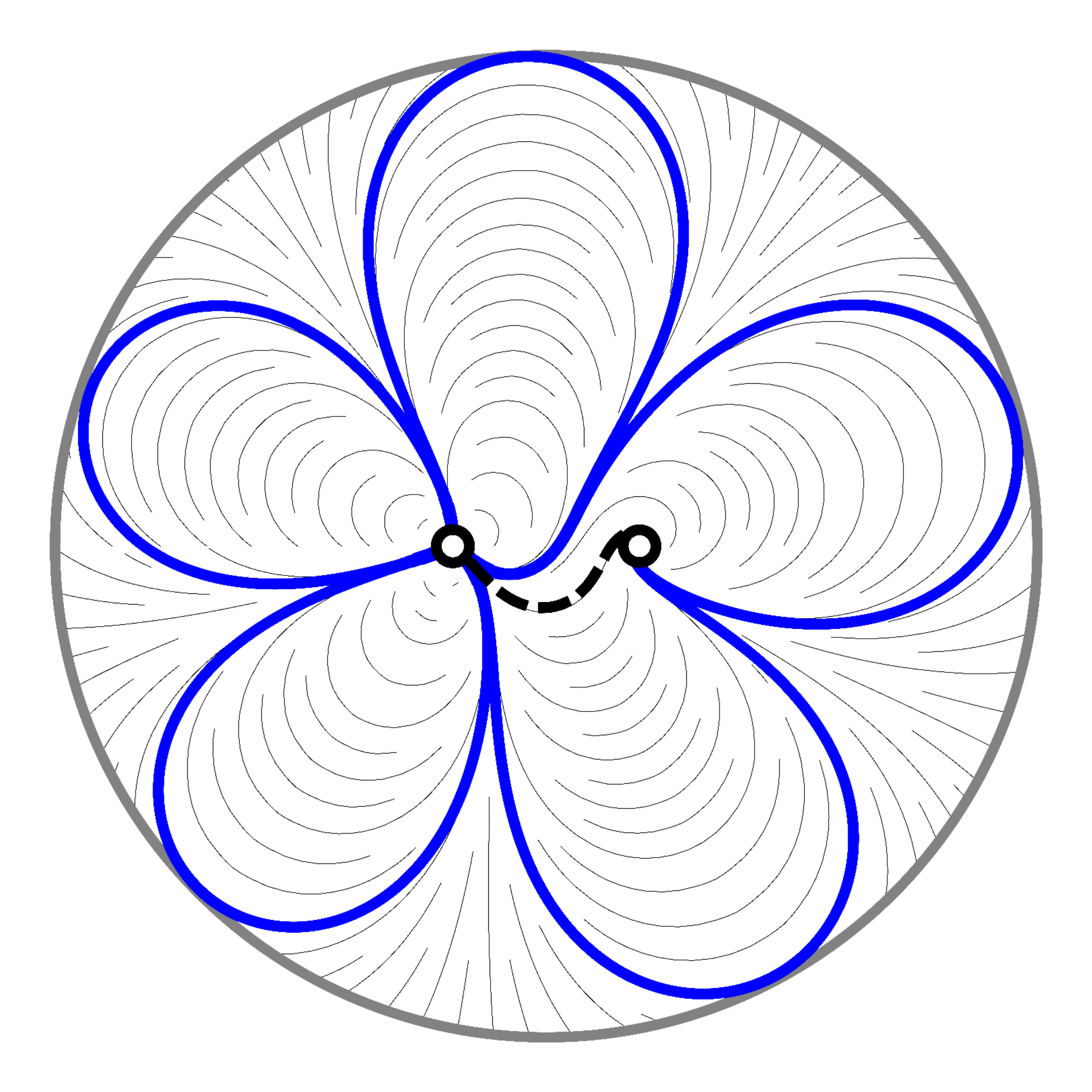}
		\caption{$\Delta_\epsilon=\frac{c}{(x-\epsilon)^2(x+\epsilon)^5}(\d x)^2$}	
	\end{subfigure}
	
	\begin{subfigure}[t]{0.49\textwidth}
		\includegraphics [width=\textwidth]{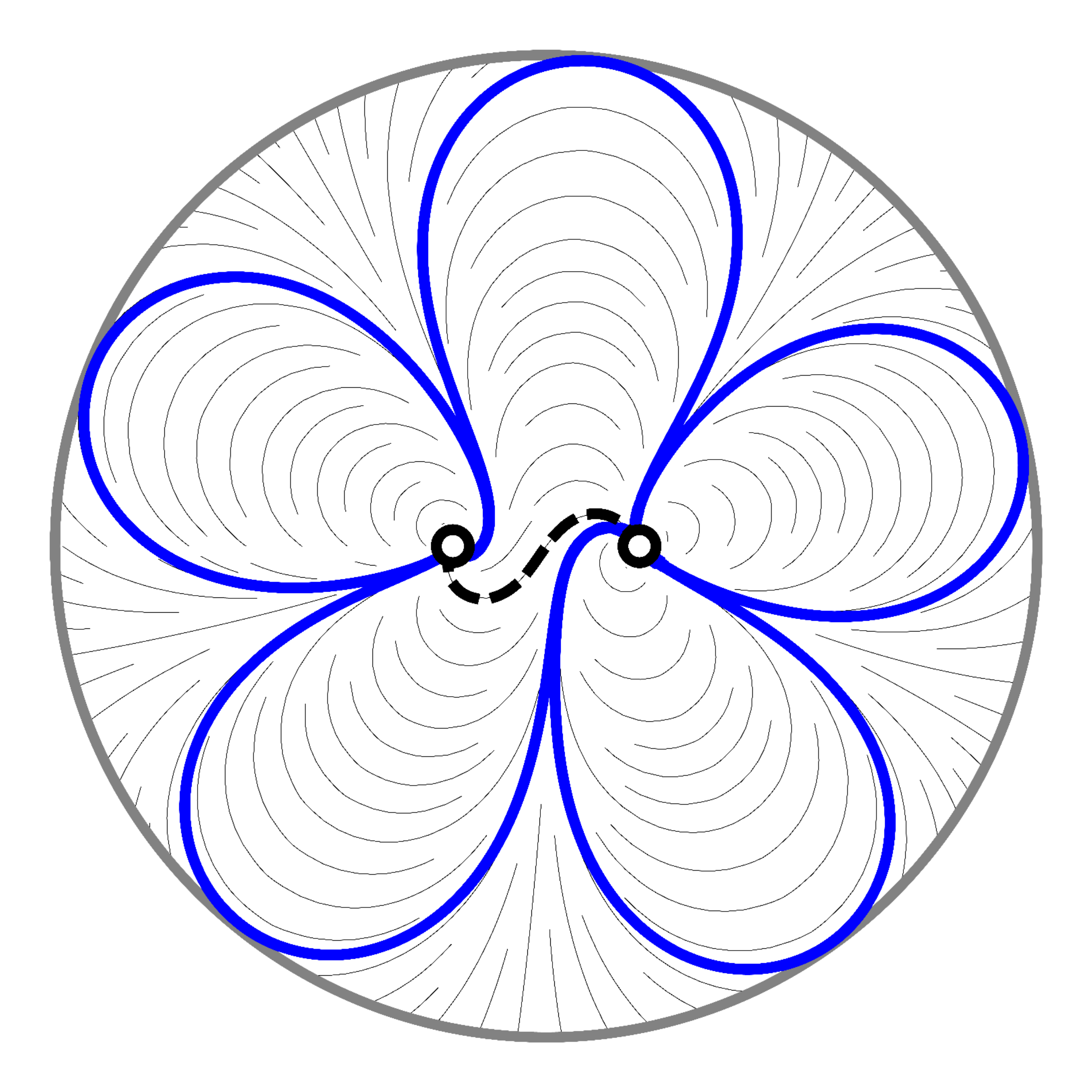}
		\caption{$\Delta_\epsilon=\frac{c}{(x-\epsilon)^3(x+\epsilon)^4}(\d x)^2$}	
	\end{subfigure}
	\begin{subfigure}[t]{0.49\textwidth}
		\includegraphics [width=\textwidth]{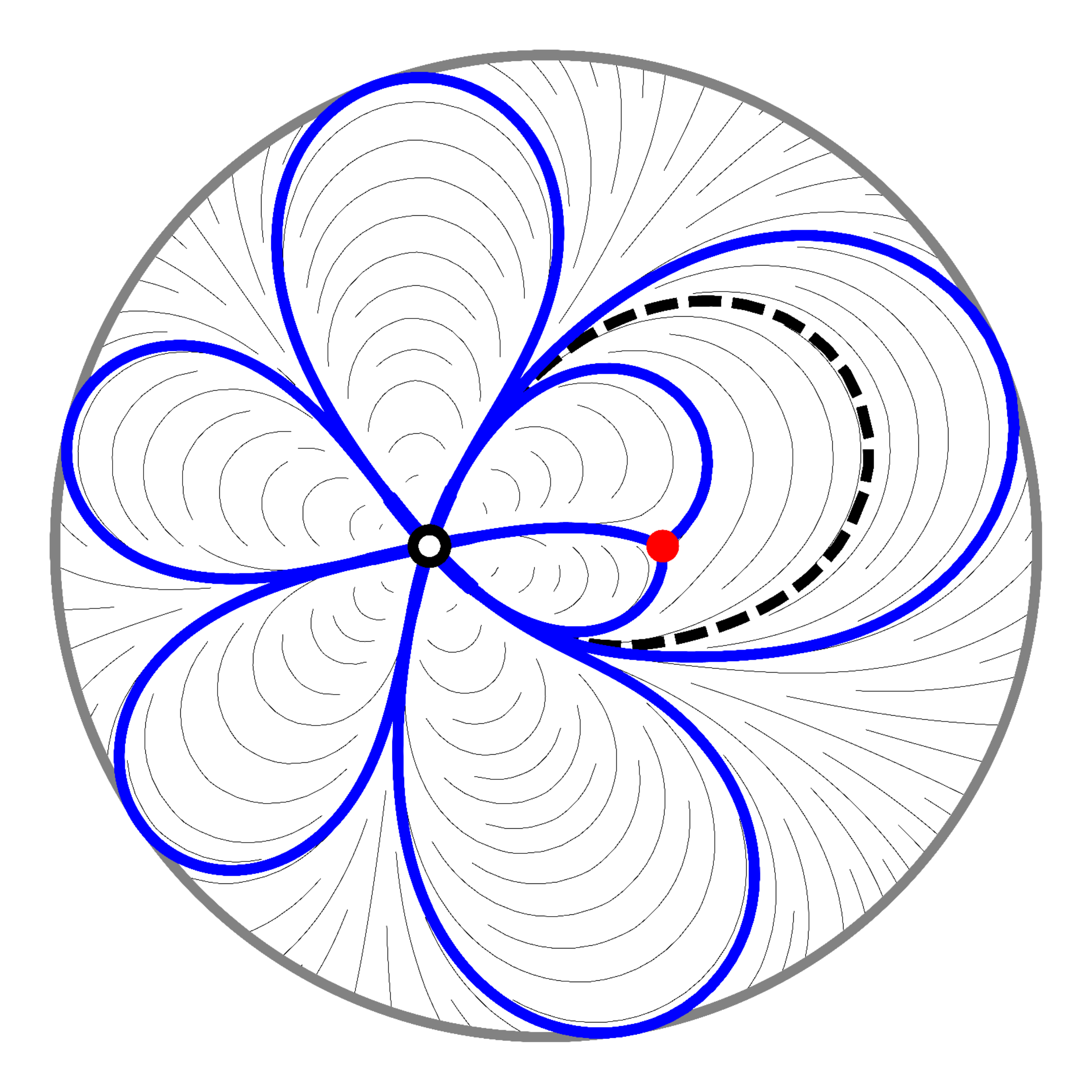}
		\caption{$\Delta_\epsilon=\frac{c(x-\epsilon)}{(x+\epsilon)^8}(\d x)^2$}	
	\end{subfigure}
	\caption{Examples of zones/petals relative to a disc $\sX$.}
	\label{figure:b}
\end{figure}

\begin{figure}[t]
	\centering
	\begin{subfigure}[t]{0.49\textwidth}
		\includegraphics [width=\textwidth]{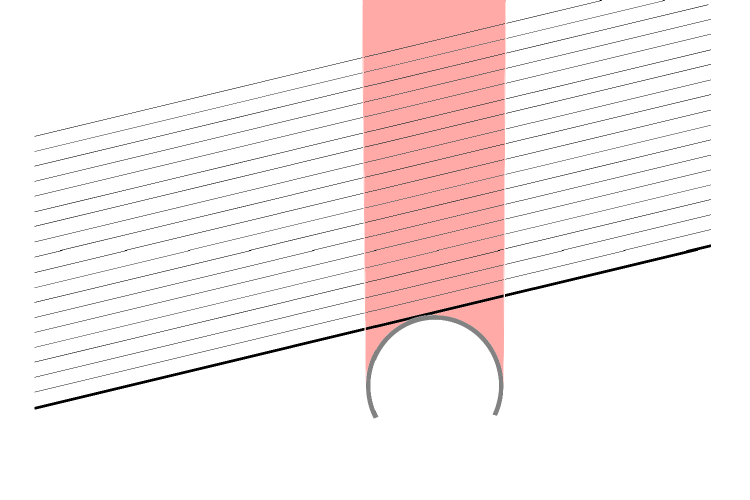}
		\caption{Sepal zone relative to $\sX$}	
	\end{subfigure}
	\begin{subfigure}[t]{0.49\textwidth}
		\includegraphics [width=\textwidth]{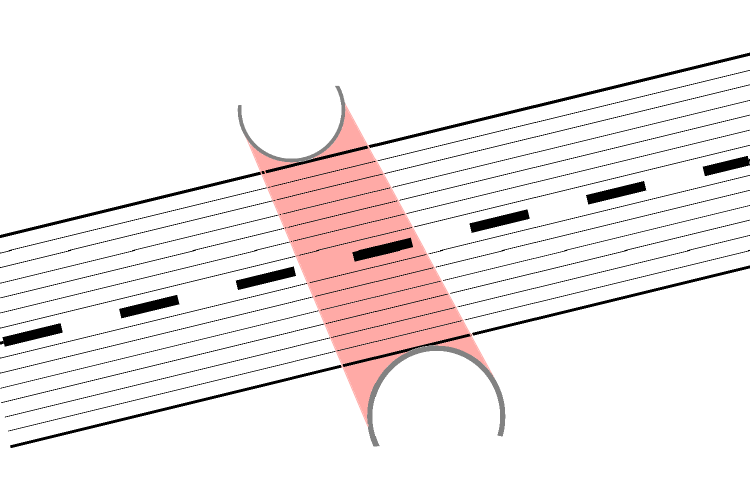}
		\caption{$\alpha\omega$-zone relative to $\sX$}	
	\end{subfigure}
	\caption{Images of zones relative to a domain $\sX$ in the coordinate $\pm\bt$: 
		boundary of  $\sX\smallsetminus\Saddle$ in gray, fat transverse graph in pink, gate graph dashed black.}
	\label{figure:zones2}
\end{figure}

\subsubsection{Zones relative to a domain}

Instead of considering a meromorphic quadratic differential $\Delta$ globally on a compact Riemann surface, we shall now consider it on some domain $\sX$, and reformulate the above theory relative to the domain.

\begin{definition}\label{def:relativezones}~
Let $\sX$ be an open domain with a compact closure, $\Delta$ a meromorphic quadratic differential on the closure of $\sX$ with at least one equilibrium inside $\sX$ and none on the boundary.
\begin{enumerate}[label=\roman*),leftmargin=\parindent]
\item The role of saddle points is played by the boundary components of  $\sX\smallsetminus\Saddle(\Delta)$.	
\item For a fixed $\vartheta$, the role of the separating graph  for $\e^{-2\i\vartheta}\Delta$ is played by the
\emph{fat separating graph  relative to $\sX$}: the set of all trajectories that escape $\sX\smallsetminus\Saddle(\Delta)$ in at least one direction.
Near a simple equilibrium it takes the form of a union of \emph{separating spiral sectors}, while at a parabolic equilibrium it is a union of \emph{separating cusps};
both correspond to semi-infinite parallel strips in the direction $\vartheta$ in the translation coordinate $\bt(x)$.

\item Its complement in $\sX$ is a union of \emph{zones relative to $\sX$}: they consist of complete trajectories inside $\sX\smallsetminus\Saddle(\Delta)$. 

\item The role of a saddle connection is played by \emph{fat saddle connection relative to $\sX$}: 
union over varying $\vartheta$ of a connected component of the set spanned by the trajectories of $\e^{-2\i\vartheta}\Delta$ with which escape $\sX\smallsetminus\Saddle(\Delta)$ in both directions depending continuously on $\vartheta$.
In the translation coordinate it takes the form of the convex envelope of two copies of boundary components of $\sX$: a finite strip with straight but not necessary parallel margins.
In particular, \emph{fat transversal of relative  $\alpha\omega$-zone} is the the fat saddle connection that connects the points at which the zone touches boundary of $\sX$.

\item The  \emph{gate of relative  $\alpha\omega$-zone} is the trajectory that splits its $\bt$-image into a pair of strips half the thickness (Figure~\ref{figure:zones2}).

\item The role of saddle trajectories is played by the
those trajectories that escape $\sX\smallsetminus\Saddle(\Delta)$ in both directions. If the set of them is empty for some angle $\vartheta$, then 
$\e^{-2\i\vartheta}\Delta$ \emph{is stable relative to $\sX$} and \emph{$\vartheta$ is a stable angle}.


\item For a stable $\vartheta$,  the \emph{petals relative to $\sX$} are the connected components of the complement  of the fat separating graph  and of the gate graph in $\sX$.

\end{enumerate}	
\end{definition}

\subsection{Flags and splittings}\label{sec:flagsplitting}

Assume $m<2k$.
Let 
\[\sX\times \sE=\{|x|<\delta_x\}\times\{|\epsilon|<\delta_\epsilon\},\quad \delta_x,\delta_\epsilon>0,\]
be a small poly-disc on which a system \eqref{eq:unfoldedQ} is analytic, and assume that $\delta_\epsilon>0$ small enough, in particular such that
\[\Sing_\epsilon=\{x:P(x,\epsilon)=0\}\subset \sX,\qquad \epsilon\in\sE.\]
Let 
\[\Delta(x,\epsilon)=\Delta_\epsilon(x)=\frac{\tilde Q(x,\epsilon)}{P(x,\epsilon)^2}(\d x)^2\] 
be the associated meromorphic quadratic differential,.



Following Theorem~\ref{proposition:mixedbasis}, denote
\begin{equation}\label{eq:sX}
	\sX_\epsilon=\left\{|x|<\delta_x,\quad	\left|\tilde b(x,\epsilon)\right|<\tfrac12 \right\},\qquad \tilde b(x,\epsilon)=\tfrac{P(x,\epsilon)}{4\sqrt{\tilde Q(x,\epsilon)}}\tdd{x}\log \tilde Q(x,\epsilon).
\end{equation}
We shall  \textbf{\emph{assume that}}
\begin{equation}\label{eq:ass}
\sX_\epsilon \ \text{is connected}\qquad \text{and} \qquad  \Sing_\epsilon=\Equilib_\epsilon\subset\sX_\epsilon.
\end{equation}
The connectedness assumption on $\sX_\epsilon$ is automatically satisfied for small $\epsilon$ when, for example,
there is either no saddle point, or there is only one saddle and one equilibrium (see Proposition~\ref{prop:covering}).

\begin{example}
Examples where the domain $\sX_\epsilon$ is disconnected  for every small $\epsilon\neq 0$.

\medskip
\noindent (1)
$\tilde Q(x,\epsilon)=x^{3}$, \ $P(x,\epsilon)=(x-\epsilon^2)(x-\epsilon)^2$: \
Write $x=\epsilon^{\frac43}z$, then
\[\tilde b(x,\epsilon)^{-1}=\tfrac8{3}\mfrac{x^{\frac52}}{(x-\epsilon^2)(x-\epsilon)^2}=\tfrac8{3}\mfrac{z^{\frac52}}{(z-\epsilon^\frac23)(\epsilon^\frac13 z-1)^2}\to \tfrac8{3}z^{\frac32},\] 
so the $x$-radius of
the complementary set $\big\{|\tilde b|^{-1}\leq 2\big\}$ is asymptotically proportional to $|\epsilon|^\frac43$.
However the root $x=\epsilon^2$ tends to the origin with a faster rate, therefore its open component in $\sX_\epsilon$ gets disconnected from the outer open component containing the slower root $x=\epsilon$.

Some of the theory can nevertheless be adapted to this situation by restricting to the outer open component of  $\sX_\epsilon$,
thus omitting the fast singular point $x=\epsilon^2$ from the picture: considering $\tilde Q'(x,\epsilon)=\frac{x^3}{(x-\epsilon^2)^2}\sim x$ as a perturbed saddle of rank $-\frac32$ and $P'(x,\epsilon)=(x-\epsilon)^2$.


\medskip
\noindent (2) $\tilde Q(x,\epsilon)=x$, \ $P(x,\epsilon)=(x-\epsilon^3)(x-\epsilon)$: \
 Similar to the previous example, the root $x=\epsilon^3$ tends to 0 faster than the radius of the complementary set which is asymptotically proportional to $|\epsilon|^2$.
One could again try to restrict to the outer open component of  $\sX_\epsilon$ by considering  $\tilde Q'(x,\epsilon)=\frac{x}{(x-\epsilon^3)^2}\sim \frac{1}{x}$ as a perturbation of a saddle of rank $-\frac12$ and $P'(x,\epsilon)=(x-\epsilon)$. However in this case the outer component has only one simple equilibrium and the horizontal foliation of the quadratic differential $\e^{-2\i\vartheta}\Delta_\epsilon$ doesn't have any relative zone in $\sX_\epsilon$. 

%

\end{example}

Based on Theorem~\ref{proposition:mixedbasis} on the existence of mixed solution bases and  Theorem~\ref{prop:subdominant} on the existence of subdominant solutions,
we will construct two types of natural domains:
\begin{itemize}[leftmargin=2\parindent]
	\item[(A)] \emph{Enlarged petals} on which there is a splitting of solution space into a direct sum,
	\item[(B)] \emph{Lagoons} at singular points on which there is a flag filtration on solution space according to growth rates.
\end{itemize}
For $\epsilon=0$ they'll correspond to:
\begin{itemize}[leftmargin=2\parindent]
	\item[(A)] \emph{Large sectors} (sectors of Borel summation): each of them covers a consecutive pair of anti-Stokes rays (asymptotic directions of the horizontal foliation of $\Delta_\epsilon$) and exactly one Stokes ray (asymptotic direction of the vertical foliation of $\Delta_\epsilon$),
	\item[(B)] \emph{Small sectors} (intersection sectors):  each of them covers exactly one anti-Stokes ray and no Stokes rays.
\end{itemize}

\subsubsection{Enlarged petals and splitting of solution space}

Let 
\[\Cal{RS}_\sX\cap \Big(]\!-\!\tfrac{\pi}{2},\tfrac{\pi}{2}[\ \times \sE\Big)=\Big\{(\vartheta,\epsilon)\in\ ]\!-\!\tfrac{\pi}{2},\tfrac{\pi}{2}[\ \times \sE:\text{$\e^{-2\i\vartheta}\Delta(x,\epsilon)$ is rot. stable w.r.t. $\sX_\epsilon$}\Big\}
\]
see Definition~\ref{def:relativezones}.
Let $\mathring \sW$ be a connected component of its interior,
and let 
\begin{equation}\label{eq:sW}
\sW=\coprod_\epsilon \sW_\epsilon	
\end{equation}
be the relative closure of  $\mathring \sW$ inside $\Cal{RS}_\sX\cap \Big(]\!-\!\frac{\pi}{2},\frac{\pi}{2}[\ \times \sE\Big)$.

\begin{definition}[Enlarged petals]\label{def:enlargedpetal}
For $(\vartheta,\epsilon)\in \sW$, let $\sZ_{\vartheta,\epsilon}$ be a petal of $\e^{-2\i\vartheta}\Delta_\epsilon(x)$ relative to $\sX_\epsilon$ depending continuously on $(\vartheta,\epsilon)\in \sW$.
We denote $\sZ_{\epsilon}=\bigcup_{\vartheta\in \sW_\epsilon}\sZ_{\vartheta,\epsilon}$ a possibly ramified simply connected domain (technically this union should be taken on the universal covering of $\sX_\epsilon\smallsetminus\Crit_\epsilon$).
It is an enlargement of $\sZ_{\vartheta,\epsilon}$ through stable rotations. 	
\end{definition}

\begin{proposition}\label{prop:covering}
Assume that $\Sing_\epsilon=\Equilib_\epsilon$ and
\begin{itemize}
	\item either there is no saddle point,
	\item or there is only one saddle point and one parabolic equilibrium.
\end{itemize}
Let $\proj_\epsilon(\sW)$ be a ramified projection of $\sW$ onto the $\epsilon$-space $\sE$. Assuming the radius $\delta_\epsilon>0$ of $\sE$ is sufficiently small, 
then for every $\epsilon\in\proj_\epsilon(\sW)$ the union of the various enlarged petals $\sZ_{\epsilon}$ associated to $\sW$ covers a full neighborhood of each singularity.

The union of the $\epsilon$-projections $\proj_\epsilon(\sW)$ of the various components $\sW$ covers  whole $\sE$.
\end{proposition}

\begin{proof}
The coordinate $\bt_\epsilon(x)=\int \Delta_\epsilon^{\frac12}(x)$ is a ramified analytic function on $\CP^1$ with ramification points at the equilibria.
The image of the complementary set $\CP^1\smallsetminus\sX_\epsilon$ corresponds to holes on the Riemann surface of $\bt_\epsilon$.
Let us show that:
\begin{itemize}
	\item[(i)] The sizes of these holes on the different sheets of the Riemann surface of $\bt_\epsilon$ are uniformly bounded.
	\item[(ii)] The distance between any pair of such holes on the same sheet of the Riemann surface of $\bt_\epsilon$ tends to $+\infty$ with a rate at least $|\epsilon|^{-\alpha}$ for some $\alpha>0$.
\end{itemize}
This then means that the influence of these holes on the relative petals of $\e^{-2\i\vartheta}\Delta_\epsilon$ becomes the more negligible the smaller is $\epsilon$. 
In particular, this means that for each small $\epsilon$ there exists a stable direction $\vartheta$, hence $(\vartheta,\epsilon)\in\sW$ for some $\sW$ and $\epsilon\in\proj_\epsilon(\sW)$.

Given $\sW$, the regions in $\sX_\epsilon$ not covered by the union of the relative petals correspond on the Riemann surface of $\bt_\epsilon$ to triangular shadows of the holes: the angle at their vertex is the length $|\sW_\epsilon|$ of the interval in which $\vartheta$ varies, the stretch of the shadow is roughly $\frac{R}{\sin\frac{|\sW_\epsilon|}{2}}$ where $R$ is the radius of the hole. Therefore these shadows are uniformly bounded, and for $\epsilon$ small enough they are disjoint.

Let us prove (i). When $\epsilon$ is small enough, then on the outer complement $\CP^1\smallsetminus\sX=\{|x|\geq \delta_x\}$ one has $\bt_\epsilon(x)\sim\bt_0(x)$ so the radius of its $\bt_\epsilon$-image is close to that of its $\bt_0$-image.
If there are no saddles, then $\sX_\epsilon=\sX$, and the holes in the Riemann surface are images of the outer complement only.

If there is a single equilibrium $x=a(\epsilon)$ of rank $\nu_a>0$, and a single saddle point $x=s(\epsilon)$ of rank $-\nu_s<0$, then the complement of $\sX_\epsilon$ has also a component containing the saddle point.
One can assume $s(\epsilon)=0$. Write $x=a(\epsilon)z$, then 
$P=a(\epsilon)^{k+1}(z-1)^{k+1}$ and $\sqrt{\tilde Q}=a(\epsilon)^{k+\nu_s-\nu_a-1}z^{\nu_s-1}(z-1)^{k-\nu_a}\sqrt{U(a(\epsilon)z,\epsilon)}$ for some $U(x,\epsilon)$, $U(0,0)\neq0$,
\[a^{\nu}\tilde b(x,\epsilon)^{-1}\sim \tfrac{2\sqrt{U(0,0)}}{(\nu_s-1)(-1)^{\nu_a}}z^{\nu_s},\qquad \nu=\nu_a+1-\nu_s>0.\]
Hence the $z$-size of this connected component at $0$ of  
$\big\{|\tilde b|^{-1}\leq 2\big\}$ is asymptotically proportional to $|a(\epsilon)|^{\frac{\nu}{\nu_s}}\to 0$.
Therefore, expanding $a^{\nu}\bt_\epsilon(x)$ in the powers of $z$ on this component,
one obtains
\[\bt_\epsilon(x)\,\tilde b(x,\epsilon)\sim \frac{(\nu_s-1)}{2\nu_s},\]
i.e. the $\bt_\epsilon$-radius of this component tends to $\frac{|\nu_s-1|}{\nu_s}$ as $\epsilon\to 0$ and therefore is bounded.

Let us prove (ii). We want to estimate the $\bt_\epsilon$-lengths of the relative saddle connections in $\sX_\epsilon$.
The case with no saddles is the same as \cite[Lemma~4.7]{Hurtubise-Lambert-Rousseau}.
The proof can be adapted also for the second case
(cf. \cite[Lemma 6.29]{Klimes-Stolovitch}).
As before write $x=a(\epsilon)z$, thus
$\Delta_\epsilon^{\frac12}=a(\epsilon)^{-\nu}\frac{z^{\nu_s-1}\sqrt{U(a(\epsilon)z,\epsilon)}}{(z-1)^{\nu_a+1}}\d z$,
and let
\[W=\left\{z:\left|\tfrac{z^{\nu_s}}{(z-1)^{\nu_a+1}}\right|>C\right\},\qquad
\tilde W=\left\{z:\left|\tfrac{z^{\nu_s}}{(z-1)^{\nu_a+1}}\right|>\tfrac{C}{2}\right\},
\]
for some $0<C<1$ small enough.
A saddle connection $\gamma$ has to pass through $U$, which means it has to cross at least one of the two components of $\tilde W\smallsetminus W$, and we want to estimate the $\bt_\epsilon$-length of this crossing.
For $C$ small enough, 
\[\tilde W\smallsetminus W\sim \Big\{\left(\tfrac{C}{2}\right)^{\frac{1}{\nu_s}}<|z|\leq C^{\frac{1}{\nu_s}}\Big\}\cup
\Big\{\left(\tfrac{1}{C}\right)^{\frac{1}{\nu}}\leq|z|<\left(\tfrac{2}{C}\right)^{\frac{1}{\nu}}\Big\},\]
hence
\[\Big|\int_{\gamma\cap(\tilde W\smallsetminus W)}\Delta_\epsilon^{\frac12}\Big|\geq |a(\epsilon)|^{-\nu}K,\qquad \text{for}\quad K\sim\tfrac{C}{2}(1-2^{-\frac{1}{\nu_*}})\min_{\overline{\tilde W}}\sqrt{|U|},\quad \nu_*=\nu_s,\nu, \]
i.e. the $\bt_\epsilon$-length of the saddle connection grows like $|a(\epsilon)|^{-\nu}\to \infty$.
\end{proof}

By Proposition~\ref{proposition:mixedbasis}, $\sZ_{\epsilon}$ gives rise to a direct decomposition of the solution space into a pair of spaces of solutions subdominant along positive or negative half-trajectories.
To this decomposition we associate two different \emph{mixed solution bases}: they will differ in the ordering of the pair of solutions and in their normalization.

\begin{definition}[Mixed basis fundamental solution matrices]
Let us choose a branch of $\sqrt{\tilde Q}$ on $\sZ_{\epsilon}$ and hence orientation on the vector field $\e^{\i\vartheta}\frac{P(x,\epsilon)}{\sqrt{\tilde Q(x,\epsilon)}}\dd{x}$.
We define $Y_\sZ^+(x,\epsilon)$, resp. $Y_\sZ^-(x,\epsilon)$, as the mixed basis fundamental solution matrices (\S\,\ref{sec:subdominant}) that are asymptotic along positive, resp. negative, half-trajectories of the vector field to
\begin{equation}\label{eq:Y+-}
\begin{aligned}
Y_\sZ^+(x,\epsilon)
&\sim \tfrac{\i}{\sqrt 2} \tilde Q(x,\epsilon)^{-1/4} \begin{psmallmatrix}1&0\\ 0 &\sqrt{\tilde Q(x,\epsilon)}\end{psmallmatrix} \begin{psmallmatrix} 1 & 1 \\[4pt] 1& -1\end{psmallmatrix}
	\e^{-\int \frac{\sqrt{\tilde Q(x,\epsilon)}}{P(x,\epsilon)}\d x\begin{psmallmatrix}1&0\\[2pt]0&-1\end{psmallmatrix}},\\
Y_\sZ^-(x,\epsilon)
&\sim \tfrac{1}{\sqrt 2} \tilde Q(x,\epsilon)^{-1/4} \begin{psmallmatrix}1&0\\ 0 &-\sqrt{\tilde Q(x,\epsilon)}\end{psmallmatrix} \begin{psmallmatrix} 1 & 1 \\[4pt] 1& -1\end{psmallmatrix}
\e^{\int \frac{\sqrt{\tilde Q(x,\epsilon)}}{P(x,\epsilon)}\d x\begin{psmallmatrix}1&0\\[2pt]0&-1\end{psmallmatrix}}.		
\end{aligned}	
\end{equation}
In either case the first column is the normalized subdominant solution along the respective half-trajectory.
The two solution matrices  are related by an anti-diagonal matrix
\begin{equation}\label{eq:permutation}
Y_\sZ^+(x,\epsilon)=Y_\sZ^-(x,\epsilon) \begin{psmallmatrix}0 & \i\kappa_\sZ \\[4pt] \i\kappa_\sZ^{-1} & 0\end{psmallmatrix}.
\end{equation}
In the case when $\sZ$ is of sepal type and the same determination of $\int \frac{\sqrt{\tilde Q(x,\epsilon)}}{P(x,\epsilon)}\d x$ over $\sZ$ is chosen for $Y_\sZ^+$ and $Y_\sZ^-$, then $\kappa_\sZ=1$.
\end{definition}

\begin{remark}
An explicit form of the multiplier $\kappa_\sZ$ in \eqref{eq:permutation} is given in \cite[Proposition 3.19]{Klimes1} for general unfoldings in the case $k=1$, $m=1$. 
Another explicit example is the Gauss--Kummer formula for a determinant of certain mixed solution bases in a confluent family of hypergeometric equations, $k=1$, $m=0$, (see e.g. \cite{Klimes3}).
In both cases the zeros and poles of $\kappa_\sZ$ 
form a countable union of divisors which accumulates to the origin in a way that it asymptotically outlines the boundary of the maximal natural domain  
in the parameter space on which the mixed basis exists. We conjecture that this is true in general.
\end{remark}

\begin{figure}[t]
	\centering
	\begin{subfigure}[t]{0.49\textwidth}
		\includegraphics [width=\textwidth]{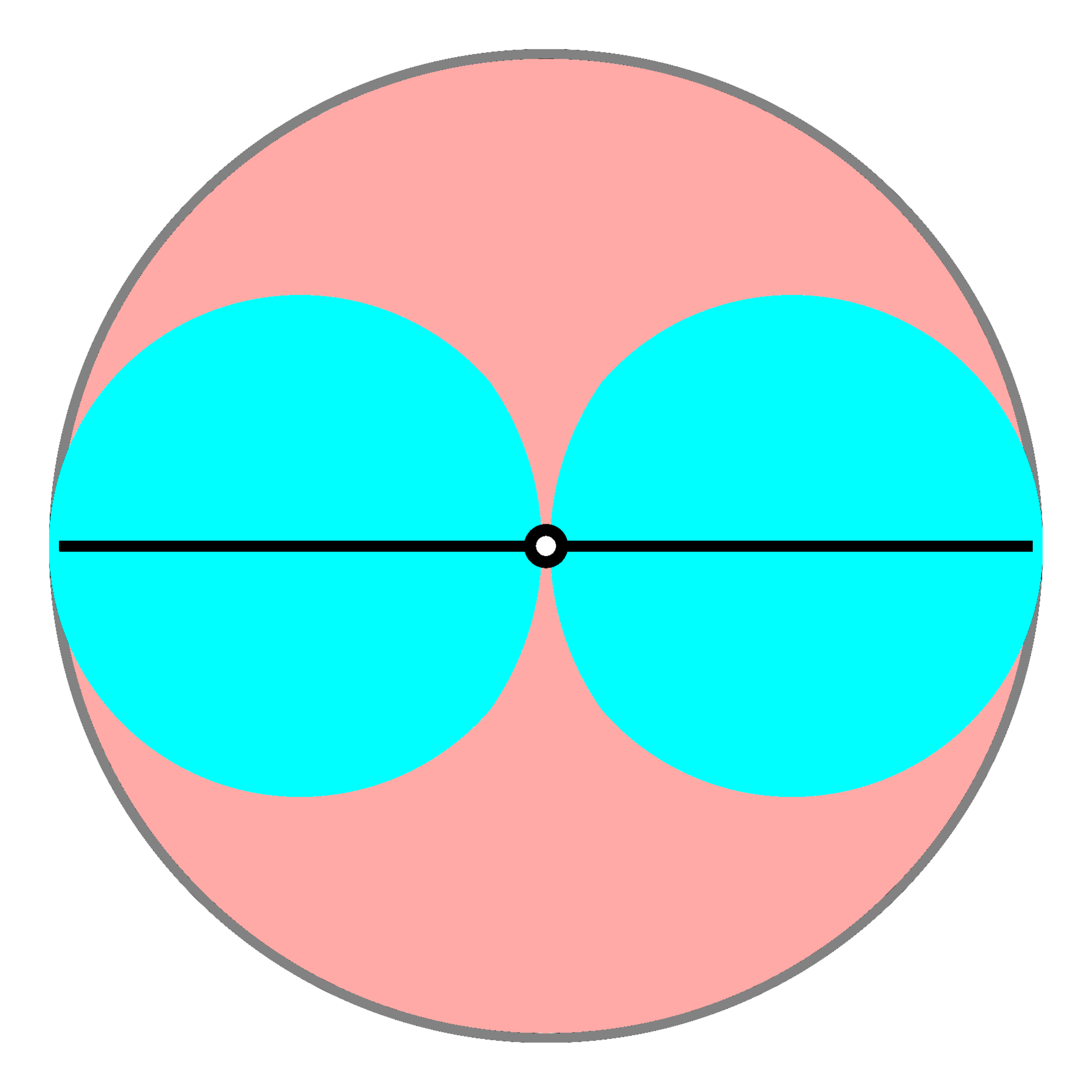}
		\caption{$\Delta_0=\frac{c}{x^4}(\d x)^2$}		\label{figure:lagoons2A}
	\end{subfigure}
	\begin{subfigure}[t]{0.49\textwidth}
		\includegraphics [width=\textwidth]{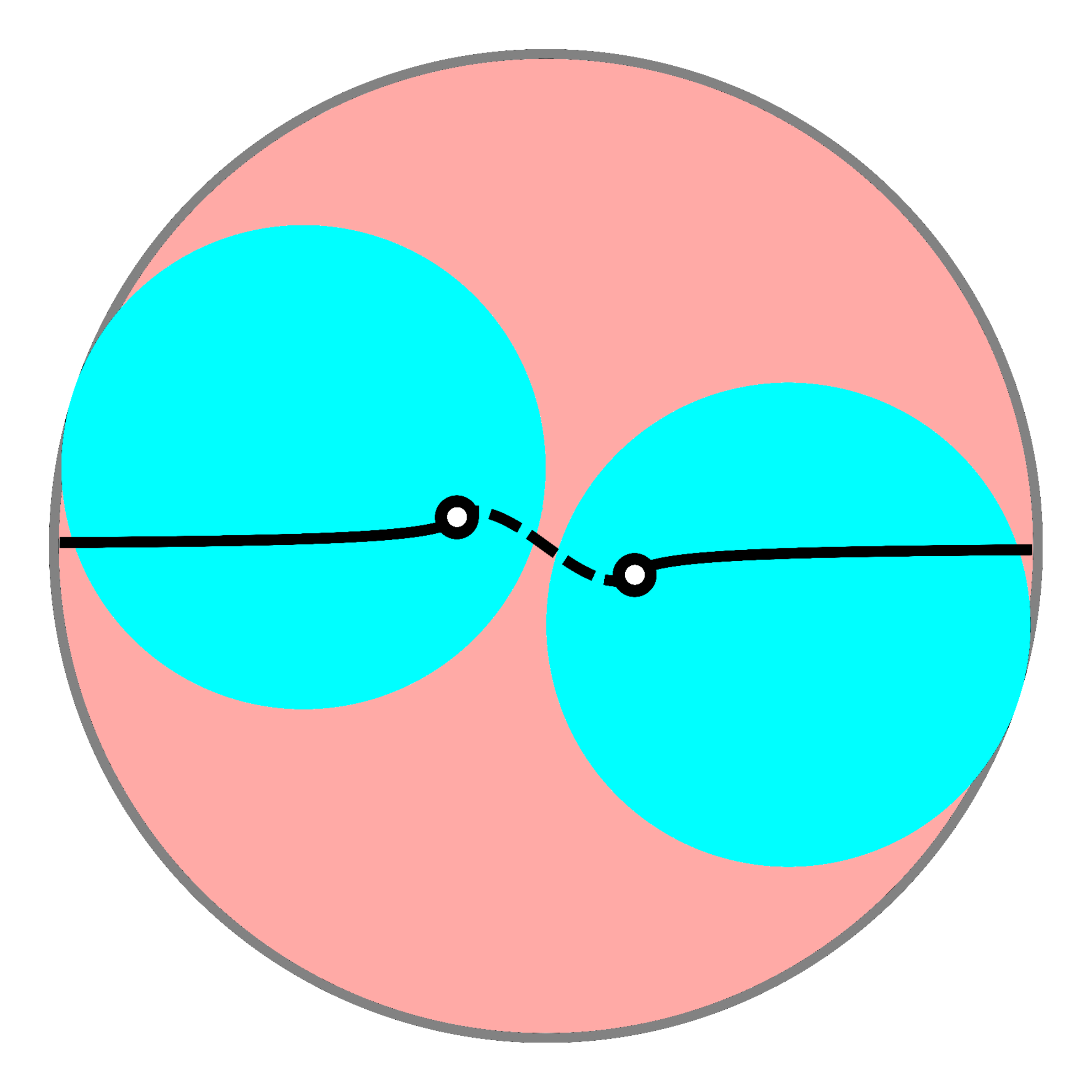}
		\caption{$\Delta_\epsilon=\frac{c}{(x-\epsilon)^2(x+\epsilon)^2}(\d x)^2$}		\label{figure:lagoons2B}
	\end{subfigure}
\caption{Lagoons (turquoise) are the connected components of the complement the fat extended transverse graph (pink). The separatrices (thick black) of $\Delta_\epsilon$ for reference.}
	\label{figure:lagoons2}
\end{figure}

\begin{figure}[t]
	\centering
	\begin{subfigure}[t]{0.49\textwidth}
		\includegraphics [width=\textwidth]{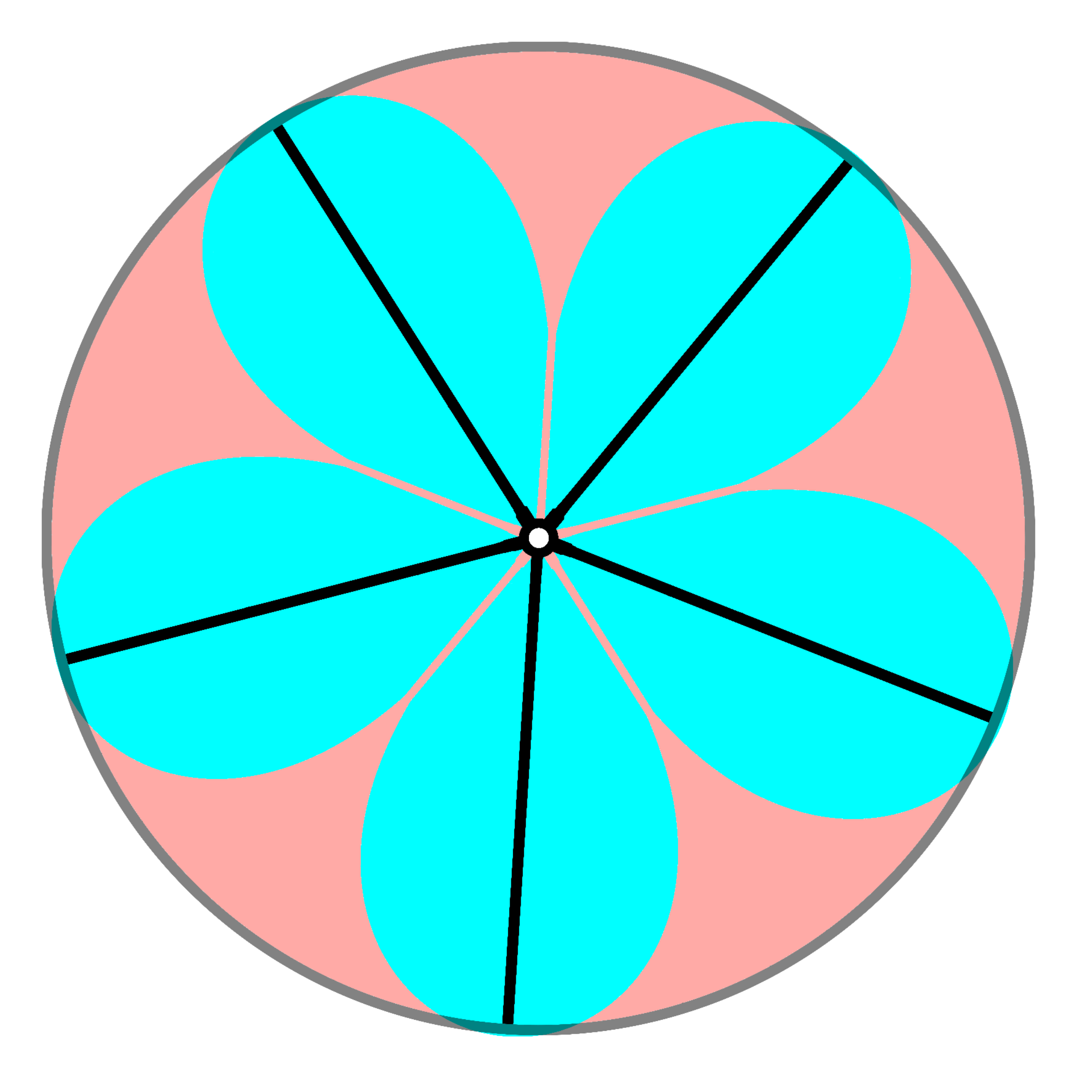}
		\caption{$\Delta_0=\frac{c}{x^7}(\d x)^2$}		\label{figure:lagoonsA}
	\end{subfigure}
	\begin{subfigure}[t]{0.49\textwidth}
		\includegraphics [width=\textwidth]{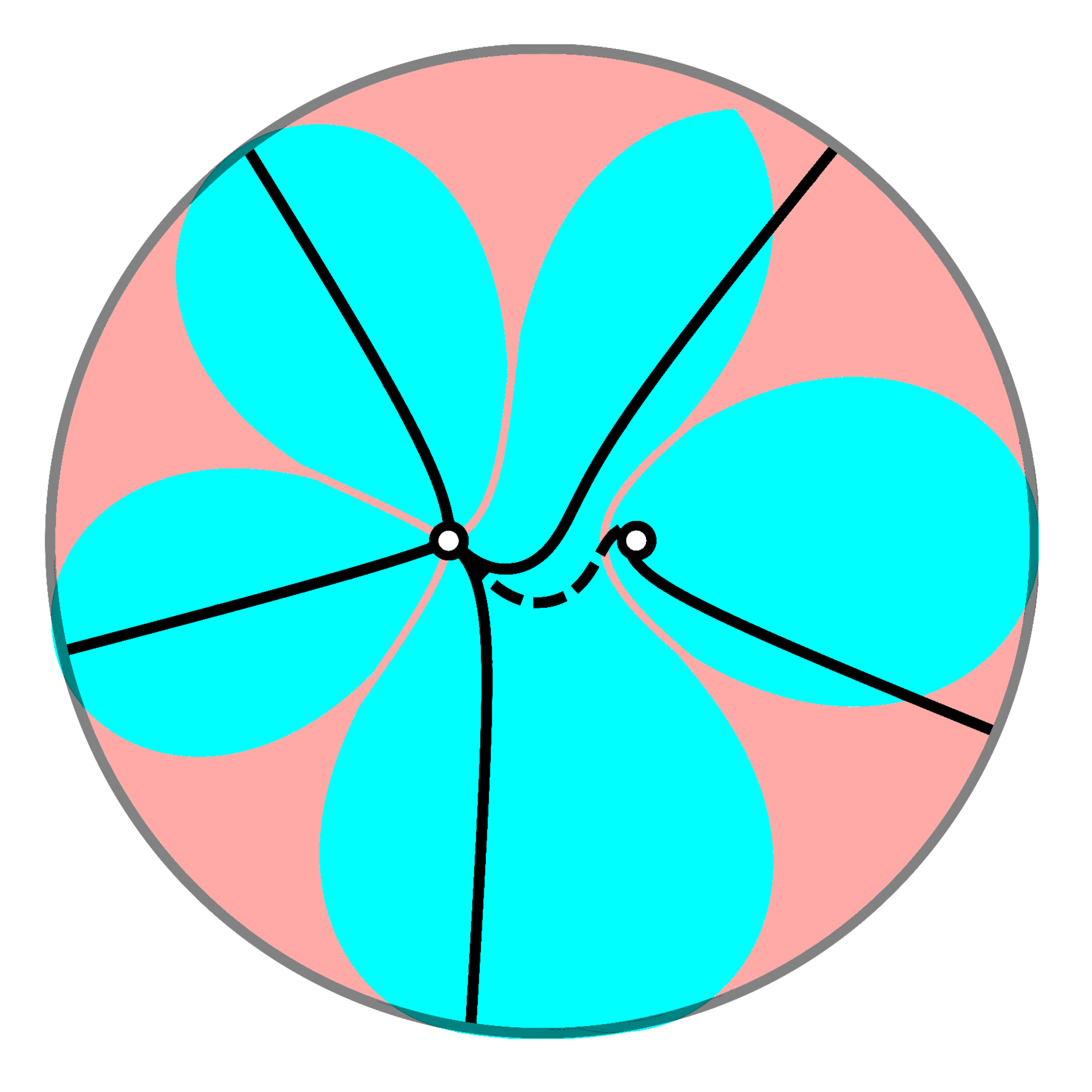}
		\caption{$\Delta_\epsilon=\frac{c}{(x-\epsilon)^2(x+\epsilon)^5}(\d x)^2$}		\label{figure:lagoonsB}
	\end{subfigure}
	
	\begin{subfigure}[t]{0.49\textwidth}
		\includegraphics [width=\textwidth]{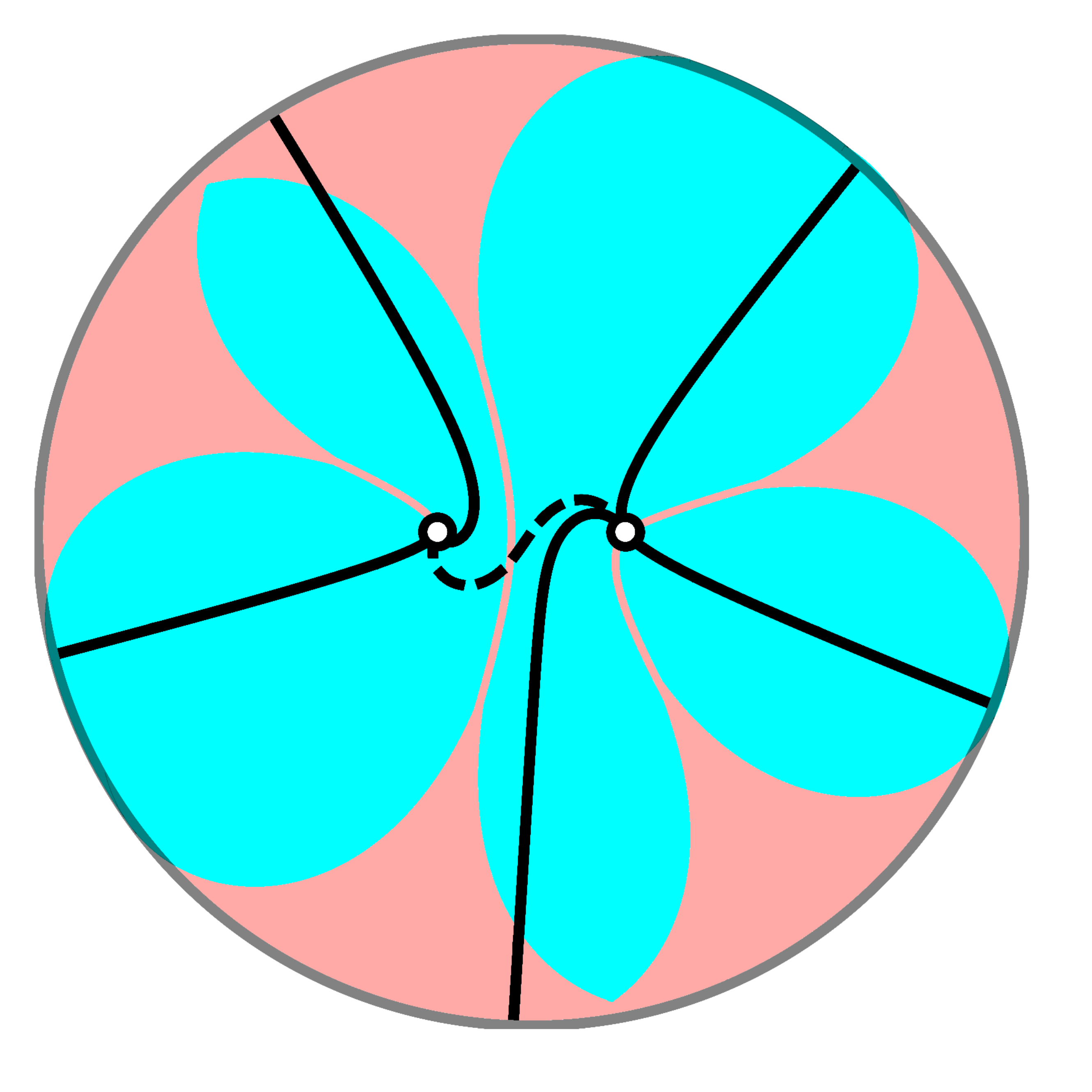}
		\caption{$\Delta_\epsilon=\frac{c}{(x-\epsilon)^3(x+\epsilon)^4}(\d x)^2$}		\label{figure:lagoonsC}
	\end{subfigure}
	\begin{subfigure}[t]{0.49\textwidth}
		\includegraphics [width=\textwidth]{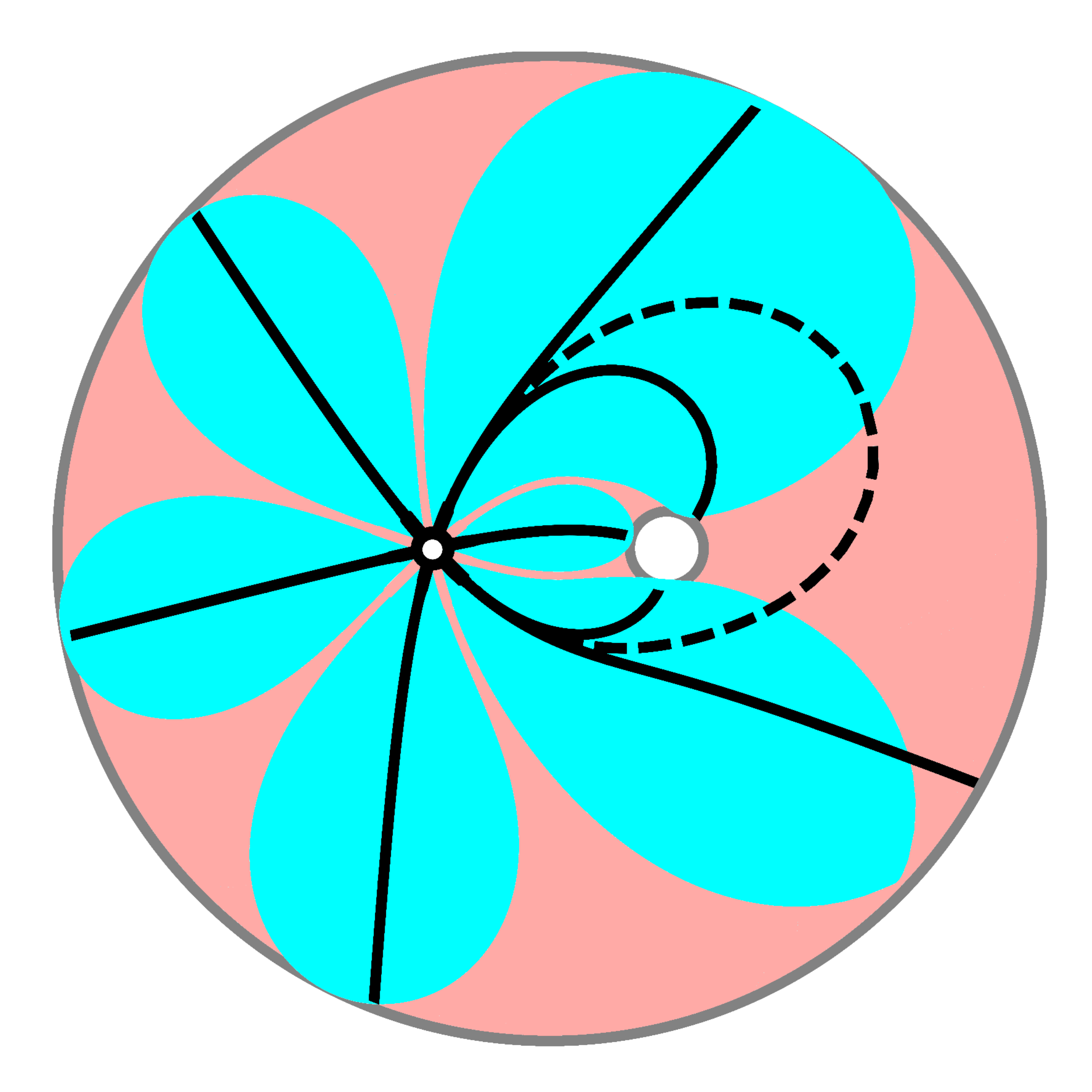}
		\caption{$\Delta_\epsilon=\frac{c(x-\epsilon)}{(x+\epsilon)^8}(\d x)^2$}		\label{figure:lagoonsD}
	\end{subfigure}
\caption{Lagoons (turquoise) are the connected components of the complement the  fat extended transverse graph (pink). The separatrices (thick black) of $\Delta_\epsilon$ for reference.}
	\label{figure:lagoons}
\end{figure}

\subsubsection{Lagoons and flag structure on the solution space}

\begin{definition}[Lagoons]
Let $\sX_\epsilon$ \eqref{eq:sX} and $\sW_\epsilon$ \eqref{eq:sW} be as above, $\vartheta\in\sW_\epsilon$. 
We define \emph{fat extended transverse graph} as the union of:
\begin{itemize}[leftmargin=2\parindent]
	\item for each $\alpha\omega$-zone of $\e^{-2\i \vartheta}\Delta_\epsilon$ relative to $\sX_\epsilon$ take its fat transversal (see Definition~\ref{def:relativezones}),
	\item for each  sepal zone of $\e^{-2\i \vartheta}\Delta_\epsilon$ relative to $\sX_\epsilon$ take the separating cusp of the vertical foliation of $\Delta_\epsilon$ which passes through the zone.
\end{itemize}
See Figure \ref{figure:lagoons}.	

The connected components of the complement of the fat extended transverse graph in $\sX_\epsilon$ will be called \emph{lagoons}.

Each lagoon is attached to an equilibrium point: either as an open neighborhood of a simple equilibrium or as a sepal at a parabolic equilibrium of the vertical foliation.
The boundary of each lagoon is transverse to the trajectories of  $\e^{-2\i\vartheta}\Delta_\epsilon(x)$ for all $\vartheta\in\sW_\epsilon$. Any trajectory of $\e^{-2\i\vartheta}\Delta_\epsilon(x)$ that enters the lagoon stays in the lagoon and will asymptotically tend to the equilibrium.
\end{definition}

By means of Theorem~\ref{prop:subdominant}, each lagoon gives rise to a space of subdominant solutions along trajectories of $\e^{-2\i\vartheta}\Delta_\epsilon(x)$, $\vartheta\in\sW_\epsilon$. Therefore the solution space above it is endowed with a complete flag filtration. 
Geometrically this means that the automorphism group $\SL_2(\C)$ of the solution space is restricted above each lagoon to the Borel subgroup preserving the filtration.

\begin{remark}
The complement of the lagoon of a regular singularity is a surface with a simple hole.
The complement of the $2\nu_i$ lagoons of an irregular singularity of Katz rank $\nu_i=k_i-\frac{m_i}2>0$ is a surface with a topological hole
with $2\nu_i$ cusps on its boundary corresponding to the Stokes rays (the separating directions of the vertical foliation).
This is similar to the pictures of L.~Chekhov, M.~Mazzocco and V.~Rubtsov \cite{Chekhov-Mazzocco, Chekhov-Mazzocco-Rubtsov}.
Also the confluence procedure of the lagoons is like the one described in \cite{Chekhov-Mazzocco, Chekhov-Mazzocco-Rubtsov}:
	\begin{enumerate}[label=(\arabic*), leftmargin=2\parindent]
		\item \emph{Hole hooking}: confluence of two singularities of ranks $\nu_1,\nu_2\geq0$,
		\[\Delta_\epsilon(x)=\frac{U(x,\epsilon)^2(\d x)^2}{(x+\epsilon)^{2\nu_1+2}(x-\epsilon)^{2\nu_2+2}},\quad U(0,0)\neq 0.\]
		 There is a unique gate connecting the two singularities and a unique associated fat transversal; during the confluence the fat transversal is stretched (its period \eqref{eq:period} tends to $\infty$) and broken, creating two new cusps. 
		 See Figures~\ref{figure:lagoons2B}, \ref{figure:lagoonsB},  \ref{figure:lagoonsC}.
		\item \emph{Cusps removal}: confluence of an irregular singularity of rank $\nu_a\geq 0$ and a saddle point of rank $\nu_s\leq 0$, with $\nu=\nu_a+\nu_s+1>0$:
		\[\Delta_\epsilon(x)=\frac{(x+\epsilon)^{-2\nu_s-2}U(x,\epsilon)^2(\d x)^2}{(x-\epsilon)^{2\nu_a+2}},\quad U(0,0)\neq 0.\]
		From the hole in $\sX_\epsilon$ around the saddle point emanate $2|\nu_s|$ separating regions: of which $2|\nu_s|-1$ extend as separating cusps toward the singularity, while the last one, a fat transversal of an $\alpha\omega$-zone, extends towards the outer boundary of $\sX_\epsilon$. As the saddle point merges with the singularity, the $2|\nu_s|-1$ cusps shrink together and disappear, while the fat transversal forms a single new cusp at the limit.
		The generic situation is with $\nu_s=-\frac32$, when 2 existing cusps disappear and 1 new one is created, see Figure~\ref{figure:lagoonsD}.
		The case $\nu_s=-\frac12$, in which the total number of cusps would increase by one, is at odds with the assumption \eqref{eq:ass}.
	\end{enumerate}
\end{remark}

\subsection{Wild monodromy representations}\label{sec:wildmonodromy}

\subsubsection{Fundamental groupoid}

We now consider all the different petals of the horizontal foliation of $\e^{-2\i\vartheta}\Delta_\epsilon(x)$ for $(\vartheta,\epsilon)\in \sW$,
along with their associated fundamental solution matrices.
Our focus will be on the various connection matrices that relate these fundamental solution matrices.
Following the approach of E.~Paul and J.P.~Ramis \cite{Paul-Ramis, Paul-Ramis1}, we interpret the connection matrices as representations of a fundamental groupoid. 
The concept of "wild fundamental groupoids" has appeared previously in the works of J.~Martinet and J.P.~Ramis \cite{Martinet-Ramis, Deligne-Malgrange-Ramis} and P.~Boalch \cite{Boalch14, Boalch18}.
However, in our treatment, we neither blow up irregular singularities nor attach infinitesimal halos to them.
For our purposes, the fundamental groupoid is to be
$\Pi_1(\sX\smallsetminus\Crit_\epsilon, \sD_{\sW_\epsilon})$, where the base-points $\sD_{\sW_\epsilon}$ 
correspond to the intersections of petals and lagoons.
Each petal thus contains precisely two base points, one associated with each of the two fundamental solution matrices \eqref{eq:Y+-}.
In the confluent setting it depends on the choice of $\sW_\epsilon$ \eqref{eq:sW}. 
On the other hand, if considered for an individual irregular singularity, or for a parametric family of such singularities with constant Katz rank, then there is only one $\sW_\epsilon=\ ]\!-\!\frac{\pi}{2},\frac{\pi}{2}[$, so it doesn't play a role.

\begin{definition}[Fundamental groupoid]
Let $\epsilon\in\proj_\epsilon(\sW)$. Each enlarged petal $\sZ_{\epsilon}$ with respect to $\sX_\epsilon$ intersects two different lagoons (corresponding to the positive and negative half-trajectories).
Chose a basepoint $d_\sZ^\pm$ in each of the two intersection.
The notation is such that the fundamental solution $Y_\sZ^\pm(x,\epsilon)$ \eqref{eq:Y+-} is the one 
whose first column is subdominant on the intersection containing  $d_\sZ^\pm$.
Let $\sD_{\sW_\epsilon}=\{d_\sZ^+,\ d_\sZ^-\mid \sZ_{\epsilon}\ \text{petal}\}$ be the set of these basepoints.
The \emph{fundamental groupoid} is a small category whose
\begin{itemize}
	\item set of objects $\sD_{\sW_\epsilon}$ is the set of basepoints,
	\item set of morphisms
	\[\Pi_1(\sX\smallsetminus\Crit_\epsilon, \sD_{\sW_\epsilon})\]
	is the set of homotopy classes of paths in $\sX\smallsetminus\Crit_\epsilon$  with the operation of concatenation of paths:
	\[
	\begin{tikzcd}
	d_1
	\arrow[r,"\gamma_{12}"]
	\arrow[rr,bend left=45,"\gamma_{12}\gamma_{23}"]
	& d_2
	\arrow[r,"\gamma_{23}"]
	& d_3
	\end{tikzcd}.
	\]
\end{itemize}

We draw it schematically by choosing $\vartheta\in\sW_\epsilon$ and representing respectively
\begin{enumerate}[label=\alph*), leftmargin=2\parindent]
	\item[a)] the fat separating graph of $\e^{-2\i\vartheta}\Delta_\epsilon(x)$ in $\sX_\epsilon$,
	\item[b)] the gate graph of $\e^{-2\i\vartheta}\Delta_\epsilon(x)$ in $\sX_\epsilon$,
	\item[c)] the fat extended transversal graph of $\e^{-2\i\vartheta}\Delta_\epsilon(x)$ in $\sX_\epsilon$,
\end{enumerate}
by their non-fat versions (i.e. by choosing a particular trajectory to represent in each part of the fat graph), and by taking one base-point $d_\sZ^\pm$ in each connected component in $\sX\smallsetminus\Crit_\epsilon$ of the complement of the three graphs.
See Figures~\ref{figure:wild2},~\ref{figure:wild}.
In this presentation the only intersection is between the gates and the transversals of $\alpha\omega$-zones.

There are three kinds of \emph{elementary paths} which generate the groupoid:
\begin{enumerate}[label=\alph*), leftmargin=2\parindent]
	\item[a)] those that cross the separating graph: stay inside the same lagoon,
	\item[b)] those that cross the gate graph: stay inside the same $\alpha\omega$-zone and the same lagoon,
	\item[c)] those that cross the extended transversal graph: stay inside the same petal.
\end{enumerate}
\end{definition}

For $(\vartheta,\epsilon)=(0,0)$ the separating graph represents the anti-Stokes (steepest descent) curves, and the extended transverse graph represent the Stokes (oscillation) curves.

\begin{figure}[t]
	\centering
	\begin{subfigure}[t]{0.49\textwidth}
		\includegraphics [width=\textwidth]{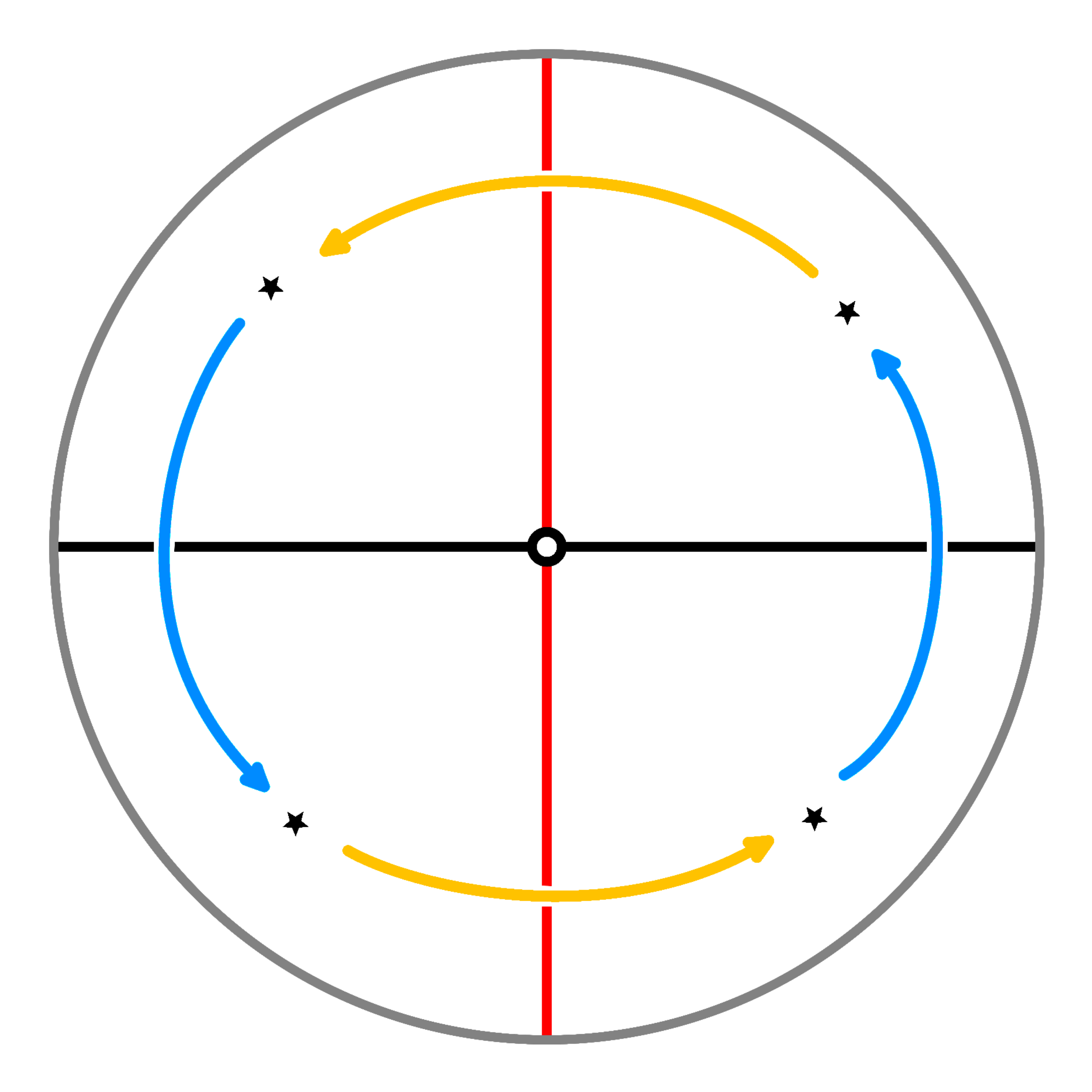}
		\caption{$\Delta_0=\frac{c}{x^4}(\d x)^2$}		\label{figure:wild2A}
	\end{subfigure}
	\begin{subfigure}[t]{0.49\textwidth}
		\includegraphics [width=\textwidth]{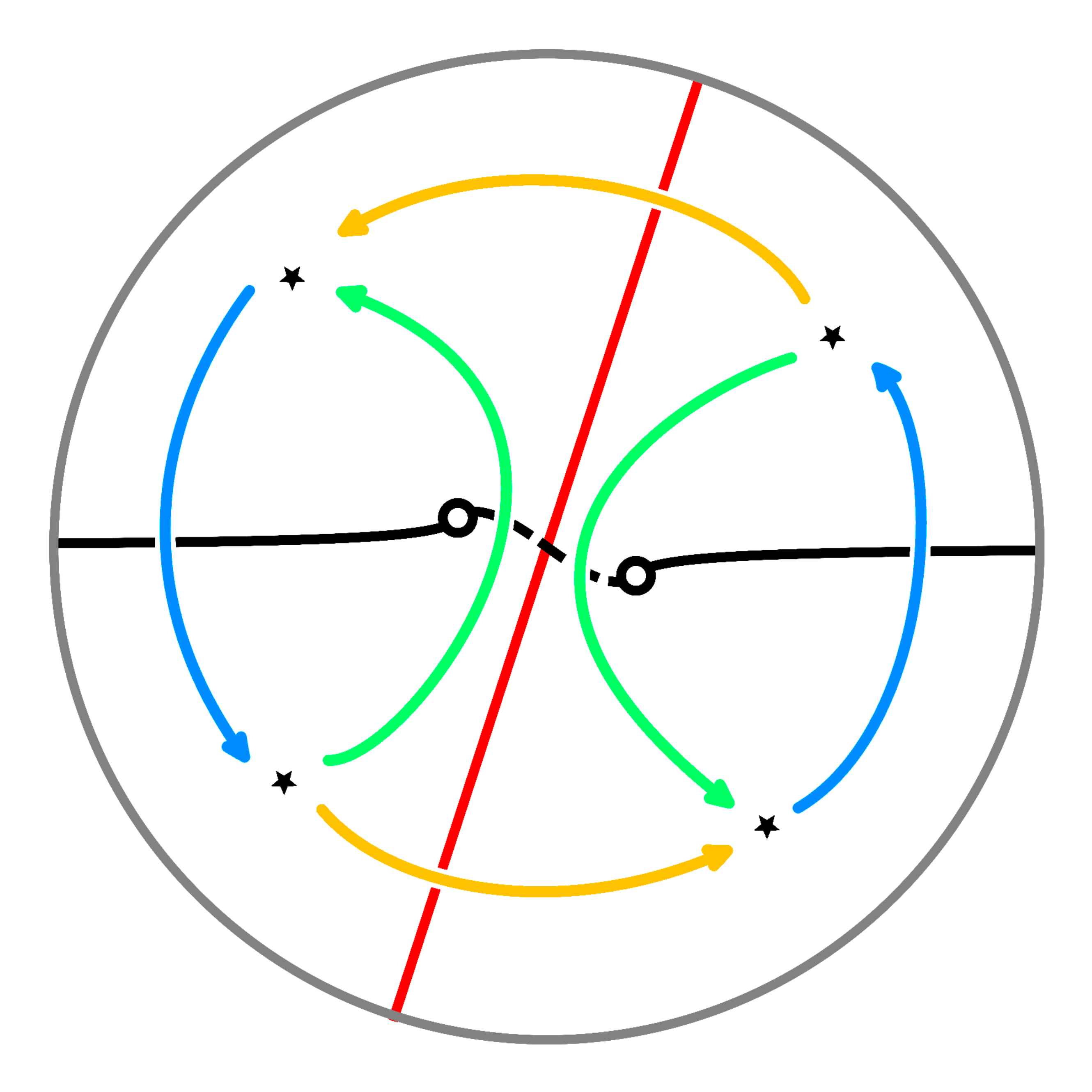}
		\caption{$\Delta_\epsilon=\frac{c}{(x-\epsilon)^2(x+\epsilon)^2}(\d x)^2$}		\label{figure:wild2B}
	\end{subfigure}
	\caption{Schematic presentation of the fundamental groupoid associated with $\Delta_\epsilon$.
		The elementary paths in a) blue, b) green, c) yellow crossing a) separating graph  in black, b) gate graph dashed, c) extended transverse graph in red.}
	\label{figure:wild2}
\end{figure}

\begin{figure}[t]
	\centering
	\begin{subfigure}[t]{0.49\textwidth}
		\includegraphics [width=\textwidth]{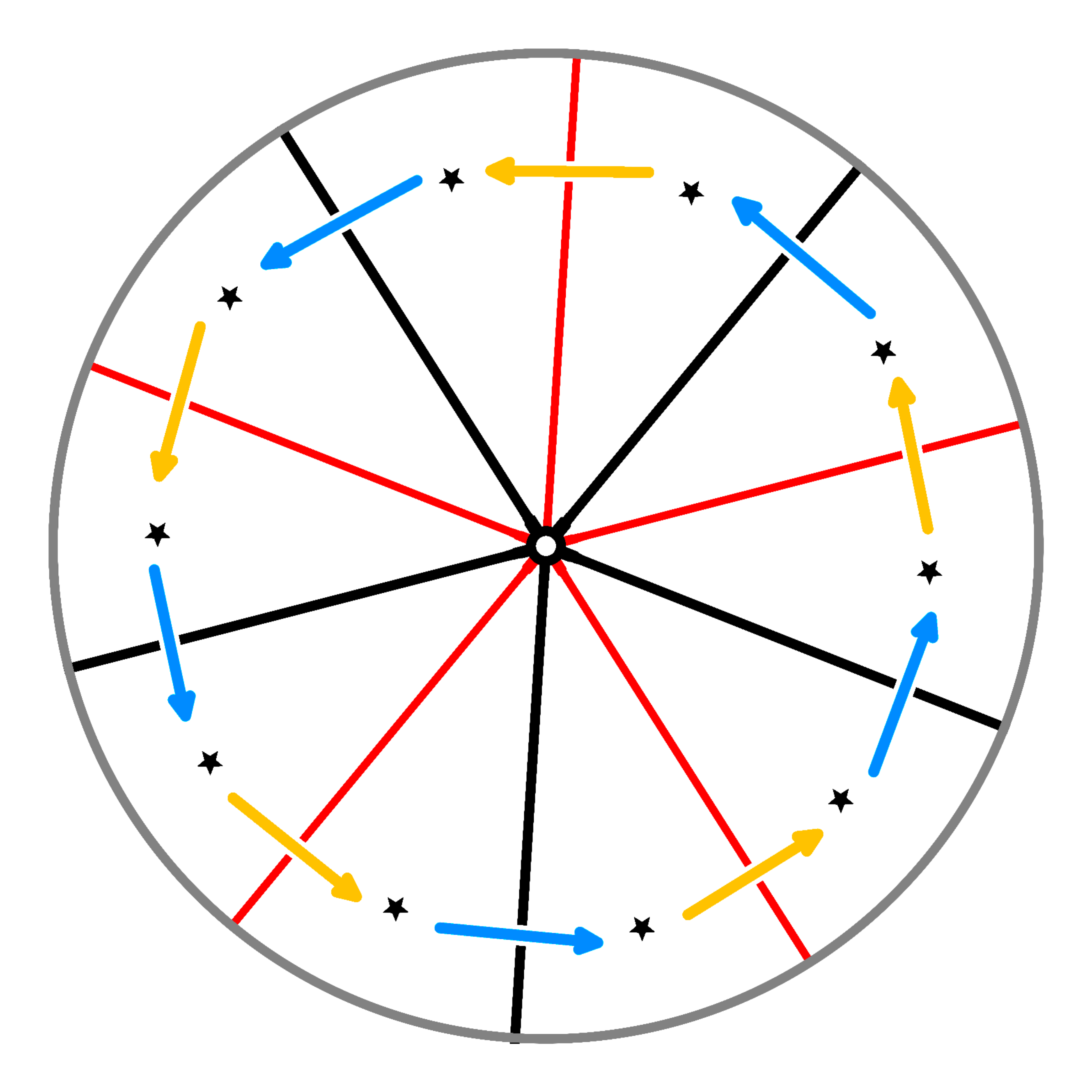}
		\caption{$\Delta_0=\frac{c}{x^7}(\d x)^2$}	
	\end{subfigure}
	\begin{subfigure}[t]{0.49\textwidth}
		\includegraphics [width=\textwidth]{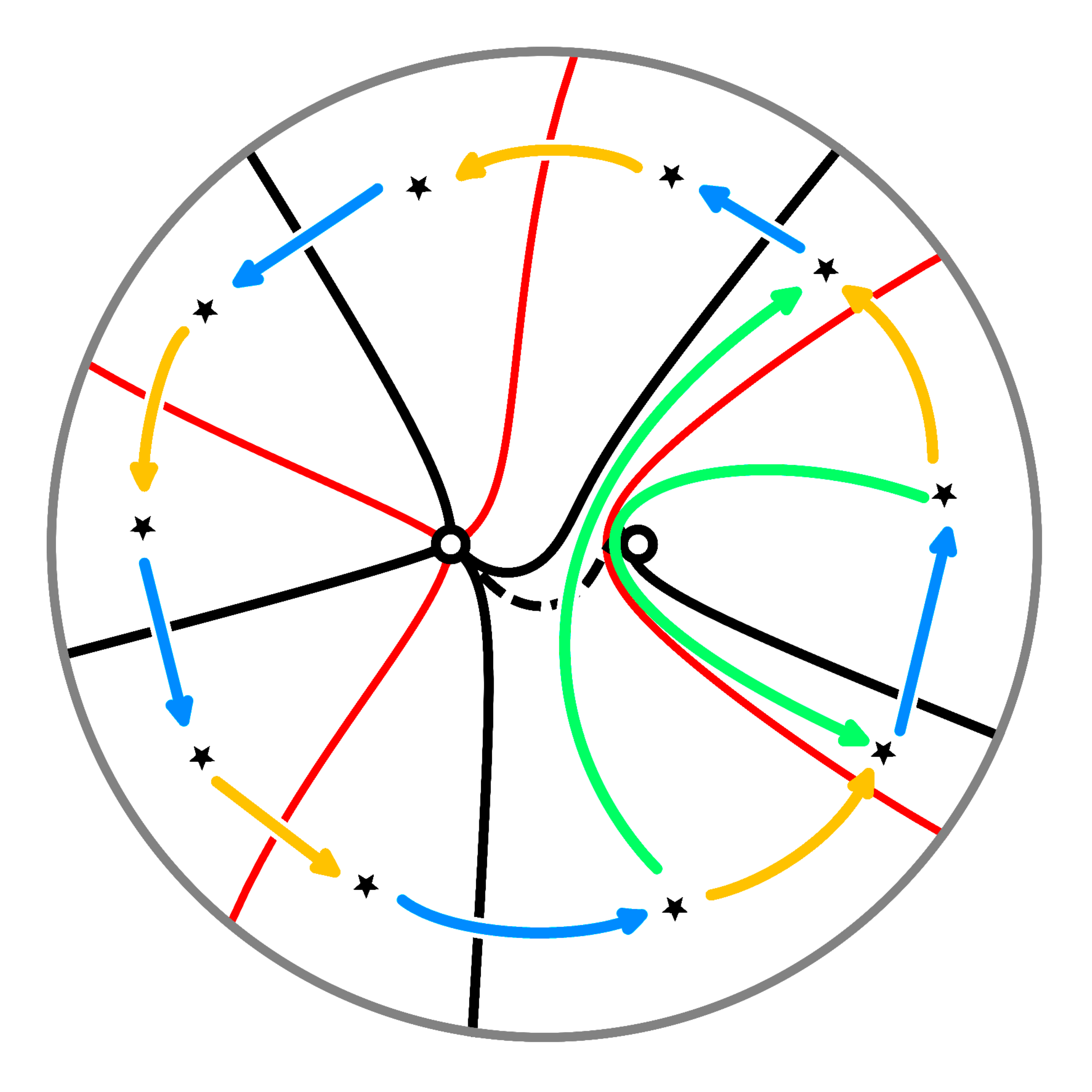}
		\caption{$\Delta_\epsilon=\frac{c}{(x-\epsilon)^2(x+\epsilon)^5}(\d x)^2$}	
	\end{subfigure}
	
	\begin{subfigure}[t]{0.49\textwidth}
		\includegraphics [width=\textwidth]{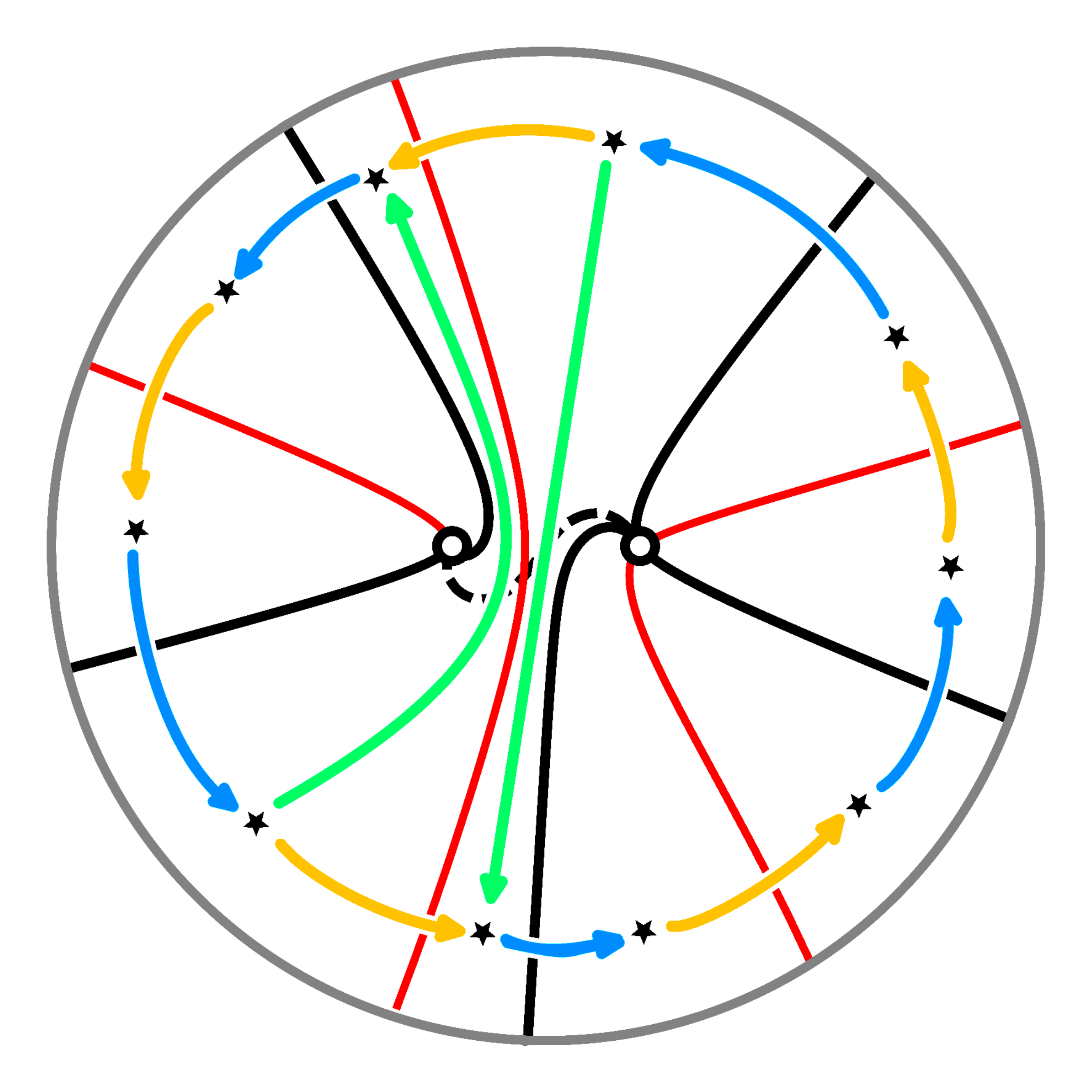}
		\caption{$\Delta_\epsilon=\frac{c}{(x-\epsilon)^3(x+\epsilon)^4}(\d x)^2$}	
	\end{subfigure}
	\begin{subfigure}[t]{0.49\textwidth}
		\includegraphics [width=\textwidth]{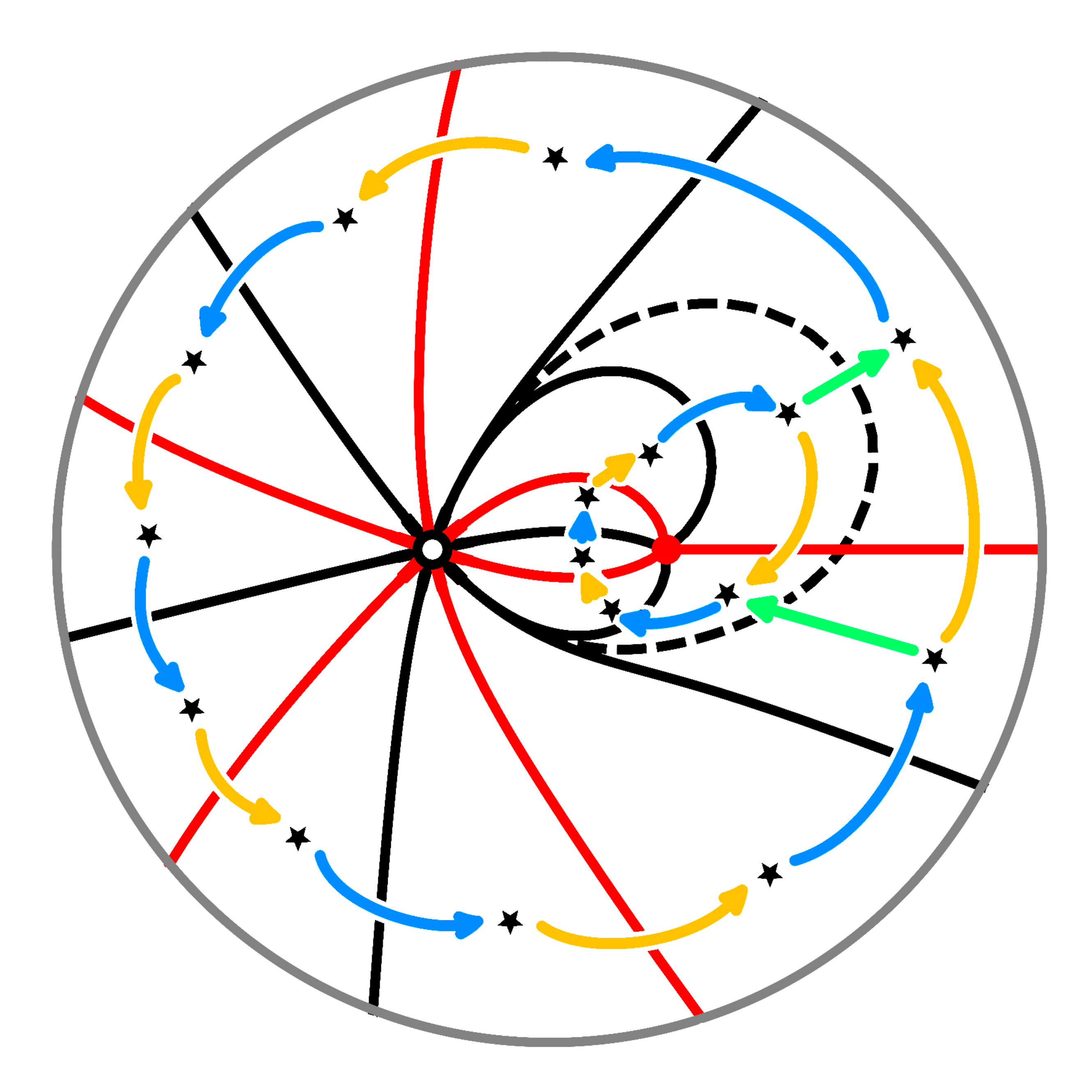}
		\caption{$\Delta_\epsilon=\frac{c(x-\epsilon)}{(x+\epsilon)^8}(\d x)^2$}		\label{figure:wildD}
	\end{subfigure}
	\caption{Schematic presentation of the fundamental groupoid associated with $\Delta_\epsilon$.
		The elementary paths in a) blue, b) green, c) yellow crossing a) separating graph  in black, b) gate graph dashed, c) extended transverse graph in red.}
	\label{figure:wild}
\end{figure}

\subsubsection{Wild monodromy representations.}

Denote:
\begin{itemize}[leftmargin=2\parindent]
	\item[--] $\bT$ the standard Cartan subgroup of $\SL_2(\C)$ consisting of diagonal matrices,
	\item[--] $\rotT=\bT\cdot\begin{psmallmatrix}0&\i\\ \i&0\end{psmallmatrix}$ the set of anti-diagonal matrices,
	\item[--] $\bN=\bT\cup\rotT$ the  normalizer of $\bT$,
	\item[--] $\bB$ the Borel subgroup of upper-triangular matrices,
	\item[--] $\bU$ the subgroup of $\bB$ consisting of unipotent upper-triangular matrices.
\end{itemize}

\begin{definition}\label{def:wildmonodromy}
A \emph{wild monodromy representation of the fundamental groupoid}
is a contravariant functor
\[\rho:\Pi_1(\sX\smallsetminus\Crit_\epsilon, \sD_{\sW_\epsilon})\to\SL_2(\C)\] 
which
\begin{enumerate}[leftmargin=\parindent]
	\item associates to each basepoint $d\in\sD_{\sW_\epsilon}$ the triple $\big(\C^2,\ \textbf{flag},\  \textbf{splitting}\big)$, where
	\[\textbf{flag}=\left(\{0\}\subset\langle\begin{psmallmatrix}1\\0\end{psmallmatrix}\rangle\subset\C^2\right),\qquad \textbf{splitting}=\langle\begin{psmallmatrix}1\\0\end{psmallmatrix}\rangle\oplus\langle\begin{psmallmatrix}0\\1\end{psmallmatrix}\rangle\]
	are the standard flag  and splitting structures associated to the groups $\bB$ and $\bT$,
	\item represents the elementary paths by elements of
	\begin{enumerate}[label=\alph*), leftmargin=2\parindent]
		\item[a)] $\bB$ for those that cross the separating graph: preserve flags,
		\item[b)] $\bT$ for those that cross the gate graph: preserve both flags and splittings,
		\item[c)] $\rotT$ for those that cross the extended transverse graph: preserve splittings and make flags transverse,
	\end{enumerate}
	\item factors through the projection 
	\[\Pi_1(\sX\smallsetminus\Crit_\epsilon, \sD_{\sW_\epsilon})\to \Pi_1(\sX\smallsetminus\Sing_\epsilon, \sD_{\sW_\epsilon}),\] 
	meaning that for each saddle point of rank $\nu$ the product of the $4|\nu|$ alternating upper-triangular and anti-diagonal matrices along a loop around the saddle is the identity.
\end{enumerate}
Two representations  $\rho,\rho'$ are \emph{equivalent},  $\rho\sim\rho'$,  if there exists a map (natural transformation)
$\tau\in\bT^{\sD_{\sW_\epsilon}}$ from the set of base-points $\sD_{\sW_\epsilon}$ to the group $\bT$ of automorphisms of $\big(\C^2,\ \textbf{flag},\  \textbf{splitting}\big)$,
that conjugates the two representations: for each path $d_1\xrightarrow{\gamma} d_2$ one asks that $\tau(d_2)\cdot\rho(\gamma)=\rho'(\gamma)\cdot\tau(d_1)$. 
\end{definition}


\begin{definition}
Denote $\tilde\sX\smallsetminus\Crit_\epsilon$ the 1- or 2-sheeted covering of $\sX\smallsetminus\Crit_\epsilon$ on which the form  $\sqrt{\Delta_\epsilon}$ is defined.
Given a wild monodromy representation $\rho$, we will associate to it its \emph{formalization} 
\[\rho^\natural:\Pi_1(\tilde\sX\smallsetminus\Crit_\epsilon, \tilde\sD_{\sW_\epsilon})\to \bN,\] 
which to every lifted elementary path $\gamma$ of type a) associates $\rho^\natural(\gamma)=\diag(\rho(\gamma))$ the diagonal of $\rho(\gamma)\in \bB$ and
to elementary paths $\gamma$ of type b) \& c) associates $\rho^\natural(\gamma)=\rho(\gamma)\in \bN$.

The wild monodromy representation $\rho$ is \emph{compatible} with the quadratic differential $\Delta_\epsilon$ if for  any closed loop $\gamma\in \Pi_1(\tilde\sX\smallsetminus\Crit_\epsilon, \tilde\sD_{\sW_\epsilon})$ based at some $d\in \tilde\sD_{\sW_\epsilon}$
the $(1,1)$-element $[\rho^\natural(\gamma)]_{1,1}$ of $\rho^\natural(\gamma)$ is equal to
\begin{equation}\label{eq:compatible}
	[\rho^\natural(\gamma)]_{1,1}=\e^{\pm\int_{\gamma}\sqrt\Delta_\epsilon},
\end{equation}
where the branch $\e^{\pm\int\sqrt\Delta_\epsilon}$ is the one with subdominant behavior in the lifted lagoon containing $d$.

For any path $\gamma\in\Pi_1(\sX\smallsetminus\Crit_\epsilon, \sD_{\sW_\epsilon})$, the action of the equivalence of representations on the product $\big(\rho^\natural(\gamma)\big)^{-1}\cdot\rho(\gamma)$ is by conjugation with elements of $\bT$, therefore its diagonal \[\diag\left(\big(\rho^\natural(\gamma)\big)^{-1}\cdot\rho(\gamma)\right)\]
is an \emph{invariant} of the wild monodromy representation $\rho$.
\end{definition}


Denote
$\Rep_{\Delta_\epsilon,\sW_\epsilon}$
the space of wild monodromy representations $\Pi_1(\sX\smallsetminus\Crit_\epsilon, \sD_{\sW_\epsilon})\to \SL_2(\C)$ compatible with $\Delta_\epsilon$.
The space of their equivalence classes is the quotient 
\[\Mon_{\Delta_\epsilon,\sW_\epsilon}= \left.\raisebox{.2em}{$\Rep_{\Delta_\epsilon,\sW_\epsilon}$}\Big/\raisebox{-.2em}{$\bT^{\sD_{\sW_\epsilon}}$}\right..\]
The \emph{local wild character variety} is the GIT-version of this quotient, 
\[\Char_{\Delta_\epsilon,\sW_\epsilon}=\left.\raisebox{.2em}{$\Rep_{\Delta_\epsilon,\sW_\epsilon}$}\Big/\!\!\!\Big/\raisebox{-.2em}{$\bT^{\sD_{\sW_\epsilon}}$}\right.,\]
i.e. the affine variety associated to the ring of functions on $\Rep_{\Delta_\epsilon,\sW_\epsilon}$ invariant by the equivalence relation.

\begin{remark}\label{remark:flag}
	Another way to define the moduli space $\Mon_{\Delta_\epsilon,\sW_\epsilon}$ would be through representations of
	$\Pi_1(\sX\smallsetminus\Crit_\epsilon, \sD_{\sW_\epsilon})$ respecting the flag structure only:
	\begin{itemize}[leftmargin=2\parindent]
		\item[--] To each basepoint from $\sD_{\sW_\epsilon}$ attach the pair $\big(\C^2,\ \textbf{flag}\big)$,
		\item[--] To each path from $\Pi_1(\sX\smallsetminus\Crit_\epsilon, \sD_{\sW_\epsilon})$ attach linear map from $\SL_2(\C)$ between the source and target spaces, such that
		\begin{itemize}
			\item if the path is elementary of type a) or b), i.e. crosses the extended separating graph (union of separating and gate graphs), then it maps between the two flags, i.e. it is an element of $\bB$,
			\item if the path is elementary of type c), i.e. crosses the extended transverse graph, then it makes the two flags transverse, i.e. it is an element of 
			$\SL_2(\C)\smallsetminus\bB=\bB\cdot \begin{psmallmatrix}0&\i\\ \i&0\end{psmallmatrix}\cdot\bB=\bU\cdot \rotT\cdot\bU$, 
		\end{itemize}
		and such that is factors as a representation of $\Pi_1(\sX\smallsetminus\Sing_\epsilon, \sD_{\sW_\epsilon})\to\SL_2$. 
		\item[--] An equivalence relation is a linear automorphism of each of the vector spaces preserving the flag, i.e. an elements of $\bB^{\sD_{\sW_\epsilon}}$.
	\end{itemize}
	The source and targets of the elementary paths of type c) provide a partitioning of $\sD_{\sW_\epsilon}$ in pairs of points, and every element of 	$\SL_2(\C)\smallsetminus\bB$ has a unique decomposition as a product in $\bU\cdot \rotT\cdot\bU$, while every element of $\bB$ decomposes uniquely as $\bT\cdot \bU$. Therefore this flag-only definition leads to the same space of equivalence classes as those of wild monodromy representations of Definition~\ref{def:wildmonodromy} (the diagonality for elementary paths of type b) comes from the commutation relation over the gates:  $\rotT\cdot\bB\cdot \rotT\cap\bB=\bT$).
	
	The formalization of matrices in $\bB$ representing paths of type a) and b) is their diagonal, the formalization of matrices in $\SL_2(\C)\smallsetminus\bB$ representing paths of type c) is the matrix $\begin{psmallmatrix}0&-c^{-1}\\c&0\end{psmallmatrix}$ where $c\neq 0$ is the $(2,1)$-element of the matrix.
\end{remark}

\subsubsection{Wild monodromy representation associated to a differential system}

\begin{definition}
Each path  $d_{\sZ}^{*}\xrightarrow{\gamma} d_{\tilde\sZ}^{\tilde *}$, $*,\tilde*=\pm$, in $\Pi_1(\sX\smallsetminus\Crit_\epsilon, \sD_{\sW_\epsilon})$ gives rise to a connection matrix $M_\gamma$ defined by
\[\cont_\gamma Y_{\sZ}^{*}=Y_{\tilde\sZ}^{\tilde*}\cdot M_\gamma,\]
where $\cont_\gamma$ is the operator of analytic continuation along the path $\gamma$, and $Y_{\sZ}^{*},Y_{\tilde\sZ}^{\tilde*}$ are the fundamental solutions \eqref{eq:Y+-}.
The map $\rho:\gamma\mapsto M_\gamma$ defines a \emph{wild monodromy representation associated to the differential system} \eqref{eq:unfoldedsystem}.
\end{definition}

Let $\proj_\epsilon(\sW)$ be the projection of $\sW$ onto the $\epsilon$-coordinate as a possibly ramified set, and $\sC\subset\sE$ the confluence divisor where critical points merge (cf. \S\ref{sec:parametricfoliation}).
Then the wild monodromy representation $\rho:\Pi_1(\sX\smallsetminus\Crit_\epsilon, \sD_{\sW_\epsilon})\to \SL_2(\C)$ associated to the system depends analytically on $\epsilon\in\proj_\epsilon(\sW)\smallsetminus\sC$ and that the representations of those paths that have a limit when $\epsilon\to\epsilon_0\in\sC$ have a limit as well.

%

\begin{figure}[t]
	\captionsetup[subfigure]{labelformat=empty}
	\centering
	\begin{subfigure}[t]{0.45\textwidth}
		\includegraphics [width=\textwidth]{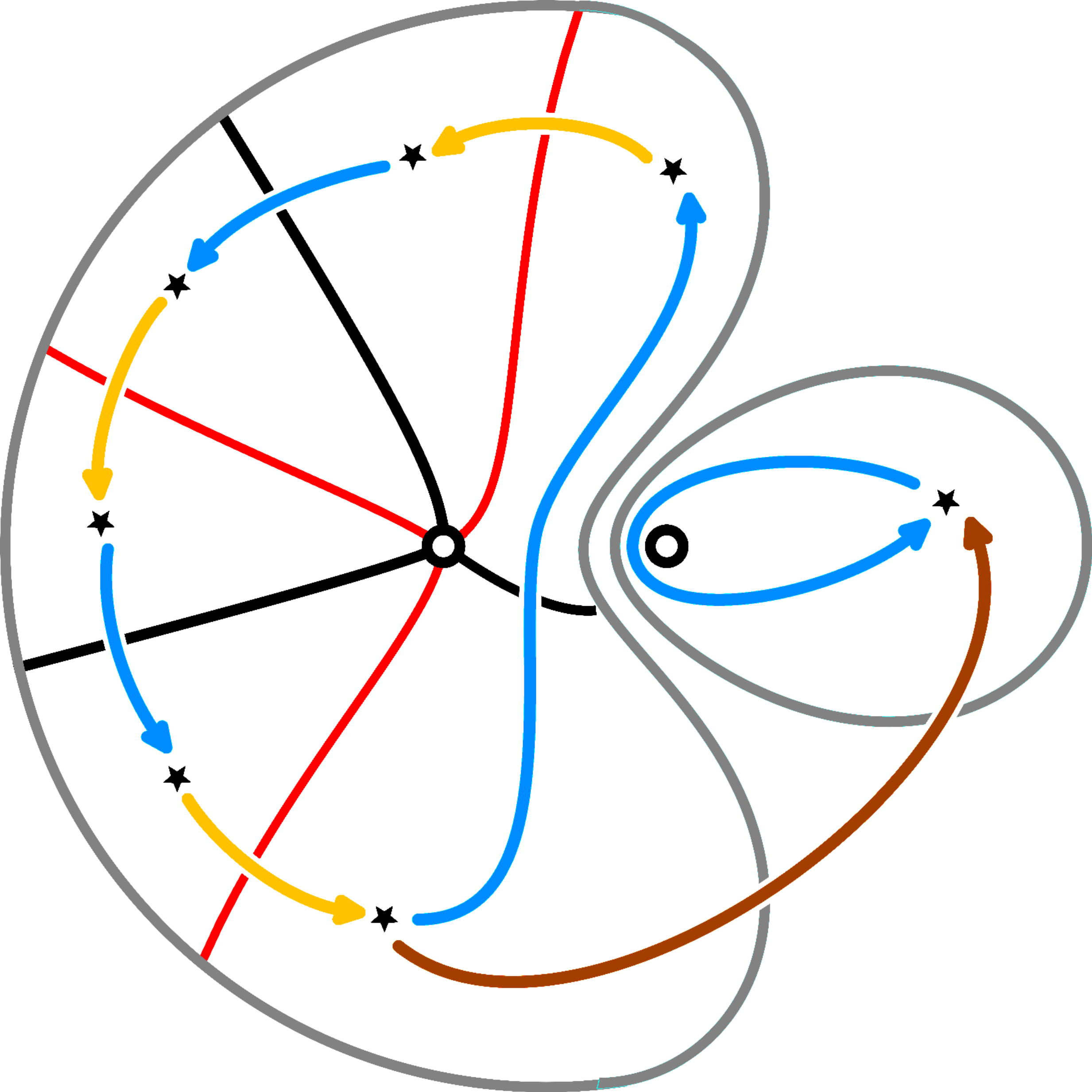}
		\caption{({\sc b}) \ $\Delta_\epsilon=\frac{c}{(x-\epsilon)^2(x+\epsilon)^5}(\d x)^2$}	
	\end{subfigure}
\hskip0.04\textwidth
	\begin{subfigure}[t]{0.45\textwidth}
		\includegraphics [width=\textwidth]{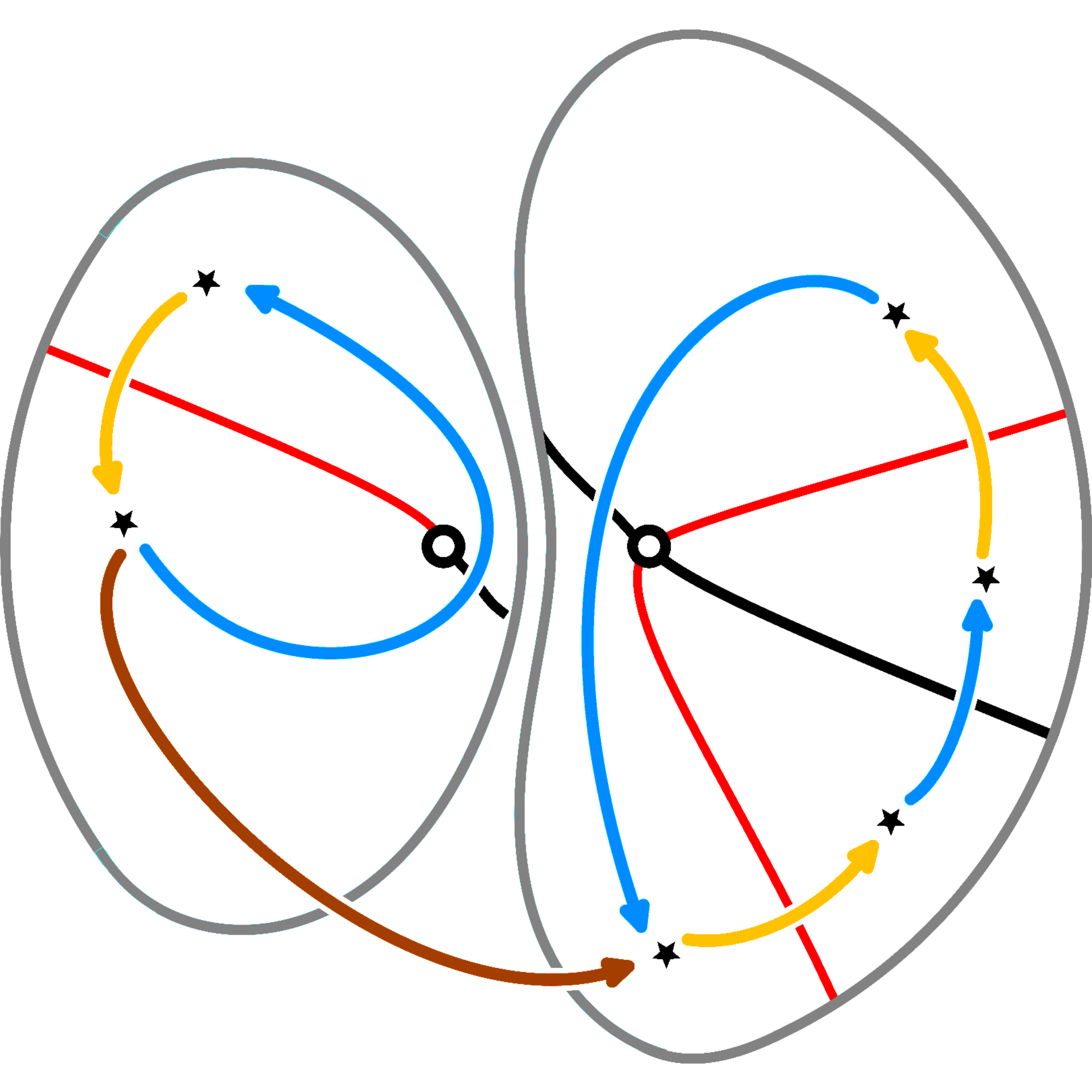}
		\caption{({\sc c}) \ $\Delta_\epsilon=\frac{c}{(x-\epsilon)^3(x+\epsilon)^4}(\d x)^2$}	
	\end{subfigure}
	\caption{Schematic presentation of a cluster wild monodromy compatible with $\Delta_\epsilon$.
		 Compare with Figure~\ref{figure:wild}\,({\sc b}) and ({\sc c}).}
	\label{figure:cluster}
\end{figure}

\subsection{Confluent, cluster and outer wild monodromy representations}\label{sec:cluster}

\subsubsection{Confluent wild monodromy vs. cluster wild monodromy}

When the critical point divisor $\Crit_\epsilon$ is non-smooth at $\epsilon=0$ there is a \emph{confluence} of equilibria or equilibria and saddles of the quadratic differential $\Delta_\epsilon$.
When we treat all the coalescing critical points together, as we did in the previous sections, 
we talk about \emph{confluent wild monodromy representations} 
\[\rho^\conf_\epsilon:\Pi_1(\sX\smallsetminus\Crit_\epsilon, \sD_{\sW_\epsilon})\to \SL_2(\C)\]
where $\rho^\conf_\epsilon=\rho$ are as in \S\,\ref{sec:wildmonodromy}, $\epsilon\in\proj_\epsilon(\sW)$.

On the other hand, for $\epsilon\in\sE\smallsetminus\sC$, one can also consider the set of coalescing singularities as a cluster of separate singularities,
where each of them has its own local invariants.
Combined together, the local monodromy, resp.  local wild monodromy, of regular singularities, resp. irregular, singularities
form a \emph{cluster wild monodromy representation}. See Figure~\ref{figure:cluster}. Such description is naturally discontinuous at the points of confluence $\epsilon\in\sC$.

\begin{definition}[Cluster wild monodromy]
Fix $\epsilon$. For each of the singularities $a_i(\epsilon)\in\Sing_\epsilon$ we choose a simply connected neighborhood $\sX_{a_i(\epsilon)}\subset \sX_\epsilon$,
in a way that they are pairwise disjoint.
Namely, we may choose $\sX_{a_i(\epsilon)}$ to be the connected component of the complement of the fat transverse graph (not the extended one) in $\sX_\epsilon$. Let $\Delta_{a_i(\epsilon)}$ be the germ of $\Delta_{\epsilon}$ at $a_i(\epsilon)$.
\begin{itemize}
	\item For an irregular (parabolic equilibrium) $a_i(\epsilon)$ of rank $\nu_i>0$, let $\sD_{a_i(\epsilon)}\subset \sX_{a_i(\epsilon)}$ 
	consist of the $4\nu_i$ base-points situated in the $2\nu_i$ sepal zones of $a_i(\epsilon)$.
	\item For a regular (simple equilibrium) $a_i(\epsilon)$ of rank $\nu_i=0$,  let $\sD_{a_i(\epsilon)}=\{d_{a_i(\epsilon)}\}\subset \sX_{a_i(\epsilon)}$ consist of a single new base-point.
\end{itemize}

A \emph{cluster wild monodromy representation} 
\[\rho^\cluster_\epsilon:\Pi_1(\sX\smallsetminus\Sing_\epsilon, \sD_\epsilon^\cluster)\to \SL_2(\C),\qquad \sD_\epsilon^\cluster=\bigcup_{a_i(\epsilon)\in\Sing_\epsilon}\sD_{a_i(\epsilon)},\]
is a monodromy representation whose restriction to each $\Pi_1(\sX_{a_i(\epsilon)}\smallsetminus\{a_i(\epsilon)\}, \sD_{a_i(\epsilon)})$
is:
\begin{itemize}
	\item local wild monodromy representation associated to $\Delta_{a_i(\epsilon)}$  when $a_i(\epsilon)$ is irregular,
	\item ordinary $\SL_2(\C)$-monodromy representation when $a_i(\epsilon)$ is regular.
\end{itemize}	
Namely, it associates to a base-point $d\in\sD_{a_i(\epsilon)}$ the structure
$\big(\C^2,\ \textbf{flag},\ \textbf{splitting}\big)$ if $a_i(\epsilon)$ is irregular, and the vector space $\C^2$ if $a_i(\epsilon)$ is regular.

An equivalence relation is provided by an element 
\[\tau\in \prod_{a_i(\epsilon)\ \text{irr.}} \bT^{\sD_{a_i(\epsilon)}}\times \prod_{a_i(\epsilon)\ \text{reg.}} \SL_2(\C)^{\sD_{a_i(\epsilon)}}.\] 
\end{definition}

\begin{definition}
A cluster wild monodromy representation $\rho^\cluster_\epsilon$ depending continuously on  $\epsilon\in\proj_\epsilon(\sW)\smallsetminus\sC$
 is \emph{compatible with $\Delta_\epsilon$}
if, up to the equivalence relation continuous in $\epsilon$, its restriction to each $\Pi_1(\sX_{a_i(\epsilon)}\smallsetminus\{a_i(\epsilon)\}, \sD_{a_i(\epsilon)})$
is:
\begin{itemize}
	\item compatible with $\Delta_{a_i(\epsilon)}$  when $a_i(\epsilon)$ is irregular,
	\item can  be equipped with the standard flag structure $\big(\C^2,\ \textnormal{\textbf{flag}}\big)$ when $a_i(\epsilon)$ is regular in such a way that $\rho_\epsilon^\cluster(\gamma)\in\bB$ and satisfies \eqref{eq:compatible}
	for each $\gamma\in\pi_1(\sX_{a_i(\epsilon)}\smallsetminus\{a_i(\epsilon)\}, d_{a_i(\epsilon)})$. 
\end{itemize}	
\end{definition}

\begin{remark}
Let $a_i(\epsilon)$ be a regular singularity, and assume that the local monodromy $\rho_\epsilon^\cluster$ satisfies the trace relation:
	\[	\tr\rho^\cluster_\epsilon(\gamma)=\e^{\int_{\gamma}\sqrt\Delta_\epsilon}+\e^{-\int_{\gamma}\sqrt\Delta_\epsilon},\quad\text{for each}\quad \gamma\in\pi_1\big(\sX_{a_i(\epsilon)}\smallsetminus\{a_i(\epsilon)\}, d_{a_i(\epsilon)}\big).
	\]	
Then for $\epsilon$ for which $a_i(\epsilon)$ is non-resonant, not only is $\rho^\cluster_\epsilon(\gamma)$ is diagonalizable by $\SL_2(\C)$, but one can also choose which eigenvalue is in the $(1,1)$-position, so the compatibility condition is automatically satisfied. 
However it can fail near the values of $\epsilon$ for which there is a resonance:
although one can still upper-triangularize by $\SL_2(\C)$, in general, one may not choose which eigenvalue (as a function of $\epsilon$) gets to be the first.  
\end{remark}

%

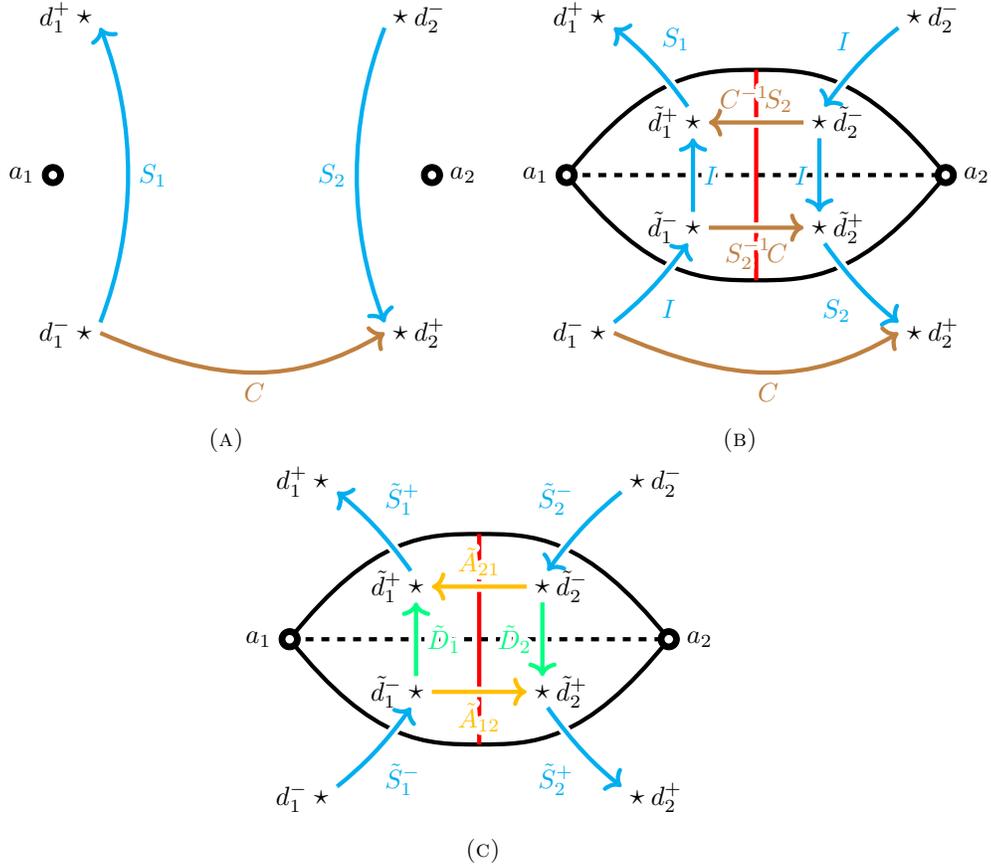
\begin{figure}
	\centering
	\begin{subfigure}[t]{0.45\textwidth}\label{figure:gaterep-a}
		\begin{tikzpicture}[scale=1.4]
			\filldraw (0.2,0) circle (0.1) node[left]{$a_1\ $}; \filldraw[color=white] (0.2,0) circle (1pt);
			\filldraw (3.8,0) circle (0.1) node[right]{$\ a_2$}; \filldraw[color=white] (3.8,0) circle (1pt);
			
			\filldraw (0.5,1.5) node{\large$\filledstar$} node[left]{$d_1^+\, $};
			\filldraw (0.5,-1.5) node{\large$\filledstar$} node[left]{$d_1^-\, $};
			\filldraw (3.5,1.5) node{\large$\filledstar$}node[right]{$\, d_2^-$};
			\filldraw (3.5,-1.5) node{\large$\filledstar$} node[right]{$\, d_2^+$};
			
			\draw[ultra thick, color=cyan,->] (0.65,-1.4) .. controls (1,-.5) and (1,.5) .. (0.65,1.4)  node[midway,right] {$S_1$};
			\draw[ultra thick, color=cyan,<-] (3.35,-1.4) .. controls (3,-.5) and (3,.5) .. (3.35,1.4)  node[midway,left] {$S_2$};
			\draw[ultra thick, color=brown,<-] (3.35,-1.5) .. controls (2.5,-2) and (1.8,-2) .. (0.65,-1.5) node[midway,below] {$C$};
		\end{tikzpicture}
		\caption{}	
	\end{subfigure}	
	\qquad	
	\begin{subfigure}[t]{0.45\textwidth}\label{figure:gaterep-b}
		\begin{tikzpicture}[scale=1.4]
			\draw[ultra thick,dashed] (0.2,0) -- (3.8,0);
			\draw[ultra thick] (0.2,0) .. controls (1,1) and (1.5,1) .. (2,1) .. controls (2.5,1) and (3,1) .. (3.8,0);
			\draw[ultra thick] (0.2,0) .. controls (1,-1) and (1.5,-1) .. (2,-1) .. controls (2.5,-1) and (3,-1) .. (3.8,0);
			\draw[ultra thick, color=red] (2,1) -- (2,-1); 
			\filldraw (0.2,0) circle (0.1) node[left]{$a_1\ $}; \filldraw[color=white] (0.2,0) circle (1pt);
			\filldraw (3.8,0) circle (0.1) node[right]{$\ a_2$}; \filldraw[color=white] (3.8,0) circle (1pt);
			
			\filldraw (1.4,0.5) node{\large$\filledstar$} node[left]{$\tilde d_1^+\, $};
			\filldraw (1.4,-0.5) node{\large$\filledstar$} node[left]{$\tilde d_1^-\, $};
			\filldraw (0.5,1.5) node{\large$\filledstar$} node[left]{$d_1^+\, $};
			\filldraw (0.5,-1.5) node{\large$\filledstar$} node[left]{$d_1^-\, $};
			\filldraw (2.6,0.5) node{\large$\filledstar$} node[right]{$\, \tilde d_2^-$};
			\filldraw (2.6,-0.5) node{\large$\filledstar$} node[right]{$\, \tilde d_2^+$};
			\filldraw (3.5,1.5) node{\large$\filledstar$}node[right]{$\, d_2^-$};
			\filldraw (3.5,-1.5) node{\large$\filledstar$} node[right]{$\, d_2^+$};

			\draw[line width=4pt, color=white] (0.65,-1.4) .. controls (0.9,-1.2) and (1.1,-1) .. (1.35,-0.65);
			\draw[ultra thick, color=cyan,->] (0.65,-1.4) .. controls (0.9,-1.2) and (1.1,-1) .. (1.35,-0.65)  node[midway,below right] {$I$};
	
			\draw[line width=4pt, color=white] (1.4,-0.35) -- (1.4,0.35);
			\draw[ultra thick, color=cyan,->] (1.4,-0.35) -- (1.4,0.35)  node[midway,right] {\contour{white}{\color{cyan}$I$}};	
	
			\draw[line width=4pt, color=white] (0.65,1.4) .. controls (0.9,1.2) and (1.1,1) .. (1.35,0.65);
			\draw[ultra thick, color=cyan,<-] (0.65,1.4) .. controls (0.9,1.2) and (1.1,1) .. (1.35,0.65)  node[midway,above right] {$S_1$};	
	
			\draw[line width=4pt, color=white] (3.35,-1.4) .. controls (3.1,-1.2) and (2.9,-1) .. (2.65,-0.65);
			\draw[ultra thick, color=cyan,<-] (3.35,-1.4) .. controls (3.1,-1.2) and (2.9,-1) .. (2.65,-0.65)  node[midway,below left] {$S_2$};	
	
			\draw[line width=4pt, color=white] (2.6,-0.35) -- (2.6,0.35);
			\draw[ultra thick, color=cyan,<-] (2.6,-0.35) -- (2.6,0.35)  node[midway,left] {\contour{white}{\color{cyan}$I$}};	
	
			\draw[line width=4pt, color=white] (3.35,1.4) .. controls (3.1,1.2) and (2.9,1) .. (2.65,0.65);
			\draw[ultra thick, color=cyan,->] (3.35,1.4) .. controls (3.1,1.2) and (2.9,1) .. (2.65,0.65)  node[midway,above left] {$I$};	
	
			\draw[line width=4pt, color=white] (3.35,-1.5) .. controls (2.5,-2) and (1.8,-2) .. (0.65,-1.5);
			\draw[ultra thick, color=brown,<-] (3.35,-1.5) .. controls (2.5,-2) and (1.8,-2) .. (0.65,-1.5) node[midway,below] {$C$};	
	
			\draw[line width=4pt, color=white] (1.55,-0.5) -- (2.45,-0.5);
			\draw[ultra thick, color=brown,->] (1.55,-0.5) -- (2.45,-0.5) node[midway,below] {\contour{white}{\color{brown}$S_2^{-1}\!C$}};	
	
			\draw[line width=4pt, color=white] (2.45,0.5) -- (1.55,0.5);
			\draw[ultra thick, color=brown,->] (2.45,0.5) -- (1.55,0.5) node[midway,above]  {\contour{white}{\color{brown}$C^{-1}\!S_2$}};
		\end{tikzpicture}
		\caption{}	
	\end{subfigure}		
	\begin{subfigure}[t]{0.49\textwidth}\label{figure:gaterep-c}	
		\begin{tikzpicture}[scale=1.4]
			\draw[ultra thick,dashed] (0.2,0) -- (3.8,0);
			\draw[ultra thick] (0.2,0) .. controls (1,1) and (1.5,1) .. (2,1) .. controls (2.5,1) and (3,1) .. (3.8,0);
			\draw[ultra thick] (0.2,0) .. controls (1,-1) and (1.5,-1) .. (2,-1) .. controls (2.5,-1) and (3,-1) .. (3.8,0);
			\draw[ultra thick, color=red] (2,1) -- (2,-1); 
			\filldraw (0.2,0) circle (0.1) node[left]{$a_1\ $}; \filldraw[color=white] (0.2,0) circle (1pt);
			\filldraw (3.8,0) circle (0.1) node[right]{$\ a_2$}; \filldraw[color=white] (3.8,0) circle (1pt);
			
			\filldraw (1.4,0.5) node{\large$\filledstar$} node[left]{$\tilde d_1^+\, $};
			\filldraw (1.4,-0.5) node{\large$\filledstar$} node[left]{$\tilde d_1^-\, $};
			\filldraw (0.5,1.5) node{\large$\filledstar$} node[left]{$d_1^+\, $};
			\filldraw (0.5,-1.5) node{\large$\filledstar$} node[left]{$d_1^-\, $};
			\filldraw (2.6,0.5) node{\large$\filledstar$} node[right]{$\, \tilde d_2^-$};
			\filldraw (2.6,-0.5) node{\large$\filledstar$} node[right]{$\, \tilde d_2^+$};
			\filldraw (3.5,1.5) node{\large$\filledstar$}node[right]{$\, d_2^-$};
			\filldraw (3.5,-1.5) node{\large$\filledstar$} node[right]{$\, d_2^+$};
			
			\draw[line width=4pt, color=white] (0.65,-1.4) .. controls (0.9,-1.2) and (1.1,-1) .. (1.35,-0.65);
			\draw[ultra thick, color=cyan,->] (0.65,-1.4) .. controls (0.9,-1.2) and (1.1,-1) .. (1.35,-0.65)  node[midway,below right] {$\tilde S_1^-$};

			\draw[line width=4pt, color=white] (1.4,-0.35) -- (1.4,0.35);		
			\draw[ultra thick, color=bgreen,->] (1.4,-0.35) -- (1.4,0.35)  node[midway,right] {\contour{white}{\color{bgreen}$\tilde D_1$}};
			
			\draw[line width=4pt, color=white] (0.65,1.4) .. controls (0.9,1.2) and (1.1,1) .. (1.35,0.65);
			\draw[ultra thick, color=cyan,<-] (0.65,1.4) .. controls (0.9,1.2) and (1.1,1) .. (1.35,0.65)  node[midway,above right] {$\tilde S_1^+$};
			
			\draw[line width=4pt, color=white] (3.35,-1.4) .. controls (3.1,-1.2) and (2.9,-1) .. (2.65,-0.65);
			\draw[ultra thick, color=cyan,<-] (3.35,-1.4) .. controls (3.1,-1.2) and (2.9,-1) .. (2.65,-0.65)  node[midway,below left] {$\tilde S_2^+$};
			
			\draw[line width=4pt, color=white] (2.6,-0.35) -- (2.6,0.35);		
			\draw[ultra thick, color=bgreen,<-] (2.6,-0.35) -- (2.6,0.35)  node[midway,left] {\contour{white}{\color{bgreen}$\tilde D_2$}};

			\draw[line width=4pt, color=white] (3.35,1.4) .. controls (3.1,1.2) and (2.9,1) .. (2.65,0.65);
			\draw[ultra thick, color=cyan,->] (3.35,1.4) .. controls (3.1,1.2) and (2.9,1) .. (2.65,0.65)  node[midway,above left] {$\tilde S_2^-$};

			\draw[line width=4pt, color=white] (1.55,-0.5) -- (2.45,-0.5);
			\draw[ultra thick, color=amber,->] (1.55,-0.5) -- (2.45,-0.5) node[midway,below] {\contour{white}{\color{amber}$\tilde A_{12}$}};

			\draw[line width=4pt, color=white] (2.45,0.5) -- (1.55,0.5);		
			\draw[ultra thick, color=amber,->] (2.45,0.5) -- (1.55,0.5) node[midway,above] {\contour{white}{\color{amber}$\tilde A_{21}$}};
		\end{tikzpicture}
		\caption{}	
	\end{subfigure}	
	\caption{Change of presentation of wild monodromy as a gate between singularities $a_1$ and $a_2$ is added. Some of the base-points $d_i^\pm$ can be equal.
	Each arrow is labeled by its representing matrix.}	
	\label{figure:gaterep}
\end{figure}

For each regular singularity $a_i(\epsilon)$, the local moduli space of equivalence classes of flag representations $\Pi_1(\sX_{a_i(\epsilon)}\smallsetminus\{a_i(\epsilon)\},\ \sD_{a_i(\epsilon)})\to\bB$ is isomorphic
to $\Pi_1(\sX_{a_i(\epsilon)}\smallsetminus\{a_i(\epsilon)\}, \sD_{\sW_\epsilon}\cap\sX_{a_i(\epsilon)})\to\bB$.
So once $\rho^\conf_\epsilon$ is equipped with compatible flag structure at the regular singularities, it can be identified with a representation of $\Pi_1(\sX\smallsetminus\Crit_\epsilon, \sD_{\sW_\epsilon})$ respecting the flag structure on the lagoons. The only difference between this and a confluent wild monodromy representation in the sense of Remark~\ref{remark:flag} lies in the condition on transversality of the pairs of flags along gates.

\begin{proposition}
For $\epsilon\in\proj_\epsilon(\sW)$, the space $\Mon^\conf_{\Delta,\sW,\epsilon}$ of confluent wild monodromy representations compatible with $\Delta_\epsilon$ is isomorphic to the subspace of $\Mon^\cluster_{\Delta,\epsilon}$ of cluster wild monodromy representations compatible with $\Delta_\epsilon$ defined by the conditions of transversality of the pairs of flags along each of the gates.
%
\end{proposition}

\begin{proof}
We assume that the vector spaces representing the base-points $d_{a_i}$ at regular singularities $a_i$  are equipped by compatible flag structure.
We proceed by adding gates one at a time, together with the base-points in the associated $\alpha\omega$-zone to which we extend the representation $\rho_\epsilon^\cluster$. 
See Figures~\ref{figure:gaterep}(a) and \ref{figure:gaterep}(b).
Assuming transversality along the gate we act by $\bB$ on the 4 base-points of the $\alpha\omega$-zone as in Remark~\ref{remark:flag} to get it to the correct form  satisfying the constraints of Definition~\ref{def:wildmonodromy}. See Figure~\ref{figure:gaterep}(c).
After that, if some of the original base-points at either of the singularity was of the extra base-points $d_{a_i}$ of a regular singularity $a_i$ we remove it. 
This way all the local base-points in the Figure~\ref{figure:gaterep}(c) will be represented by the full structure $\big(\C^2,\ \textbf{flag},\ \textbf{splitting}\big)$.
 
In explicit terms: Let 
\[S_1=\begin{psmallmatrix}\alpha_1&s_1\\[3pt]0&\frac{1}{\alpha_1}\end{psmallmatrix},\qquad
S_2=\begin{psmallmatrix}\alpha_2&s_2\\[3pt]0&\frac{1}{\alpha_2}\end{psmallmatrix},\qquad
C=\begin{psmallmatrix}a&b\\[3pt]c&d\end{psmallmatrix},
\] 
in Figure~\ref{figure:gaterep}(b), then the transversality condition along the gate is $c\neq 0$,
and the $\bB$-conjugated representation is for example
\begin{align*}
\tilde S_1^-&=\begin{psmallmatrix}1&\frac{d}{c}\\[3pt]0&1\end{psmallmatrix},\quad&
\tilde S_1^+&=\begin{psmallmatrix}1&\alpha_1s_1-\frac{\alpha_1^2d}{c}\\[3pt]0&1\end{psmallmatrix},\quad&
\tilde D_1&=\begin{psmallmatrix}\alpha_1&0\\[3pt]0&\frac{1}{\alpha_1}\end{psmallmatrix},\quad&
\tilde A_{12}&=\begin{psmallmatrix}0&-\frac{1}{c}\\[3pt]c&0\end{psmallmatrix},\\
\tilde S_2^+&=\begin{psmallmatrix}1&\frac{a}{c}\\[3pt]0&1\end{psmallmatrix},\quad&
\tilde S_2^-&=\begin{psmallmatrix}1&\frac{s_2}{\alpha_2}-\frac{a}{\alpha_2^2 c}\\[3pt]0&1\end{psmallmatrix},\quad&
\tilde D_2&=\begin{psmallmatrix}\alpha_2&0\\[3pt]0&\frac{1}{\alpha_2}\end{psmallmatrix},\quad&
\tilde A_{21}&=\begin{psmallmatrix}0&\frac{\alpha_1}{\alpha_2 c}\\[3pt]-\frac{\alpha_2c}{\alpha_1}&0\end{psmallmatrix},
\end{align*}	
so that 
\[C=\tilde S_2^+\tilde A_{12}\tilde S_1^-,\quad S_1=\tilde S_1^+\tilde D_{1}\tilde S_1^-,\quad S_2=\tilde S_2^+\tilde D_{2}\tilde S_2^-,
\quad I=\tilde A_{12}\tilde D_{1}^{-1}\tilde A_{21}\tilde D_2^{-1}.
\]	
In the case when $a_i$ is regular singularity and $d_i^+=d_i^-=d_{a_i}$ is to be removed, then the pair $\tilde S_i^+,\ \tilde S_i^-$ is replaced by the composition $\tilde S_i=\tilde S_i^-\tilde S_i^+=\begin{psmallmatrix}1&\tilde s_i\\[3pt]0&1\end{psmallmatrix}$,
\[\tilde s_1=\tfrac{\alpha_1}{c}\big(s_1c+d(\tfrac{1}{\alpha_1}-\alpha_1)\big),\qquad
\tilde s_2=\tfrac{1}{\alpha_2c}\big(s_2c+a(\alpha_2-\tfrac{1}{\alpha_2})\big).
\]
Depending whether this happens, there are 3 possibilities:

(1) Neither of $d_i^\pm$ is removed:
The local invariants associated to the paths represented by $C=\tilde S_2^+\tilde A_{12}\tilde S_1^-$, $S_1C^{-1}S_2=\tilde S_1^+\tilde A_{21}\tilde S_2^-$,
$C^{-1}S_2=(\tilde S_1^-)^{-1}\tilde A_{12}^{-1}\tilde D_2\tilde S_2^-$ and $S_1C^{-1}=\tilde S_1^+\tilde D_1\tilde A_{12}^{-1}(\tilde S_2^+)^{-1}$ are
\begin{align*}
\diag\big(\tilde A_{12}^{-1}\tilde S_2^+\tilde A_{12}\tilde S_1^-\big)&=\begin{psmallmatrix}1&0\\[3pt]0&1+\xi_0\end{psmallmatrix},\quad&
\diag\big(\tilde D_2^{-1}\tilde A_{12}(\tilde S_1^-)^{-1}\tilde A_{12}^{-1}\tilde D_2\tilde S_2^-\big)&=\begin{psmallmatrix}1&0\\[3pt]0&1+\xi_2\end{psmallmatrix},\\
\diag\big(\tilde A_{21}^{-1}\tilde S_1^+\tilde A_{21}\tilde S_2^-\big)&=\begin{psmallmatrix}1&0\\[3pt]0&1+\xi_1\end{psmallmatrix},\quad&
\diag\big(\tilde A_{12}\tilde D_1^{-1}\tilde S_1^+\tilde D_1\tilde A_{12}^{-1}(\tilde S_2^+)^{-1}\big)&=\begin{psmallmatrix}1&0\\[3pt]0&1+\xi_3\end{psmallmatrix},
\end{align*}
where
\[\xi_0=-ad,\quad\xi_1=-\tfrac{\alpha_2}{\alpha_1}\big(s_1c-d\alpha_1\big)\big(s_2c-a\tfrac{1}{\alpha_2}\big),\quad
\xi_2=\alpha_2d\big(s_2c-a\tfrac{1}{\alpha_2}\big),\quad\xi_3=\tfrac{1}{\alpha_1}a\big(s_1c-d\alpha_1\big)\]
are subject to the relation
\[\xi_0\xi_1-\xi_2\xi_3=0.\]
The codimension 1 subspace of representations with $c=0$, violating the transversality condition, is mapped to the point $(\xi_0,\xi_1,\xi_2,\xi_3)=(-1,0,0,0)$.

(2) One of the points, let's say $d_1^+=d_1^-$, is removed: 
Then the formal monodromy around $a_1$
\[(\tilde S_1\tilde D_1)^\natural=\tilde D_1,\]
determines a formal invariant $\alpha_1$, and the compositions $\tilde S_2^+\tilde A_{12}\tilde S_1$ and $\tilde S_1\tilde A_{21}\tilde S_2^-$ determine the analytic invariants
\begin{align*}
	\diag\big(\tilde A_{12}^{-1}\tilde S_2^+\tilde A_{12}\tilde S_1\big)&=\begin{psmallmatrix}1&0\\[3pt]0&1+\xi_0\end{psmallmatrix},\quad&
	\diag\big(\tilde A_{21}^{-1}\tilde S_1\tilde A_{21}\tilde S_2^-\big) &=\begin{psmallmatrix}1&0\\[3pt]0&1+\xi_1\end{psmallmatrix},
\end{align*}
where
\[\xi_0=-\alpha_1a\big(s_1c+d(\tfrac{1}{\alpha_1}-\alpha_1)\big),\qquad
\xi_1=-\tfrac{\alpha_2}{\alpha_1}\big(s_1c+d(\tfrac{1}{\alpha_1}-\alpha_1)\big)\big(s_2c-a\tfrac{1}{\alpha_2}\big).\]
The codimension 1 subspace of representations with $c=0$, violating the transversality condition, is mapped to the point $(\xi_0,\xi_2)=\big(\alpha_1-\tfrac{1}{\alpha_1},\ 0\big)$.

(3) Both $d_1^+=d_1^-$ and $d_2^+=d_2^-$ are removed:
Then the formal monodromies around $a_1$ and $a_2$
\[(\tilde S_1\tilde D_1)^\natural=\tilde D_1=\begin{psmallmatrix}\alpha_1&0\\[3pt]0&\frac{1}{\alpha_1}\end{psmallmatrix},\qquad 
(\tilde S_2\tilde D_2)^\natural=\tilde D_2=\begin{psmallmatrix}\alpha_2&0\\[3pt]0&\frac{1}{\alpha_2}\end{psmallmatrix},\]
determine formal invariants $\alpha_1$ and $\alpha_2$, 
and the compositions $\tilde S_2\tilde A_{12}\tilde S_1$ and $\tilde S_1\tilde A_{21}\tilde S_2$ determine the analytic invariants
\begin{align*}
	\diag\big(\tilde A_{12}^{-1}\tilde S_2\tilde A_{12}\tilde S_1\big)&=\begin{psmallmatrix}1&0\\[3pt]0&1+\xi_0\end{psmallmatrix},\quad&
	\diag\big(\tilde A_{21}^{-1}\tilde S_1\tilde A_{21}\tilde S_2\big) &=\begin{psmallmatrix}1&0\\[3pt]0&1+\left(\tfrac{\alpha_2}{\alpha_1}\right)^2\xi_0\end{psmallmatrix},
\end{align*}
where
\[\xi_0=-\tfrac{\alpha_1}{\alpha_2}\big(s_1c+d(\tfrac{1}{\alpha_1}-\alpha_1)\big)\big(s_2c+a(\alpha_2-\tfrac{1}{\alpha_2})\big).\]
The codimension 1 subspace of representations with $c=0$, violating the transversality condition, is mapped to the point
 $\xi_0=\big(\alpha_1-\tfrac{1}{\alpha_1}\big)\big(\alpha_2-\tfrac{1}{\alpha_2}\big)$.
\end{proof}

\subsubsection{Outer wild monodromy}

Let $\Gamma_\epsilon$ be the closed gate graph containing all singular points, and denote
\[\sX_\epsilon^\out\]
the \emph{outer component} of $\sX\smallsetminus\Gamma_\epsilon$, i.e. the one whose closure contains the boundary of $\sX$.
The \emph{outer domains} $\sZ_{\epsilon}$ are those in $\sX^\out_\epsilon$, and $\sD_{\sW_\epsilon}^\out= \sD_{\sW_\epsilon}\cap\sX_\epsilon^\out$ their basepoints.
At the limit $\proj_\epsilon(\sW)\ni\epsilon\to 0$ the gate graph shrinks to the origin, so the paths of the fundamental groupoid that persist are precisely those that lie in $\sX_\epsilon^\out$.
The \emph{outer fundamental groupoid}
\[\Pi_1(\sX_\epsilon^\out, \sD_{\sW_\epsilon}^\out)\]
is a subgroupoid of $\Pi_1(\sX\smallsetminus\Crit_\epsilon, \sD_{\sW_\epsilon})$.

\begin{definition}
A \emph{outer wild  monodromy representation}
is a contravariant functor	$\rho_\epsilon^\out:\Pi_1(\sX_\epsilon^\out, \sD_{\sW_\epsilon}^\out)\to \SL_2(\C)$
subject to the same restriction and the same equivalence relations as in Definition~\ref{def:wildmonodromy}. Only elementary paths of type a) and c) are contained in $\Pi_1(\sX_\epsilon^\out, \sD_{\sW_\epsilon}^\out)$.

It is \emph{compatible} with $\Delta(x,\epsilon)$ if either the Katz rank $\nu$ of the limit singularity is integer and \eqref{eq:compatible} holds for the loop generating the fundamental group $\pi_1(\sX_\epsilon^\out, d)$, $d\in\sD_{\sW_\epsilon}^\out$, or $\nu\in\frac12\Z\smallsetminus\Z$.

Denote $\Rep^\out_{\Delta,\sW,\epsilon}$ the space  of outer wild monodromy representations compatible with $\Delta(x,\epsilon)$, and
\[\Mon^\out_{\Delta,\sW,\epsilon}= \left.\raisebox{.2em}{$\Rep^\out_{\Delta,\sW,\epsilon}$}\Big/\raisebox{-.2em}{$\bT^{\sD_{\sW_\epsilon}^\out}$}\right.,\]
the space of their equivalence classes.
\end{definition}

\begin{figure}
	\centering
	\begin{tikzpicture}[scale=1.2]	
		\draw[ultra thick,dashed] (-2,0) -- (2,0);
		\draw[ultra thick] (0,0) -- (0,2);
		\draw[ultra thick] (-2,0) .. controls (-2,1) and (-1,3) .. (0,2);
		\draw[ultra thick] (2,0) .. controls (2,1) and (1,3) .. (0,2);
		\draw[ultra thick] (-2,0) .. controls (-4,1) and (-4,4) .. (0,4) .. controls (4,4) and (4,1) .. (2,0);
		\draw[ultra thick,dashed] (-2,0) .. controls (-2.5,1) and (-3,3) .. (0,3) .. controls (3,3) and (2.5,1) .. (2,0);
		\draw[ultra thick, color=red] (-1,0)  .. controls (-1,1) and (-1,2) .. (0,2); 
		\draw[ultra thick, color=red] (1,0)  .. controls (1,1) and (1,2) .. (0,2); 
		\draw[ultra thick, color=red] (0,2)  -- (0,4); 
		
		\filldraw (-2,0) circle (2pt) node[below]{$a_2\ $}; \filldraw[color=white] (-2,0) circle (1pt);
		\filldraw (0,0) circle (2pt) node[below]{$a_1\ $}; \filldraw[color=white] (0,0) circle (1pt);
		\filldraw (2,0) circle (2pt) node[below]{$a_0\ $}; \filldraw[color=white] (2,0) circle (1pt);
		\filldraw[color=red] (0,2) circle (2pt);
		
		\filldraw (-2.7,1.8) node{\large$\filledstar$} node[below]{$d_2\, $};
		\filldraw (-2.05,1.2) node{\large$\filledstar$} node[below]{$\tilde d_2^+\, $};
		\filldraw (-1.4,0.9) node{\large$\filledstar$} node[below]{$\tilde d_2^-\, $};
		\filldraw (-.5,0.7) node{\large$\filledstar$} node[below]{$\tilde d_1^+$};
		\filldraw (.5,0.7) node{\large$\filledstar$} node[below]{$\tilde d_1^-$};
		\filldraw (1.4,0.9) node{\large$\filledstar$} node[below]{$\,\tilde d_0^+$};
		\filldraw (2.05,1.2) node{\large$\filledstar$} node[below]{$\,\tilde d_0^-$};
		\filldraw (2.7,1.8) node{\large$\filledstar$} node[below]{$d_0\, $};
		
		\draw[line width=4pt, color=white] (-2.7,2) .. controls (-2.5,2.5) and (-2,3.4) .. (0,3.4) .. controls (2,3.4) and (2.5,2.5) .. (2.7,2);
		\draw[ultra thick, color=amber,<-] (-2.7,2) .. controls (-2.5,2.5) and (-2,3.4) .. (0,3.4) .. controls (2,3.4) and (2.5,2.5) .. (2.7,2); \draw[color=amber] (0,3.38) node[above right] {\contour{white}{\color{amber}$A_{02}$}}; 
		
		\draw[line width=4pt, color=white] (-2,1.3) .. controls (-1.7,2) and (-1.3,2.5) .. (0,2.5) .. controls (1.3,2.5) and (1.7,2) .. (2,1.3); 
		\draw[ultra thick, color=amber,->] (-2,1.3) .. controls (-1.7,2) and (-1.3,2.5) .. (0,2.5) .. controls (1.3,2.5) and (1.7,2) .. (2,1.3); \draw[color=amber] (0,2.48) node[above left] {\contour{white}{\color{amber}$\tilde A_{20}$}}; 
		
		\draw[line width=4pt, color=white]  (-.65,0.75)  -- (-1.25,0.85);
		\draw[ultra thick, color=amber,->]  (-.65,0.75)  -- (-1.25,0.85) node[midway,above] {\contour{white}{\color{amber}$\,\tilde A_{12}$}};
		
		\draw[line width=4pt, color=white]  (.65,0.75)  -- (1.25,0.85);
		\draw[ultra thick, color=amber,<-]  (.65,0.75)  -- (1.25,0.85) node[midway,above] {\contour{white}{\color{amber}$\,\tilde A_{01}$}};
		
		\draw[line width=4pt, color=white]  (-.35,0.7)  -- (.35,0.7);
		\draw[ultra thick, color=cyan,<-]  (-.35,0.7)  -- (.35,0.7) node[midway,above] {\contour{white}{\color{cyan}$\,\tilde S_{1}$}};
		
		\draw[line width=4pt, color=white]  (-1.95,1.15)  -- (-1.55,0.95);
		\draw[ultra thick, color=cyan,<-]  (-1.95,1.15)  -- (-1.55,0.95) node[midway,above] {\contour{white}{\color{cyan}$\ \tilde S_{2}$}};
		
		\draw[line width=4pt, color=white]  (1.95,1.15)  -- (1.55,0.95);
		\draw[ultra thick, color=cyan,->]  (1.95,1.15)  -- (1.55,0.95) node[midway,above] {\contour{white}{\color{cyan}$\tilde S_{0}$}};
		
		\draw[line width=4pt, color=white]  (-2.6,1.7)  -- (-2.15,1.3);
		\draw[ultra thick, color=bgreen,<-]  (-2.6,1.7)  -- (-2.15,1.3); \draw[color=bgreen]  (-2.15,1.4) node[above] {\contour{white}{\color{bgreen}$\tilde D_{2}$}};
		
		\draw[line width=4pt, color=white]  (2.6,1.7)  -- (2.15,1.3);
		\draw[ultra thick, color=bgreen,->]  (2.6,1.7)  -- (2.15,1.3); \draw[color=bgreen]  (2.2,1.45) node[above] {\contour{white}{\color{bgreen}$\tilde D_{0}$}};
	\end{tikzpicture}
	\caption{Wild monodromy near a simple saddle point (rank $-\frac32$). The inner domains are bounded by the dashed triangular gate graph. Some of the singularities $a_0,a_1,a_2$ can be equal. 
		Each arrow is labeled by its representing matrix.}	
	\label{figure:saddlerep}
\end{figure}
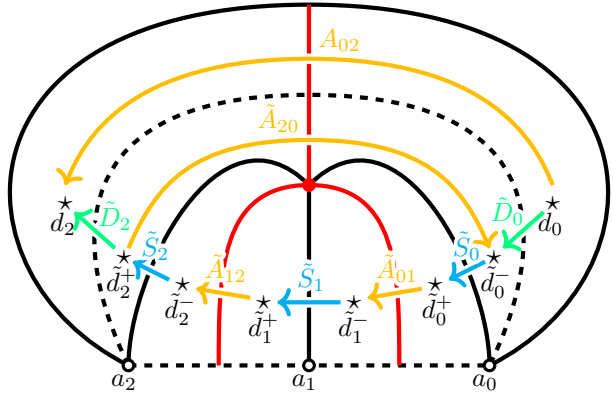

\begin{theorem}
Assume that $\Delta_\epsilon$ 
has at most one saddle point which is simple (of rank $\frac{3}{2}$) and that all singularities have integer ranks.
Then 
\[\Mon^\conf_{\Delta,\sW,\epsilon}\simeq\Mon^\out_{\Delta,\sW,\epsilon}.\]
\end{theorem}

\begin{proof}
We need to show that an outer wild monodromy $\rho_\epsilon^\out$ determines uniquely the confluent wild monodromy $\rho_\epsilon^\conf$.

When no saddle points are present all domains are outer, $\sD_{\sW_\epsilon}=\sD_{\sW_\epsilon}^\out$, and the difference between $\Pi_1(\sX_\epsilon^\out, \sD_{\sW_\epsilon}^\out)$ and $\Pi_1(\sX\smallsetminus\Crit_\epsilon, \sD_{\sW_\epsilon})$ are the paths that cross the gate graph $\Gamma_\epsilon$. The representation of the elementary paths of this kind (type b)) is fully determined  the formal outer monodromy since $\Gamma_\epsilon$ is a tree and all singularities have diagonal formal monodromy determined by $\Delta_\epsilon$.

When there is one simple saddle point, the local situation near it is as in Figure~\ref{figure:saddlerep}.
Acting by $\bT$ over the base-points $\tilde d_i^\pm$, $i=0,1,2$, one can conjugate the representation $\rho_\epsilon^\conf$ so that 
$\tilde S_i=\begin{psmallmatrix}	1&\tilde s_i\\0&1\end{psmallmatrix}$ and $\tilde A_{ij}=\begin{psmallmatrix}0&-1\\1&0	\end{psmallmatrix}$.
The circular relation $\tilde A_{20}\tilde S_2\tilde A_{12}\tilde S_1\tilde A_{01}\tilde S_0=I$ then implies $\tilde s_1=1$, $\tilde s_0=\tilde s_2=-1$.
And the relation $\tilde D_2\tilde A_{20}^{-1}\tilde D_0=A_{02}$, means that the representative $A_{02}$ of an outer path determines the product $\tilde D_2^{-1}\tilde D_0=\tilde A_{20}A_{02}$.
When the sum of the ranks of the three points $a_0$, $a_1$, $a_2$ is integer, as is in the situation when $m=1$, the formal monodromy around the three of them is diagonal,
and its outer segment, not shown in the figure, determines also the product $\tilde D_2\tilde A_{12}\tilde A_{01}\tilde D_0=-\tilde D_2\tilde D_0$.
\end{proof}

\subsection{Analytic equivalence of systems}

\begin{lemma}\label{lemma:monodromy}
	Given a wild monodromy representation $\rho$, the set of equivalences $\tau:\sD_{\sW_\epsilon}\to \bT$ that conjugate $\rho$ to itself consist of constant maps with values in
	\begin{itemize}[leftmargin=2\parindent]
		\item[(i)] $\bT$ if $\rho=\rho^\natural$ takes values in $\bN$ and all critical points have integer ranks,
		\item[(ii)] $\{I,-I\}$ otherwise.
	\end{itemize}
\end{lemma}

\begin{corollary}\label{corollary:monodromy}
	If two wild monodromy representations $\rho,\rho'$ are equivalent by means of a conjugacy $\tau$, $\tau\cdot\rho=\rho'\cdot\tau$,
	then such $\tau$ is determined either
	\begin{itemize}[leftmargin=2\parindent]
		\item[(i)] up to multiplication by an element of $\bT$, or
		\item[(ii)] up to a sign,
	\end{itemize}
	as an analytic function of the restriction of $\rho$ and $\rho'$ to elementary paths.
\end{corollary}

\begin{theorem}\label{theorem:unfoldedclassification}
Assume $m<2k$ and \eqref{eq:ass}. Two germs of analytic parametric systems \eqref{eq:unfoldedQ} satisfying \eqref{eq:reducedDelta}, hence with the same associated quadratic differential \eqref{eq:quaddiffomega},
are analytically gauge equivalent if and only if for some connected component $\sW$ \eqref{eq:sW} their (confluent) wild monodromy representations are conjugated.
\end{theorem}

\begin{proof}
Let us show that if the wild monodromies $\rho,\rho'$ of the two systems on $\sW$ are conjugated, then the systems are analytically gauge equivalent.

Assume first $\rho$ is of the type (ii) of Lemma~\ref{lemma:monodromy}.
By Corollary~\ref{corollary:monodromy} the conjugacy $\tau$ is uniquely determined up to a sign by $\rho,\rho'$ and analytic over the projection $\proj_\epsilon(\sW)$ of $\sW$ on the $\epsilon$-coordinate. 
If $\big\{ Y_{\sZ}^\pm\big\}$ are the fundamental solution matrices of the first system determining $\rho$, and  $\big\{ {Y'}_{\sZ}^\pm\big\}$ of the second system determining $\rho'$, then the modified fundamental solution matrices of the second system  $\big\{ {Y'}_{\sZ}^\pm\cdot\tau(d_{\sZ}^\pm)\big\}$ determine $\rho$.
The gauge transformation $T(x,\epsilon)$ given by
\[{Y'}_{\sZ}^\pm\cdot\tau(d_{\sZ}^\pm)=T\cdot {Y}_{\sZ}^\pm\]
is well-defined and univalued on $\coprod_{\epsilon\in\proj_\epsilon(\sW)}\sX\smallsetminus\Sing_\epsilon$.
It is bounded near all the singular points, therefore extends analytically on $\sX\times\proj_\epsilon(\sW)$.

Now let $\tilde\sW$ be a different connected component \eqref{eq:sW} whose projection $\proj_\epsilon(\tilde\sW)$ intersects $\proj_\epsilon(\sW)$, and let $\big\{Y_{\tilde\sZ}^\pm\big\}$, $\big\{ {Y'}_{\tilde\sZ}^\pm\big\}$ be the mixed-basis fundamental solution matrices determining wild monodromy representations $\tilde\rho$, $\tilde\rho'$.
Over the intersection $\proj_\epsilon(\sW)\cap\proj_\epsilon(\tilde\sW)$, one has also the mixed-basis fundamental solutions  $\big\{T\cdot Y_{\tilde\sZ}^\pm\big\}$ which means that
\begin{equation}\label{eq:Ttau}
	T\cdot Y_{\tilde\sZ}^\pm={Y'}_{\tilde\sZ}^\pm\cdot\tilde\tau(d_{\tilde\sZ}^\pm) 
\end{equation}
for some $\tilde\tau:\sD_{\tilde\sW_\epsilon}\to\bT$.
By Corollary~\ref{corollary:monodromy} this conjugacy $\tilde\tau$ is analytically determined by $\tilde\rho$ and $\tilde\rho'$, therefore it is analytic over whole $\proj_\epsilon(\tilde\sW)$.
By the same reasoning as before there is an analytic gauge transformation $\tilde T(x,\epsilon)$ on  $\sX\times\proj_\epsilon(\tilde\sW)$ between the systems, which must agree with $T(x,\epsilon)$ over  $\proj_\epsilon(\sW)\cap\proj_\epsilon(\tilde\sW)$.

Continuing like this, one extend the analytic gauge transformation $T(x,\epsilon)$ to the whole polydisc $\sX\times\sE$.

Assume now instead that $\rho$ is of the type (i) of Lemma~\ref{lemma:monodromy}. 
One can proceed in a similar manner, except this time the conjugacy $\tau=\tau_\sW$ analytic over $\proj_\epsilon(\sW)$ is defined only up to a multiplication by an analytic function $\proj_\epsilon(\sW)\to\bT$.
The conjugacy $\tilde\tau$ \eqref{eq:Ttau} between $\tilde\rho$ and $\tilde\rho'$ is defined only  over  $\proj_\epsilon(\sW)\cap\proj_\epsilon(\tilde\sW)$ and might not extend to $\proj_\epsilon(\tilde\sW)$.
However, as $\tilde\rho$ is again of type  (i) of Lemma~\ref{lemma:monodromy}, and $\tilde\rho$, $\tilde\rho'$ are conjugated over $\proj_\epsilon(\sW)\cap\proj_\epsilon(\tilde\sW)$,
then they are conjugated over the whole $\proj_\epsilon(\tilde\sW)$ by some other conjugacy  $\tau_{\tilde\sW}$.

So one obtains a collection of gauge transformations $T_\sW$ associated to conjugacies $\tau_\sW$, over the projections $\proj_\epsilon(\sW)$ of the different $\sW$'s.
One now has to correct the $\tau_\sW$'s by maps $\proj_\epsilon(\sW)\to\bT$ so that the corrected   $T_\sW$'s would match. This is a Cousin problem in the abelian Lie group $\bT\simeq\C^*$ which can be solved by means of Cauchy integrals, cf. \cite[\S6]{Hurtubise-Lambert-Rousseau}, \cite[\S7.2]{Balser}, \cite[Lemma 17.14]{Ilashenko-Yakovenko}.
If needed the $\proj_\epsilon(\sW)$'s can be restricted so that that their intersections are simply connected, and a codimension 1 subset of the confluence divisor (codimension 2 in $\sE$) can also be omitted by the Hartogs' theorem.
\end{proof}

\begin{theorem}\label{prop:unfoldedpointequivalence}
	Assume $m<2k$ and \eqref{eq:ass}.
	Two parametric families of systems \eqref{eq:unfoldedQ} with $P(x,\epsilon)$, $Q(x,\epsilon)$, and $P'(x',\epsilon)$, $Q'(x',\epsilon)$,  are 	
	\emph{analytically gauge--coordinate equivalent}
	if and only if they are \emph{analytically point equivalent}.
\end{theorem}

\begin{proof} 
	We need to show that two systems \eqref{eq:unfoldedQ} which are analytically gauge equivalent are also analytically point equivalent.
	Let $y'=T(x,\epsilon)y$ be an analytic gauge transformation between the two systems and let 
	$Y_\sZ(x,\epsilon)=\big(y_{\sZ,ij}(x,\epsilon)\big)_{i,j}$ be a mixed basis fundamental solution matrix for the first system on a domain $\sZ=\coprod_{\epsilon\in\proj_\epsilon(\sW)}\sZ_{\epsilon}$ of Definition~\ref{def:enlargedpetal}, and $Y_\sZ'(x,\epsilon)=T(x,\epsilon)Y_\sZ(x,\epsilon)$ the corresponding fundamental solution matrix for the second system over the same domain.
	Denote $f_\sZ(x,\epsilon)=\frac{y_{\sZ,11}}{y_{\sZ,12}}$ and likewise $f_\sZ'(x,\epsilon)=\frac{y_{\sZ,11}'}{y_{\sZ,12}'}$, which are locally invertible away from the singular divisor $\{P(x,\epsilon)=0\}$.
	We are looking for an analytic coordinate transformation $\phi(x,\epsilon)$ such that
	\begin{equation}\label{eq:rel}
		f'_\sZ=f_\sZ\circ\phi.
	\end{equation}
	This equation defines a unique solution $\phi_\sZ$ on $\sZ$.
	More precisely, writing $T(x,\epsilon)=\big(t_{ij}(x,\epsilon)\big)_{i,j}$, then
\[f_\sZ'=f_\sZ\cdot\frac{t_{11}+t_{12}\frac{y_{\sZ,21}}{y_{\sZ,11}}}{t_{11}+t_{12}\frac{y_{\sZ,22}}{y_{\sZ,12}}}\sim f_\sZ\cdot\frac{t_{11}+t_{12}\sqrt{\tilde Q}(1+o(1))}{t_{11}-t_{12}\sqrt{\tilde Q}(1+o(1))}.\]
As	$T(x,\epsilon)=t_{11}(x,\epsilon)I+t_{12}(x,\epsilon)\begin{psmallmatrix}0&1\\ \tilde Q(x,\epsilon)&0\end{psmallmatrix}\mod P(x,\epsilon)$, then
\[t_{11}(x,\epsilon)^2-t_{12})(x,\epsilon)^2\tilde Q(x,\epsilon)=\det T(x,\epsilon)\mod P(x,\epsilon)\neq 0,\]
and
$f_\sZ'=f_\sZ\cdot \e^{2\sqrt{\tilde Q} \cdot h_\sZ}$ for some function $h_\sZ(x,\epsilon)$ bounded on $\sZ$ (if $m\neq 0$ then $t_{11}\neq 0$,   and if $m=0$ then $\tilde Q\neq 0$).
We search for $\phi_\sZ(x,\epsilon)=x+P(x,\epsilon)g-\sZ(x,\epsilon)$ fixing the singularities.
The relation \eqref{eq:rel}becomes
	\[F_\sZ(g_\sZ,x,\epsilon)=h_\sZ,\qquad\text{where}\quad F_\sZ(g_\sZ,x,\epsilon)=\tfrac{1}{2\sqrt{\tilde Q}}\big(\log f_\sZ\circ\phi_\sZ-\log f_\sZ\big).\]
As $f_\sZ(x,\epsilon)\sim \e^{-2\int\sqrt{\tilde Q\cdot(1+\tilde b^2)}\frac{\d x}{P}}$
$\dd{g} F_\sZ(g,x,\epsilon)\big|_{g=0}\big|\sim -\sqrt{(1+\tilde b^2)}\neq 0$ on $\sZ$,
so by the implicit function theorem the equation there exists a unique  solution $g_\sZ(x,\epsilon)$, and it is bounded on $\sZ$.	

If  $\tilde\sZ$ is another domain, which has a non-empty connected intersection with $\sZ$, 
If $Y_{\tilde\sZ}(x,\epsilon)$ is a mixed basis  
then the corresponding fundamental solution matrices $Y_{\tilde\sZ}$ and $Y_{\tilde\sZ}'$ are related by the same connection matrix $M_{\sZ,\tilde\sZ}(\epsilon)$:
 $Y_{\tilde\sZ}=Y_{\sZ}M_{\sZ,\tilde\sZ}$ and $Y_{\tilde\sZ}'=Y_{\sZ}'M_{\sZ,\tilde\sZ}$ .
Therefore $f_{\tilde\sZ}=m_{\sZ,\tilde\sZ}\circ f_{\sZ}$ and $f_{\tilde\sZ}'=m_{\sZ,\tilde\sZ}\circ f_{\sZ}'$ for the projective transformation $m_{\sZ,\tilde\sZ}$
associated to $M_{\sZ,\tilde\sZ}$.
This means that $\phi_{\tilde\sZ}=\phi_\sZ$.
	
So the equation \eqref{eq:rel} defines a solution $\phi$ on the union of all such domains, bounded near the divisor $\{P(x,\epsilon)=0\}$,
so by Proposition~\ref{prop:covering} it extends analytically on the full neighborhood of the origin $\sX\times\sE$.	
\end{proof}

\section{Families of isomonodromic deformations}\label{sec:confluentisomonodromy}

A parametric family of meromorphic connections on a Riemann surface $\sX$
\[\nabla_{t,\epsilon}(x)=\d -\Omega_0(x,t,\epsilon),\qquad (t,\epsilon)\in\sT\times\sE,\]
with $\Omega_0$ a germ of meromorphic $\sl_2(\C)$-valued $\d x$-form
depending on a deformation variable $t=(t_1,\ldots,t_n)$ and a parameter $\epsilon$, is \emph{isomonodromic in $t$}
if it comes from a parametric family of \emph{flat} meromorphic connections
\begin{equation}\label{eq:nablaepsilon}
	\nabla_\epsilon(x,t)=\d -\Omega(x,t,\epsilon),\qquad \Omega(x,t,\epsilon)=\Omega_0(x,t,\epsilon)+\sum_{i=1}^n \Omega_i(x,t,\epsilon),
\end{equation}
where $\d=\d_{x,t}$ is the exterior derivative in the fibers $\{\epsilon=\const\}$ and $\Omega_i$ are germs of meromorphic $\sl_2(\C)$-valued $\d t_i$-forms, $\d\Omega=\Omega\wedge\Omega$.

We shall restrict our attention to a local situation where $\sX$, $\sT$, $\sE$ are germs of (poly)discs.
Then 
\[\Omega_0=A(x,t,\epsilon)\tfrac{\d x}{P(x,t,\epsilon)}\]
for some analytic $\sl_2(\C)$-valued function germ $A(x,t,\epsilon)$ with $A(0,0,0)\neq0$ and 
\[P(x,t,\epsilon)=x^{k+1}+P_k(t,\epsilon)x^k+\ldots+ P_0(t,\epsilon),\qquad P(x,0,0)=x^{k+1},\qquad k\geq 0.\]


\begin{definition}\label{assumptions}
\begin{enumerate}
\item The connection $\nabla_\epsilon(x,t)=\d-\Omega(x,t,\epsilon)$ satisfies \emph{transversality condition} if the polar divisors of $\Omega$ and $\d\Omega$ are bounded by the polar divisor of $\Omega_0(x,t,\epsilon)$ and are transverse to the fibration $\{(t,\epsilon)=\const\}$.

\item The divisor $\{P(x,t,\epsilon)=0\}$ satisfies \emph{non-crossing condition} if its restriction to each fiber $\{\epsilon=\const\}$ has no crossings.
Factoring $P(x,t,\epsilon)=\prod_j P_j(x,t,\epsilon)^{l_j}$ into product of irreducible polynomials in $x$, denote $p(x,t,\epsilon)=\prod_j P_j(x,t,\epsilon)$, which for generic values of $(t,\epsilon)$ has only simple roots.
Then the non-crossing is equivalent to the zeroes of the discriminant of $p(x,t,\epsilon)$ being independent of $t$, i.e. 
\begin{equation}\label{eq:disc}
\Disc_x p(x,t,\epsilon)=D(\epsilon)U(t,\epsilon),\qquad U(0,0)\neq0.	
\end{equation}
\end{enumerate}	
\end{definition}

\begin{lemma}\label{lemma:straightening}
A divisor $\{P(x,t,\epsilon)=0\}$ satisfies the non-crossing condition if and only if there exists an analytic coordinate change $\tilde x=\phi(x,t,\epsilon)$ which straightens it in $t$, i.e. under which it becomes $\{\tilde P(\tilde x,\epsilon)=0\}$.
\end{lemma}

\begin{proof}
	For each fixed $\epsilon\in\sE\smallsetminus\{D(\epsilon)=0\}$ let $P(x,t,\epsilon)=\prod_i\big(x-a_i(t,\epsilon)\big)^{l_i(\epsilon)}$ be the root factorization, $\sum_i l_i(\epsilon)=k+1$,
	and denote $P_j(x,t,\epsilon)=\prod_{i\neq j}\big(x-a_i(t,\epsilon)\big)^{l_i(\epsilon)}$.
	Then 
	\[\phi(x,t,\epsilon)=x+\sum_j\mfrac{P_j(x,t,\epsilon)}{P_j(a_j(t,\epsilon),t,\epsilon)}\big(a_j(0,\epsilon)-a_j(t,\epsilon)\big)\]
	is the unique polynomial of order $k$ that sends each root $a_j(t,\epsilon)$ of multiplicity $l_j(\epsilon)$ to $a_j(0,\epsilon)$ with the same multiplicity.
	Analytic continuation along loops in $\pi_1\big(\sT\times(\sE\smallsetminus\{D(\epsilon)=0\}),(0,\epsilon_0)\big)$ acts on the roots $a_i(t,\epsilon)$ by permutations,
	leaving $\phi$ invariant, which means that it is single-valued in $(t,\epsilon)$.
	It extends to $\{D(\epsilon)=0\}$ following the same  prescription.
\end{proof}

\begin{lemma}\label{lemma:transversalitycondition}
Assume the divisor $\{P(x,\epsilon)=0\}$ is independent of $t$, let $P(x,\epsilon)=\prod_j P_j(x,\epsilon)^{l_j}$ be the unique factorization into irreducible polynomials, and denote $p(x,\epsilon)=\prod_j P_j(x,\epsilon)$.
A connection \eqref{eq:nablaepsilon} satisfies the transversality condition if and only if it has the form
\[\nabla_\epsilon(x,t)=\d -A(x,t,\epsilon)\mfrac{\d x}{P(x,\epsilon)}-\sum_{i=1}^n B_i(x,t,\epsilon)\mfrac{p(x,\epsilon)\d t_i}{P(x,\epsilon)}\]
with $A$ and $B_i$ analytic.
\end{lemma}

\begin{proof}
The polynomial $p(x,\epsilon)$ is the smallest polynomial that makes $p\tdd{x}P$ divisible by $P$.	
Writing $\Omega=A(x,t,\epsilon)\frac{\d x}{P(x,\epsilon)}+\sum_i C_i(x,t,\epsilon)\frac{\d t_i}{P(x,\epsilon)}$,
then the $\d x\wedge\d t_i$-component of $P\d\Omega$ is $\big(\tdd{x}C_i-C_i\frac{\tdd{x}P}{P}\big)\d x\wedge\d t_i$ which is analytic if and only if $C_i$ is divisible by $p$.
\end{proof}

We can now prove a parametric version of Proposition~\ref{proposition:Heu}.

\begin{theorem}\label{theorem:confluentisomonodromy}
	Let $\nabla_\epsilon(x,t)$ be a germ of a parametric family of meromorphic flat connections \eqref{eq:nablaepsilon}	satisfying the transversality and non-crossing conditions.
	Then there exists a local analytic gauge--coordinate transformation $(x,y,t,\epsilon)\mapsto\big(\phi(x,t,\epsilon),\, T(x,t,\epsilon)y,\, t,\, \epsilon\big)$
	which brings it to a constant deformation
	\[\tilde\nabla_\epsilon(\tilde x,t)=\d-\begin{psmallmatrix}0&1\\[3pt]\tilde Q(\tilde x,\epsilon)&0\end{psmallmatrix}\mfrac{\d \tilde x}{\tilde P(\tilde x,\epsilon)}.\]
\end{theorem}

\begin{proof}
The proof follows the same steps as in Heu \cite[Proposition 2.5]{Heu}.

Step 1: Apply an analytic coordinate transformation of Lemma~\ref{lemma:straightening} to straighten the divisor to $\{P(x,\epsilon)=0\}$.

Step 2: Apply an analytic gauge transformation of Lemma~\ref{lemma:systemQepsilon} to bring the system $P(x,\epsilon)\frac{\d y}{\d x}=A(x,t,\epsilon)y$ to one with $A(x,\epsilon,t)=\begin{psmallmatrix}0&1\\[3pt] Q(x,t,\epsilon)&0\end{psmallmatrix}$.

Using Lemma~\ref{lemma:transversalitycondition}, write
$\Omega=\begin{psmallmatrix} \alpha & \beta \\ \gamma &-\alpha\end{psmallmatrix}$
with
\begin{align*}
	\alpha&=\sum_i a_i(x,t,\epsilon)\mfrac{p(x,\epsilon)}{P(x,\epsilon)}\d t_i,\qquad
	\beta=\mfrac{\d x}{P(x,\epsilon)}+\sum_i b_i(x,t,\epsilon)\mfrac{p(x,\epsilon)}{P(x,\epsilon)}\d t_i,\\
	\gamma&=Q(x,t,\epsilon)\mfrac{\d x}{P(x,\epsilon)}+\sum_i c_i(x,t,\epsilon)\mfrac{p(x,\epsilon)}{P(x,\epsilon)}\d t_i.
\end{align*}
The integrability condition $\d\Omega=\Omega\wedge\Omega$ is equivalent to
\[\d\alpha=\beta\wedge\gamma,\qquad \d\beta=2\alpha\wedge\beta,\qquad \d\gamma=2\gamma\wedge\alpha.\]

Step 3: The foliation $P(x,\epsilon)\beta=0$ is non-singular and transverse to the fibers $\{(t,\epsilon)=\const\}$, so there exists a rectifying coordinate transformation of the form $x\mapsto \phi(x,t,\epsilon)=x+p(x,\epsilon)f(x,t,\epsilon)$ that straightens it to $\d x=0$ and preserves the divisor $\{P(x,\epsilon)=0\}$.
Applying the associated point transformation \eqref{eq:parametricpoint} to the connection, one gets $\beta=\frac{\d x}{P(x,\epsilon)}$.
Now the condition $0=\d\beta=2\alpha\wedge\beta$ means that $\alpha=0$, so the condition $0=\d\alpha=\beta\wedge\gamma$ means that $\gamma=Q(x,t,\epsilon)\frac{\d x}{P(x,\epsilon)}$, and the condition $0=2\gamma\wedge\alpha=\d\gamma$ means that $Q=Q(x,\epsilon)$.
\end{proof}

\begin{definition}
A family of meromorphic quadratic differentials $\Delta_{t,\epsilon}(x)$ is iso-residual with respect to $t$
if for every fixed $\epsilon$ the multiplicities of critical points are constant in $t$ and so are their square residues. 
\end{definition}

\begin{corollary}
	Let $\nabla_{t,\epsilon}=\d-\begin{psmallmatrix}0&1\\[3pt]Q(x,t,\epsilon)&0\end{psmallmatrix}\frac{\d x}{P(x,t,\epsilon)}$ be a germ of parametric family of isomonodromic deformations satisfying the transversality and non-crossing conditions. Then:
\begin{enumerate}[leftmargin=1.5\parindent]
	\item The associated family of quadratic differentials $\Delta_{t,\epsilon}(x)=Q(x,t,\epsilon)\left(\frac{\d x}{P(x,t,\epsilon)}\right)^2$ is iso-residual with respect to $t$.
	\item The associated family of  confluent wild monodromy representations is constant with respect to $t$.
\end{enumerate}	
\end{corollary}

\begin{example}
	A connection with an irregular singularity of Poincar\'e rank 1
	\[\nabla_{0,0}(x)=\d-\frac{A_0+xA_1}{x^2}\d x+\holom,\]
	where ``hol'' stands  for terms holomorphic at $0$, has an isomonodromic deformation of the form
	\[\nabla_{0}(x,t)=\d-\frac{\e^tA_0+xA_1}{x^2}\d x+\frac{\e^tA_0}{x}\d t+\holom=\d+A_0\,\d\frac{\e^t}{x}-A_1\d\log x+\holom.\]
	The parametric unfolding
	\[\nabla_{0,\epsilon}(x)=\d-\frac{A_0+xA_1}{x^2-\epsilon}\d x+\holom\]
	deforms isomonodromicaly in dependence of the position of the singularities as
	\[\begin{split}
		\nabla_{\epsilon}(x,t)&=\d-\frac{\e^tA_0+xA_1}{x^2-\e^{2t}\epsilon}\d x+\frac{x\,\e^tA_0}{x^2-\e^{2t}\epsilon}\d t+\holom\\
		&=\d - \tfrac12(\tfrac{1}{\sqrt\epsilon}A_0+A_1)\,\d\log(x-\e^t\!\sqrt\epsilon)+\tfrac12(\tfrac{1}{\sqrt\epsilon}A_0-A_1)\,\d\log(x+\e^t\!\sqrt\epsilon)+\holom,
	\end{split}\]
	with the matrices $A_i=A_i(t,\epsilon)$, $i=1,2$, isospectral w.r.t. $t$.
	The coordinate change $x=\e^{t}\tilde x$ straightens the polar divisor and kills the polar part of the $\d t$-terms:	\[\tilde\nabla_{\epsilon}(\tilde x,t)=\d-\frac{A_0+\tilde xA_1}{\tilde x^2-\epsilon}\d\tilde x+\holom.\]	
\end{example}

\begin{proposition}
A family of meromorphic quadratic differentials $\Delta_{t,\epsilon}(x)$ is iso-residual with respect to $t$
if and only if there exists an analytic coordinate change $x\mapsto\phi(x,t,\epsilon)$ s.t. $\Delta_{0,\epsilon}(x)=\phi^*\Delta_{t,\epsilon}(x)$.
\end{proposition}

\begin{proof}
Use Lemma~\ref{lemma:straightening} to straighten the critical divisor.
Let $\bt_{t,\epsilon}=\int\sqrt{\Delta_{t,\epsilon}}$, and denote $\alpha_{t,\epsilon}=\bt_{t,\epsilon}-\bt_{0,\epsilon}$.	
Then 
\[ \Delta_{t,\epsilon}^{-\frac12}=\mfrac{\Delta_{0,\epsilon}^{-\frac12}}{1+\Delta_{0,\epsilon}^{-\frac12}.\,\alpha_{t,\epsilon}},\]
where the Lie derivative $\Delta_{0,\epsilon}^{-\frac12}.\,\alpha_{t,\epsilon}$ is an analytic function.
Likewise the vector field  $\alpha_{t,\epsilon}\cdot\Delta_{0,\epsilon}^{-\frac12}$ is analytic.
Consider the vector fields
\[ Y_{t,\epsilon}(x,s)=\dd{s}-\mfrac{\alpha_{t,\epsilon}\cdot\Delta_{0,\epsilon}^{-\frac12}}{1+s\,\Delta_{0,\epsilon}^{-\frac12}.\,\alpha_{t,\epsilon}},\qquad
X_{t,\epsilon}(x,s)=\mfrac{\Delta_{0,\epsilon}^{-\frac12}}{1+s\,\Delta_{0,\epsilon}^{-\frac12}.\,\alpha_{t,\epsilon}},\]
and let $\varphi_{t,\epsilon}(x,s)=x\circ\exp(Y_{t,\epsilon})(x,s)$ be the $x$-coordinate of the time-1-flow map of $Y_{t,\epsilon}$.
Since $[Y_{t,\epsilon},X_{t,\epsilon}]=0$, it follows that $\phi(x,t,\epsilon)=\varphi_{t,\epsilon}(x,0)$ is such that $\Delta_{0,\epsilon}(x)=\phi^*\Delta_{t,\epsilon}(x)$, see \cite[Lemma 2.6]{Klimes-Rousseau3}.
\end{proof}

\goodbreak

\section{Appendix}

\subsection{Formal invariants of unfoldings for \texorpdfstring{$m=0,1$}{m=0,1}}

\begin{theorem}\label{thm:unfoldedformal}
Assume $k>0$, $m=0,1$. Two parametric families of systems \eqref{eq:unfoldedsystem}, resp. \eqref{eq:unfoldedQ}, with the same polynomial $P(x,\epsilon)=P'(x,\epsilon)$ are formally gauge equivalent if and only if 
\[\det A(x,\epsilon)=\det A'(x,\epsilon)\mod P(x,\epsilon), \quad\text{resp.} \ \ Q(x,\epsilon)=Q'(x,\epsilon)\mod P(x,\epsilon).\]
\end{theorem}

\begin{proof}
	Let $y'=\hat T(x,\epsilon)y$ be a formal gauge transformations between two  parametric systems \eqref{eq:unfoldedQ}.
	Write $T(x,\epsilon)$ in the form \eqref{eq:T}. All the equations \eqref{eq:conjugationT}, \eqref{eq:abcd}  stay the same except for replacing $x^{k+1}\tdd{x}$ by $P(x,\epsilon)\tdd{x}$. Namely \eqref{eq:ab} becomes:
	\begin{equation}\label{eq:ab1}
		\begin{aligned}
			P\tdd{x} a&=\tfrac12 (Q'-Q)b,\\
			(Q'+Q)P\tdd{x} b+\tfrac12b P\tdd{x}(Q'+Q)-\tfrac12\big(P\tdd{x}\big)^3b&=(Q'-Q)a.
		\end{aligned}
	\end{equation}
	If $m=0$ then either $a(0,0)\neq0$ or $b(0,0)\neq0$, and if $m>0$ then $a(0,0)\neq0$, both mean that $Q'-Q$ is divisible by $P$.
	
	Conversely, if $Q'(x,\epsilon)=Q(x,\epsilon)+P(x,\epsilon)R(x,\epsilon)$, we want to construct a formal gauge equivalence $y'=\hat T(x,\epsilon)y$. Denote $\Cal I$ the ideal of formal series in $(x,\epsilon)$ that vanish when $\epsilon=0$.
	By Theorem~\ref{theorem:formalgaugeequivalence} one can assume that $R\in\Cal I$.
	Let us show that if $R\in\Cal I^n$ for $n\geq 1$, then there exists $\hat T(x,\epsilon)=I\mod \Cal I^n$ of the form \eqref{eq:T}, that conjugates the two systems modulo $\Cal I^{n+1}$, i.e. such that the new system has $R\in\Cal I^{n+1}$. The infinite composition of such transformation is a well defined formal gauge transformation (the partial compositions converge in the Krull topology). 
	The equations \eqref{eq:ab1} where $a=1\mod\Cal I^{n}$, $b=0\mod\Cal I^{n}$, are considered modulo $\Cal I^{n+1}$: 
	\begin{equation*}
		\begin{aligned}
			x^{k+1}\!\tdd{x} a&=0\mod\Cal I^{n+1},\\
			2Q_0x^{k+1}\!\tdd{x} b+b x^{k+1}\!\tdd{x}Q_0-\tfrac12\big(x^{k+1}\tdd{x}\big)^3b&=x^{k+1}R\mod\Cal I^{n+1},
		\end{aligned}
	\end{equation*}
	where $Q_0(x)=Q(x,0)$.
	If $k>0$ and $m=0,1$ then this equation has a formal solution $b(x,\epsilon)=\sum_{n=0}^{+\infty}b_n(\epsilon)x^n\in \Cal I^{n}$ for any formal series $R(x,\epsilon)\in\Cal I^{n}$.
	The formal transformation $T(x,\epsilon)$ \eqref{eq:T} obtained this way is corrected by a post-composition with the formal gauge transformation of Lemma~\ref{lemma:systemQ} to make it preserve the companion form of the system.
\end{proof}


\begin{theorem}\label{cor:weakformal}
	For $k>0$, $m=0,1$, formal gauge equivalence agrees with weak formal gauge equivalence.
\end{theorem}

\begin{proof}
Indeed, if $m=0,1$, $k>0$, then the formal invariant of the confluent family \eqref{eq:unfoldedsystem} is
\[\frac{Q(x,\epsilon)}{P(x,\epsilon)^2}(\d x)^2=-\det\Big(A(x,\epsilon)\frac{\d x}{P(x,\epsilon)}\Big)\mod\frac{(\d x)^2}{P(x,\epsilon)}\C\{x,\epsilon\},\]
with $Q(x,\epsilon)$ a polynomial in $x$ of order $\leq k$.
Let $P(x,\epsilon)=\prod_{i}(x-a_i(\epsilon))^{k_i+1}$ be a factorization of $P$ in the field extension $\K$ of the fraction field of $\C\{\epsilon\}$ by the roots $a_i(\epsilon)$.
Express
\[\frac{1}{P(x,\epsilon)^2}=\sum_{i}\frac{R_i(x,\epsilon)+(x-a_i(\epsilon)S_i(x,\epsilon))^{k_i+1}}{(x-a_i(\epsilon))^{2k_i+2}}\]
with $R_i(x,\epsilon),\ S_i(x,\epsilon)\in\K[x]$ polynomials of order $\leq k_i$ in $x$.
Then for $\epsilon$ for which $a_i(\epsilon)\neq a_j(\epsilon)$, $i\neq j$, the local formal invariant at $a_i(\epsilon)$ (unless resonant regular) is given by
\[
\begin{aligned}
	\frac{Q_i\big(x-a_i(\epsilon),\,\epsilon\big)}{(x-a_i(\epsilon))^{2k_i+2}}(\d x)^2
	&=-\det\Big(A(x,\epsilon)\frac{\d x}{P(x,\epsilon)}\Big)\mod\frac{(\d x)^2}{(x-a_i(\epsilon))^{k_i+1}}\K\{x-a_i(\epsilon)\}\\
	&=\frac{Q(x,\epsilon)}{P(x,\epsilon)^2}(\d x)^2\mod\frac{(\d x)^2}{(x-a_i(\epsilon))^{k_i+1}}\K\{x-a_i(\epsilon)\},
\end{aligned}
\]
which means that 
\begin{equation}
	Q_i\big(x-a_i(\epsilon),\,\epsilon\big)=Q(x,\epsilon)R_i(x,\epsilon)\mod (x-a_i(\epsilon))^{k_i+1}\K\{x-a_i(\epsilon)\}.
\end{equation}
Reciprocally,
\begin{equation}
	\frac{Q(x,\epsilon)}{P(x,\epsilon)^2}(\d x)^2=\sum_i \frac{Q_i\big(x-a_i(\epsilon),\,\epsilon\big)}{(x-a_i(\epsilon))^{2k_i+2}}(\d x)^2 \mod\frac{(\d x)^2}{P(x,\epsilon)}\K\{x\},
\end{equation}
since
\[\frac{Q}{P^2}-\sum_i \frac{Q_i}{(x-a_i)^{2k_i+2}}
=\frac{1}{P}\sum_{i}\left(\frac{QR_i-Q_i}{(x-a_i)^{k_i+1}}+QS_i\right)\frac{P}{(x-a_i)^{k_i+1}}\in\frac{1}{P}\K[x],\]
where each term of the sum on the right hand side is a polynomial in $(x-a_i)$ and hence also in $x$.
\end{proof}

\begin{theorem}\label{thm:unfoldedpointequivalence}
	Assume $k>0$, $m=0,1$.
	Two parametric families of systems \eqref{eq:unfoldedQ} with $P(x,\epsilon)$, $Q(x,\epsilon)$, and $P'(x',\epsilon)$, $Q'(x',\epsilon)$, are \emph{formally point equivalent}
	if and only if the two meromorphic quadratic differentials
	$\frac{Q(x,\epsilon)}{P(x,\epsilon)^2}(\d x)^2$, $\frac{Q'(x,\epsilon)}{P'(x,\epsilon)^2}(\d x)^2$ are formally equivalent.
\end{theorem}

\begin{proof}
	Assume that a formal point transformation \eqref{eq:parametricpoint} associated to a map $x'=\phi(x,\epsilon)$ provides a point equivalence between the two systems. 
	Then
	\begin{equation}\label{eq:parametriccoordinateQ}
		Q(x,\epsilon)=\psi^2\cdot Q'\circ\phi-\tfrac12\Big[P(x,\epsilon)\!\tdd{x}\Big(\tfrac{P(x,\epsilon)\!\tdd{x}\psi}{\psi}\Big)-\tfrac{1}{2}\Big(\tfrac{P(x,\epsilon)\!\tdd{x}\psi}{\psi}\Big)^2\Big],
	\end{equation}
	where $\psi(x,\epsilon)$ is as in \eqref{eq:parametricpoint}.
	Hence $\phi^*\Big(\frac{Q'(x',\epsilon)}{P'(x',\epsilon)^2}(\d x')^2\Big)=\frac{Q(x,\epsilon)+P(x,\epsilon)R(x,\epsilon)}{P(x,\epsilon)^2}(\d x)^2$, for 
	\begin{equation}\label{eq:R}
		R(x,\epsilon)=\tfrac12\Big[\tdd{x}\Big(\tfrac{P(x,\epsilon)\!\tdd{x}\psi}{\psi}\Big)-\tfrac{1}{2}\Big(\tfrac{\tdd{x}\psi}{\psi}\Big)^2\Big].
	\end{equation}
	When $m=0,1$, then $Q(x,\epsilon)$ has at most one local zero, and by Lemma~\ref{lemma:a1} in the Appendix the quadratic differentials $\frac{Q(x,\epsilon)}{P(x,\epsilon)^2}(\d x)^2$ and $\frac{Q'(x',\epsilon)}{P'(x',\epsilon)^2}(\d x')^2$ are formally equivalent.

	Vice-versa, let $\phi(x,\epsilon)$ be a formal diffeomorphism such that 
	\[\frac{Q(x,\epsilon)}{P(x,\epsilon)^2}(\d x)^2=\phi^*\Big(\frac{Q'(x,\epsilon)}{P'(x,\epsilon)^2}(\d x)^2\Big).\] 
	The associated point transformation pulls back the system $P'\frac{\d y}{\d x}=\begin{psmallmatrix} 0&1\\Q'&0\end{psmallmatrix}y$ to $P\frac{\d y}{\d x}=\begin{psmallmatrix} 0&1\\ Q_0&0\end{psmallmatrix}y$ with
	$Q_0=Q-PR$ and $R$ as in \eqref{eq:R}.
	We shall construct an infinite sequence of point transformations associated with germs
	\[\phi_n(x,\epsilon)=x+P(x,\epsilon)f_n(x,\epsilon),\qquad f_n\in\Cal I^n,\quad n\geq 0,\] 
	whose formal composition will be the desired formal transformation.
	We want $\phi_n$ to transform a system $P\tdd{x}y=\begin{pmatrix}0&1\\Q_n&0\end{pmatrix}y$ with  $\frac{Q_{n}-Q}{P}\in \Cal I^{n}$ to a one 
	$Q_{n+1}$ such that $\frac{Q_{n+1}-Q}{P}\in \Cal I^{n+1}$.
	For $n=0$ we put $f_0(x,\epsilon)=\tilde f_0(x)$ where $\tilde\phi_0=x+x^{k+1}\tilde f_0$ is the formal point transformation between the system for $\epsilon=0$ (Lemma~\ref{lemma:sufficiency}).
	For $n>0$, calculating modulo $\Cal I^{n+1}$, we have $\psi_n=1+x^{k+1}\tdd{x}f_n\mod\Cal I^{n+1}$ and
	\begin{align*}
		\tfrac{Q_n-Q_{n+1}}{P}&=\tfrac{\psi_n^2\cdot Q_{n+1}\circ\phi_n-Q_{n+1}}{P}+R_n \mod\Cal I^{n+1}\\
		&=2Q_{n+1}\tdd{x}f_n+f_n\tdd{x}Q_{n+1}\\
		&\hphantom{=}-\tfrac12x^{2k}\Big((k\!+\!1)(2k\!+\!1)\tdd{x}f_n+3(k\!+\!1)x\big(\tdd{x}\big)^2f_n+x^2\big(\tdd{x}\big)^3f_n\Big)\!\!\! \mod\Cal I^{n+1}
	\end{align*}
	If $m=0$ or $m=1$, the above  equation has a formal solution $f_n\in\Cal I_n$ for any prescribed $P$, $Q_n$, $Q_{n+1}$ with $\frac{Q_{n+1}-Q_n}{P}\in\Cal I^n$.  
\end{proof}

\begin{corollary}
	Assume $k>0$, $m=0,1$.
	Two parametric families of systems \eqref{eq:unfoldedQ} with $P(x,\epsilon)$, $Q(x,\epsilon)$, and $P'(x',\epsilon)$, $Q'(x',\epsilon)$,  are 	
	\emph{formally gauge--coordinate equivalent}
	if and only if they are \emph{formally point equivalent}, that is if 
	the two meromorphic quadratic differentials
	$\frac{Q(x,\epsilon)}{P(x,\epsilon)^2}(\d x)^2$, $\frac{Q'(x,\epsilon)}{P'(x,\epsilon)^2}(\d x)^2$ are formally equivalent.
\end{corollary}

\begin{proof}
	Any gauge--coordinate transformation can be written as a composition of a point transformation and a gauge transformation. A gauge transformation preserves the divisor $P(x,\epsilon)$ and the Weierstrass remainder $Q(x,\epsilon)\mod P(x,\epsilon)$, which by Lemma~\ref{lemma:a1} implies that it also preserves the equivalence class of  $\frac{Q(x,\epsilon)}{P(x,\epsilon)^2}(\d x)^2$. The rest is Theorem~\ref{thm:unfoldedpointequivalence}.	
\end{proof}

\subsection{Normal form of an unfolded meromorphic quadratic differential}

\begin{theorem}[Normal form of an unfolded meromorphic quadratic differential]\label{prop:unfoldequadraticdifferentials} 
	Assume that a parametric family of meromorphic quadratic differentials $\Delta(x,\epsilon)=\frac{Q(x,\epsilon)}{P(x,\epsilon)^2}(\d x)^2$, with $P(x,0)=x^{k+1}$, has at most single saddle point of rank $m+1$, $0\leq m<2k$.
	Then it is analytically equivalent by a coordinate change $x\mapsto\phi(x,\epsilon)$ to a \emph{normal form}
		\begin{equation*}\label{eq:qdnormal}
		\Delta_\nf(x,\epsilon)=\frac{Q_\nf(x,\epsilon)}{P_\nf(x,\epsilon)^2}(\d x)^2=\big(x-s(\epsilon)\big)^m\left(\frac{1+\sqrt{\mu(\epsilon)} x^{k-\frac{m}{2}}}{x^{k+1}+p_{k-1}(\epsilon)x^{k-1}+\ldots+p_0(\epsilon)}\d x\right)^2,
	\end{equation*}
	where $\mu(\epsilon)=0$ for $m$ odd.	
	This form is unique up to the action of $\Z_{2k-m}$ by rotations $x\mapsto \e^{\frac{2j\pi\i }{2k-m}}x$, $j\in\Z_{2k-m}$.
\end{theorem}

The case $m=0$ is equivalent to the normal form of a deformation (unfolding) $\frac{P(x,\epsilon)}{Q(x,\epsilon)^{\frac12}}\frac{\partial}{\partial x}$ of an analytic vector field  $\frac{x^{k+1}}{Q(x,0)^{\frac12}}\frac{\partial}{\partial x}$ of Kostov \cite{Kostov, Klimes-Rousseau}.
The proof for general $m=1$ the same lines as in \cite{Klimes-Rousseau}.

First, using the Weierstrass division theorem write $Q(x,\epsilon)=\tilde Q(x,\epsilon)+P(x,\epsilon)R(x,\epsilon)$ for some $\tilde Q(x,\epsilon)=\tilde q_k(\epsilon)x^k+\ldots+\tilde q_0(\epsilon)$ and $R(x,\epsilon)$, and by Lemma~\ref{lemma:a1} bring the differential to the form $\frac{\tilde Q}{P^2}(\d x)^2$.
If $m=0$, the polynomial $\tilde Q(x,\epsilon)$ has a unique zero $x=s(\epsilon)$, $s(0)=0$, which we can bring to the origin by the translation $x\mapsto x-s(\epsilon)$. 
Up to a scaling we can easily assume that $\tilde q_1(\epsilon)=1$, i.e. that $\tilde Q=x(1+w(x,\epsilon))$.
The rational quadratic differential  $\frac{x^m W(x,\epsilon)^2}{P(x,\epsilon)^2}(\d x)^2$ is now defined on  $x\in\CP^1$. Let $\sqrt{\mu(\epsilon)}:=w_{k-\frac{m}{2}}(\epsilon)$, then
\[\mu(\epsilon)=\res_\infty^2 \frac{x^m W(x,\epsilon)^2}{P(x,\epsilon)^2}(\d x)^2,\]
and write $W(x,\epsilon)=1+\sqrt{\mu(\epsilon)}x^{k-\frac{m}{2}}+w(x,\epsilon)$, where $w(x,\epsilon)=\sum_{\substack{1\leq j\leq k\\ j\neq k-\frac{m}{2}}} w_j(\epsilon)x^j$. 
We will now deform the divisor $P(x,\epsilon)$ in a way to make the part $w(x,\epsilon)$ of $\tilde Q(x,\epsilon)$ disappear. 

\begin{lemma}\label{lemma:a1}
	Assume $Q(x,\epsilon)=(x-s(\epsilon))^mU(x,\epsilon)^2$, $U(0,0)\neq 0$, $0\leq m$. Two quadratic differentials of the form $\Delta_0(x,\epsilon)=\frac{(x-s(\epsilon))^mU(x,\epsilon)^2}{P(x,\epsilon)^2}(\d x)^2$ and $\Delta_1(x,\epsilon)=\frac{(x-s(\epsilon))^m\big(U(x,\epsilon)+P(x,\epsilon)R(x,\epsilon)\big)^2}{P(x,\epsilon)^2}(\d x)^2$, for some analytic germ $R(x,\epsilon)$, are analytically equivalent.
\end{lemma}

\begin{proof}
	Let us embed the two differentials into a family  \[\Delta_t(x,\epsilon)=\frac{(x-s(\epsilon))^m\big(U(x,\epsilon)+tP(x,\epsilon)R(x,\epsilon)\big)^2}{P(x,\epsilon)^2}(\d x)^2\]
	 and denote $\Delta_t^{-\frac12}$
	the dual family of vector fields. We want to find an analytic vector field
	\[Y=\tdd{t}+\frac{P(x,\epsilon)}{U(x,\epsilon)+tP(x,\epsilon)R(x,\epsilon)}\alpha(x,\epsilon,t)\tdd{x},\]
	that commutes with $\Delta_t^{-\frac12}$ (when considered as vector fields in $(x,t)$).
	Then the flow $\exp(sY):(x,t)\mapsto(\Phi_s(x,t),t+s)$ of $Y$ preserves the family: $\Delta_t^{-\frac12}=\Phi_s^*\Delta_{t+s}^{-\frac12}$, and hence also $\Delta_t=\Phi_s^*\Delta_{t+s}$. In particular $\phi(x):=\Phi_1(x,0)=x\circ\exp(Y)\big|_{t=0}$, is such that $\phi^*\Delta_1=\Delta_0$.
	
	The commutation relation  $[Y,\Delta_t^{-\frac12}]=0$  is equivalent to $\alpha(x,\epsilon,t)$ being a solution to the non-homogeneus linear ODE
	\begin{equation}\label{eq:alpha1}
		(x-s(\epsilon))\tdd{x}\alpha+\tfrac{m}{2}\alpha\tdd{x}+R=0,
	\end{equation}
	with an analytic solution $\alpha(x,\epsilon,t)=-(x-s(\epsilon))^{-m/2}\int_{s(\epsilon)}^x (z-s(\epsilon))^{m/2}R(z,\epsilon)\d z$.
\end{proof}

\begin{lemma}
	Consider two families of meromorphic quadratic differentials $\Delta_0$ and $\Delta_1$ of the form 
	\begin{equation} 
	\Delta_t=x^m\left(\frac{1+\sqrt{\mu(\epsilon)}x^{k-\frac{m}{2}}+tw(x)}{x^{k+1}+p(x)}\d x\right)^2,
	\label{pol_form2}
	\end{equation}
	\[p(x)=p_kx^k+\ldots+p_0,\quad w(x)=w_kx^k+\ldots+w_0,\]
	where $w_{k-\frac{m}{2}}=0$ and $\mu(\epsilon)=0$ if $m$ is odd, depending on parameters $\bm p=(p_0,\ldots,p_k)$ and $\bm w=(w_1,\ldots,\widehat{w_{k-\frac{m}{2}}},\ldots,w_k)$.
	Then there exists an analytic transformation $(x,\bm p)\mapsto\big(\phi(x,\bm p,\bm w),\psi(\bm p,\bm w)\big)$ tangent at identity at $(x,\bm p)=0$ that  conjugates
	$\Delta_1$ to $\Delta_0$.
\end{lemma}

\begin{proof}
	Let $X_t=\frac{x^{k+1}+p(x)}{x^{\frac{m}{2}}(1+\sqrt{\mu(\epsilon)}x^{k-\frac{m}{2}}+tw(x))}\tdd{x}$ be the dual vector field to \eqref{pol_form2}.	
	We want to construct a family of transformations depending analytically on $t\in[0,1]$ between $X_0$ and $X_t$, defined by a flow of a vector field $Y$ of the form
	\[
	Y=\tdd{t}+\sum_{j=0}^{k}\xi_j(t,\bm p,\bm w)\tdd{p_j}+\frac{x}{1+\sqrt{\mu(\epsilon)}x^{k-\frac{m}{2}}+tw(x,\epsilon)}h(x,t,\bm p,\bm w)\tdd{x},
	\]
	for some $\xi_j$ and $h$, such that $[Y,X_t]=0$. This is equivalent to
	\begin{equation} \label{homological_eq}
	(1+tw+\sqrt{\mu}x^{k-\frac{m}{2}})\xi+xh\tdd{x}(x^{k+1}+p)-(x^{k+1}+p)x\tdd{x}h-\tfrac{m+2}2(x^{k+1}+p)h=(x^{k+1}+p)w,
	\end{equation}
	where $\xi(x,t,\bm p,\bm w)=\xi_0(t,\bm p,\bm w)+\ldots+\xi_{k-1}(t,\bm p,\bm w)x^{k-1}$.
	We see that we can choose $h(x,t,\bm p,\bm w)$ as a polynomial of order $k$ in $x$:
	\[h(x,t,\bm p,\bm w)=h_0(t,\bm p,\bm w)+\ldots+h_k(t,\bm p,\bm w)x^k,\qquad\text{with $h_{k-\frac{m}{2}}=0$}.\]
	Write $(x^{k+1}+p(x))w(x)=b_0(\bm p,\bm w)+\ldots+b_{2k+1}(\bm p,\bm w)x^{2k+1}$, and denote 
	$\bm b(\bm p,\bm w)=\big(b_0(\bm p,\bm w),\ldots,b_{2k+1}(\bm p,\bm w)\big)$.
	Then the equation \eqref{homological_eq} takes the form of a non-homogeneous linear system for the coefficients $(\bm \xi,\bm h)=(\xi_0,\ldots,\xi_{k-1},h_0,\ldots\widehat{h_{k-\frac{m}{2}}}\ldots,h_k)$:
	\[
	A(t,\bm p,\bm w)\begin{pmatrix}\bm \xi\\ \bm h\end{pmatrix}=\bm b(\bm p,\bm w).
	\]
	For $\bm p=\bm w=0$ the equation \eqref{homological_eq} is
	\[
	\xi+x^{k+1}\left(k-\tfrac{m}2-x\tdd{x}\right)h=0,
	\]
	hence
	\[
	A(t,0,0)=
	\begin{pmatrix} 1&&&&&&\\ &\ddots&&&&&\\ &&1&&&&\\ &&&\hskip-.5em k-\tfrac{m}2 \hskip-1em&&&\\ &&&&\hskip-.5em k-1-\tfrac{m}2 \hskip-1em&&\\ &&&&&\ddots&\\ &&&&&&\hskip-.5em -\tfrac{m}2 \end{pmatrix},
	\]
	where if $m=2j$ is even then the $2k+1-j$-th row and column corresponding to the diagonal term $j-\frac{m}{2}$ are omitted.
	Hence $A(t,\bm p,\bm w)$ is invertible for $t$ from any compact in $\C$ if $\bm p,\ \bm w$ are small enough.
	Since $\bm b(0,0)=0$, the constructed vector field $Y(x,t,\bm p,\bm w)$ is such that $Y(0,t,0,0)=\tdd{t}$ and its flow is well-defined for all $|t|\leq 1$ as long as  $\bm p,\ \bm w$ are small enough.
\end{proof}

This proves the existence of the normalizing transformation. Let us prove the quasi-unicity of the normal form for $m=1$. The case $m=0$ is similar and can be found in \cite{Klimes-Rousseau}.

\begin{lemma}
	Assume $x\mapsto\phi(x,\epsilon)$ is a formal/analytic coordinate transformation between two quadratic differentials of the form \eqref{eq:qdnormal},
	with $m$ odd then $\phi(x,\epsilon)=\zeta x$ with $\zeta^{2k-m}=1$.
\end{lemma}
 
\begin{proof}

	Let $x'=\phi(x,\epsilon)$ be a transformation between the differentials
	$\frac{x^m(1+\sqrt{\mu}x^{k-\frac{m}{2}})^2}{\left(x^{k+1}+p(x,\epsilon)\right)^2}(\d x)^2$ and $\frac{x'^m(1+\sqrt{\mu}x'^{k-\frac{m}{2}})^2}{\left(x'^{k+1}+p'(x',\epsilon)\right)^2}(\d x')^2$, preserving the parameter $\epsilon$.
	If $\phi(x;0)=\zeta x+\ldots$ for some $\zeta\neq 0$ then necessarily $\zeta^{2k-1}=1$. Up to precomposition with a map $x\mapsto \zeta x$,
	we can assume that $\zeta=1$ and $\phi(x;0)=x+\ldots$ is tangent to identity.
		
	We want to prove that $\phi\equiv \id$. This is done by an infinite descent. 
	Denote $\Cal I$ the ideal of analytic functions of $(x,\epsilon)$ that vanish when $\epsilon=0$.
	Since $\phi$ must fix the origin, we can write  $\phi(x,\epsilon)=x+xf(x,\epsilon)$, and we want to show that  $f\in\Cal I^n$ for all $n\geq 0$. 
	Suppose that $f\in \Cal I^n$.
	Developing the transformation equation 
	\begin{equation*}
	\tfrac{x^{k+1}+p}{x^{\frac{m}{2}}(1+\sqrt{\mu}x^{k-\frac{m}{2}})}\tdd{x}\phi=\tfrac{\phi^{k+1}+p'\circ\phi}{\phi^{\frac{m}{2}}(1+\sqrt{\mu}\phi^{k-\frac{m}{2}})},
	\end{equation*} 
	modulo $\Cal I^{n+1}$ and multiplying by $x^{\frac{m}{2}}(1+\sqrt{\mu}x^{k-\frac{m}{2}})$ we get
	\[x^{k+1}+p+ x^{k+1}\tdd{x}(xf)
	=x^{k+1}+p'+x^{k+1}\Big(k+1-\tfrac{m}{2}-\tfrac{(k-\frac{m}{2})\sqrt{\mu}x^{k-\frac{m}{2}}}{1+\sqrt{\mu}x^{k-\frac{m}{2}}}\Big)f  \mod \Cal I^{n+1},\]
	from which
	\[p-p'=x^{k+1}\Big(\tfrac{k-\frac{m}{2}}{1+\sqrt{\mu}x^{k-\frac{m}{2}}}f-x\tdd{x}f\Big) \mod \Cal I^{n+1}.  \]
	The left side is a polynomial of order $\leq k$ in $x$, while the right side develops into a formal power series of vanishing order (valuation) $\geq k+1$, which means that both sides are null modulo $\Cal I^{n+1}$. 
	Therefore on the left side $p=p'\mod \Cal I^{n+1}$, while the right side 
	$\tfrac{k-\frac{m}{2}}{1+\sqrt{\mu}x^{k-\frac{m}{2}}}f-x\tdd{x}f =0\mod \Cal I^{n+1}$ means $f=c(\epsilon)\frac{x^k}{x^{\frac{m}{2}}(1+\sqrt{\mu}x^{k-\frac{m}{2}})}\mod \Cal I^{n+1}$. Since by assumption $m$ is odd, necessarily $c(\epsilon)=0\mod \Cal I^{n+1}$.
\end{proof}

\goodbreak

\end{document}